\documentclass[11pt]{article}
\usepackage{epsfig,mst-stylefile,amssymb,amsmath,harvard}
\usepackage{graphicx}
\usepackage[svgnames]{xcolor}
\newcommand{\bbR}{{\Bbb R}}
\newcommand{\bbN}{{\Bbb N}}
\newcommand{\bbC}{{\Bbb C}}

\newcommand{\norm}[1]{\left\|#1\right\|}
\newcommand{\dom}{\textnormal{dom}}
\newcommand{\ran}{\textnormal{ran}}
\newcommand{\supp }{\textnormal{supp}}

\renewcommand{\cite}{\citeyear}

\begin{document}

\title{Domain and range symmetries of operator fractional Brownian fields
\thanks{The first author was supported in part by the ARO grant W911NF-14-1-0475.
The second author was supported in part by the NSF grants DMS-1462156 and EAR-1344280 and the ARO grant W911NF-15-1-0562. The third author was supported in part by the NSA grant H98230-13-1-0220. The first author would like to thank Karl H.\ Hofmann and Mahir B.\ Can for the enlightening discussions.}
 }

\author{Gustavo Didier \\ Tulane University  \and Mark M.\ Meerschaert \\ Michigan State University \and Vladas Pipiras \\ University of North Carolina}

\bibliographystyle{agsm}

\maketitle

\begin{abstract}
An operator fractional Brownian field (OFBF) is a Gaussian, stationary increment $\bbR^n$-valued random field on $\bbR^m$ that satisfies the operator self-similarity property $\{X(c^E t)\}_{t \in \bbR^m} \stackrel{{\mathcal L}}= \{c^H X(t)\}_{t \in \bbR^m}$, $c > 0$, for two matrix exponents $(E,H)$. In this paper, we characterize the domain and range symmetries of OFBF, respectively, as maximal groups with respect to equivalence classes generated by orbits and, based on a new anisotropic polar-harmonizable representation of OFBF, as intersections of centralizers. We also describe the sets of possible pairs of domain and range symmetry groups in dimensions $(m,1)$ and $(2,2)$.
\end{abstract}


\section{Introduction}

A random vector is called {\it full} if its distribution is not supported on a lower dimensional hyperplane.  A random field $X = \{X(t)\}_{t \in \bbR^m}$ with values in $\bbR^n$ is called {\it proper} if $X(t)$ is full for all $t\neq 0$.  A linear operator $P$ on $\bbR^m$ is called a projection if $P^2=P$.  Any nontrivial projection $P\neq I$ maps $\bbR^m$ onto a lower dimensional subspace.  We say that a random field $X$ is {\it degenerate} if there exists a nontrivial projection $P$ such that $X(t)=X(Pt)$ for all $t\in\bbR^m$.  We say that $X$ is {\it stochastically continuous} if $X(t_n)\to X(t)$ in probability whenever $t_n\to t$.  A proper, nondegenerate, and stochastically continuous random vector field $X$ is called {\it operator self-similar} (o.s.s.) if
\begin{equation}\label{e:o.s.s.}
\{X(c^E t)\}_{t \in \bbR^m} \stackrel{{\mathcal L}}= \{c^{H}X(t)\}_{t \in \bbR^m}\quad \text{for all }c> 0.
\end{equation}
In \eqref{e:o.s.s.}, $\stackrel{{\mathcal L}}=$ indicates equality of finite-dimensional distributions, $E \in M(m,\bbR)$ and $H \in M(n,\bbR)$, where $M(p,\bbR)$ represents the space of real-valued $p \times p$ matrices, and $c^{M} = \exp(M (\log c)) = \sum^{\infty}_{k=0} (M \log c)^k/ k!$ for a square matrix $M$. For a univariate stochastic process (namely, $(m,n) = (1,1)$), the relation \eqref{e:o.s.s.} is called self-similarity (see, for example, Embrechts and Maejima \cite{EmbrechtsMaejima}, Taqqu \cite{taqqu:2003}). 

An {\it operator fractional Brownian field} (OFBF, in short) is an $\bbR^n$-valued random field $X = \{X(t)\}_{t\in \bbR^m}$ satisfying the following three properties: $(i)$ it is Gaussian with mean zero; $(ii)$ it is o.s.s.; $(iii)$ it has stationary increments, that is, for any $h\in\bbR^m$, $\{X(t+h) - X(h)\}_{t\in\bbR^m} \stackrel{{\mathcal L}}= \{X(t) - X(0)\}_{t\in\bbR^m}$. When $(m,n) = (1,1)$, OFBF is the celebrated fractional Brownian motion, widely used in applications due to the long-range dependence property of its increments (see Samorodnitsky and Taqqu \cite{samorodnitsky:taqqu:1994}, Doukhan et al.\ \cite{doukhan:2003}). When $m = 1$, $n \geq 1$, OFBF is known as operator fractional Brownian motion (OFBM).

The theory of o.s.s.\ stochastic processes ($m = 1$, $n \geq 1$) was developed by Laha and Rohatgi \cite{laha:rohatgi:1981} and Hudson and Mason \cite{hudson:mason:1982}, see also Chapter 11 in Meerschaert and Scheffler \cite{meerschaert:scheffler:2001}. OFBM was studied by Didier and Pipiras \cite{didier:pipiras:2011} (see also Amblard et al.\ \cite{amblard:coeurjolly:lavancier:philippe:2012}, and Robinson \cite{robinson:2008}, Kechagias and Pipiras \cite{kechagias:pipiras:2015:def,kechagias:pipiras:2015:ident} on the related subject of multivariate long-range dependent time series). For scalar fields (namely, $m \geq 1$, $n=1$), the analogues of fractional Brownian motion and fractional stable motion were studied in depth by Bierm\'e et al.\ \cite{bierme:meerschaert:scheffler:2007}, with related work and applications found in Benson et al.\ \cite{fractint}, Bonami and Estrade \cite{bonami:estrade:2003}, Lavancier \cite{lavancier:2007}, Bierm\'{e} and Lacaux \cite{bierme:lacaux:2009}, Bierm\'{e} et al.\ \cite{bierme:benhamou:richard:2009}, Guo et al.\ \cite{guo:lim:meerschaert:2009}, Clausel and Vedel \cite{clausel:vedel:2011,clausel:vedel:2013}, Meerschaert et al.\ \cite{WRCR:WRCR20376}, Dogan et al.\ \cite{GRL:GRL52237}, Puplinskait{\.e} and Surgailis \cite{puplinskaite:surgailis:2015}.  Li and Xiao \cite{li:xiao:2011} proved important results on o.s.s.\ random vector fields. Baek et al.\ \cite{baek:didier:pipiras:2014} bridged the gap between harmonizable and moving average integral representations for OFBF.

The \textit{domain and range symmetries} of a proper, nondegenerate random field $X$ starting at zero are defined by
\begin{equation}\begin{split}\label{e:def_G1}
G^{\textnormal{dom}}_{1}(X) :=& \Big\{A \in GL(m,\bbR):\{X(At)\}\stackrel{{\mathcal L}}= \{X(t)\} \Big\},\\
G^{\textnormal{ran}}_{1}(X) :=& \Big\{ B \in GL(n,\bbR): \{BX(t)\} \stackrel{{\mathcal L}}= \{X(t)\}\Big\},
\end{split}\end{equation}
where $GL(k,\bbR)$ denotes the general linear group of invertible $k \times k$ matrices. Cohen et al.\ \cite{cohen:meerchaert:rosinski:2010} and Didier and Pipiras \cite{didier:pipiras:2012}, respectively, characterized the range symmetries of operator stable L\'{e}vy processes and OFBM.

Symmetry is an important modeling consideration, and a useful guide to model selection (see Liao \cite{liao:1992} on Markov processes and Meerschaert and Veeh \cite{meerschaert:veeh:1995} on measures). In particular, the interest in the study of symmetries is tightly connected to two major themes: (a) anisotropy, i.e., when $G^{\textnormal{dom}}_{1}(X)$ is not the orthogonal group, and its applications in several fields such as bone radiographic imaging and hydrology; and (b) the parametric identification of operator scaling laws, which depends on both $G^{\textnormal{dom}}_{1}(X)$ and $G^{\textnormal{ran}}_{1}(X)$. The latter theme is treated in detail for general o.s.s.\ random fields in the related paper Didier et al.\ \cite{didier:meerschaert:pipiras:2016:exponents}. In regard to the former, note that the term ``anisotropy", like ``nonlinear" or ``non-Gaussian", leaves open the question of what \textit{types }of anisotropy (i.e., domain symmetry groups $G^{\textnormal{dom}}_{1}(X)$) exist. This paper is dedicated to the domain and range symmetry groups \eqref{e:def_G1} themselves, and thus, the characterization of anisotropy, when $X$ is an OFBF.

 We use the harmonizable representation of OFBF to construct separate mathematical characterizations of domain and range symmetry groups. More precisely, we establish in any dimension $m$ or $n$ that domain symmetry groups are maximal groups with respect to orbits (Proposition \ref{p:Gdom1=>maximality}), and break up range symmetry groups into a set of commutativity relations (Proposition \ref{p:Gran_structure}) based on a new anisotropic polar-harmonizable representation of OFBF. The latter representation (Proposition \ref{p:akin_to_Prop2.3_Bierme_etal2007} or Remark \ref{r:polar}) is by itself of interest for the analysis of anisotropic fractional covariance structures, and it further provides a mathematical framework for constructing OFBFs displaying all (possible) pairs of domain and range symmetry groups in dimensions $(m,n) = (m,1)$ and $=(2,2)$ (Theorem \ref{t:OFBF_m,2}). In dimension $(2,2)$, this is attained based on absolutely continuous or, in most cases, singular spectral measures, which illustrates the fact that identical symmetry structures can be attained by different covariance structures. In particular, in dimensions $m =2$ and $n = 2$, respectively, all domain and range symmetry groups are explicitly described (Corollary \ref{c:Gdom1,Gran1_dim_m=2,n=2}). As a byproduct of the analysis, it is shown that not all pairs of domain and range symmetry groups, as a Cartesian product, are possible. Notwithstanding the widespread interest in anisotropy (e.g., Xiao \cite{xiao:2009}, Pipiras and Taqqu \cite{pipiras:taqqu:2016}, chapter 9), to the best of our knowledge this paper provides the first characterization of the domain symmetry group -- namely, the types of anisotropy -- of a class of random fields for some $m \geq 2$. A full description of the pairs of symmetry groups in general dimension $(m,n)$ remains an open problem and a topic for future research (see Remarks \ref{r:difficulties_in_extending} and \ref{r:-I_on_mathcalG} on the difficulties involved).

 We provide two applications of our analysis. First, we develop a parametric characterization of the subclass of isotropic OFBFs in any dimension $(m,n)$ (Proposition \ref{p:Gdom1=>maximality}) that sheds light on the fact that isotropy is not determined solely by the domain exponent $E$ in \eqref{e:o.s.s.}. Second, following up on Didier et al.\ \cite{didier:meerschaert:pipiras:2016:exponents}, we revisit the problem of the non-identifiability of OFBF by displaying all the possible sets of exponents in dimension $(m,n) = (2,2)$.

\section{Preliminaries}\label{s:preliminaries}

In this section, we lay out the notation and conceptual framework used in the paper.

$M(n)$ and $M(n,\bbC)$ denote, respectively, the spaces of $n \times n$ matrices with real or complex entries, whereas the space of $n \times m$ matrices with real entries is denoted by $M(n, m, \bbR)$. $A^*$ and $\overline{A}$ stand for the Hermitian transpose and conjugate matrix of $A \in M(n, \bbC)$, respectively. ${\mathcal S}_{\geq 0}(n,\bbC)$, ${\mathcal S}_{> 0}(n,\bbC)$, ${\mathcal S}_{\geq 0}(n,\bbR)$, ${\mathcal S}_{> 0}(n,\bbR)$ represent, respectively, the cones of Hermitian symmetric positive semidefinite, Hermitian symmetric positive definite, symmetric positive semidefinite and symmetric positive definite matrices. $O(n)$, $U(n)$ and $SO(n)$ represent the orthogonal, unitary and special orthogonal groups, respectively. A zero matrix of appropriate dimension is denoted by ${\mathbf 0}$. We will make use of the cyclic and dihedral subgroups of $O(2)$ defined by, respectively,
\begin{equation}\label{e:Dv_Cv}
{\mathcal C}_{\nu} = \{O_{k 2\pi/\nu}: k = 1,\hdots, \nu\}, \quad  {\mathcal D}_{\nu} = \{O_{k 2\pi/\nu},F_{k 2\pi/\nu}: k = 1,\hdots, \nu\}, \quad \nu \in \bbN,
\end{equation}
as well as the dihedral group ${\mathcal D}^*_1 = \{I,\textnormal{diag}(-1,1)\}$. In \eqref{e:Dv_Cv}, we denote
\begin{equation}\label{e:Otheta_Ftheta}
SO(2) \ni O_{\theta} = \left(\begin{array}{cc}
\cos \theta & - \sin \theta \\
\sin \theta & \cos \theta
\end{array}\right), \quad
O(2) \backslash SO(2) \ni F_{\theta} = \left(\begin{array}{cc}
\cos \theta & \sin \theta \\
\sin \theta & - \cos \theta
\end{array}\right), \quad \theta \in [0,2 \pi).
\end{equation}
For example, ${\mathcal C}_2 = \{I,-I\}$, ${\mathcal D}_2 = \{I,-I, \textnormal{diag}(1,-1),\textnormal{diag}(-1,1)\}$. We also define the unitary matrix
\begin{equation}\label{e:U2}
U_2 = \frac{\sqrt{2}}{2}\left(\begin{array}{cc}
1 & 1 \\
i & -i
\end{array}\right),
\end{equation}
which allows us to write a (spectral) decomposition of all matrices $SO(2) \ni O_{\theta} = U_2 \textnormal{diag}(e^{-i \theta}, e^{i \theta})U^*_2$, $\theta \in [0,2\pi)$. The matrix $O_{\theta}$ acts by rotating a vector in $\bbR^2$ by an angle $\theta$, whereas the matrix $F_\theta$ reflects it at the angle $\theta/2$.


By Didier et al.\ \cite{didier:meerschaert:pipiras:2016:exponents}, Proposition 2.1, if $X$ is proper, nondegenerate and $X(0) = 0$ a.s., then $G^{\textnormal{dom}}_{1}(X)$ and $G^{\textnormal{ran}}_{1}(X)$ (see \eqref{e:def_G1}) are compact groups. In particular, recall that a compact subgroup ${\mathcal G}$ of $GL(m,\bbR)$ can be written as
\begin{equation}\label{e:G=WOW^(-1)}
{\mathcal G} = W {\mathcal O}W^{-1},
\end{equation}
or equivalently,
\begin{equation}\label{e:G=O}
{\mathcal G} \cong {\mathcal O},
\end{equation}
where ${\mathcal O}$ is a subgroup of $O(m)$ and the matrix $W \in {\mathcal S}_{> 0}(n,\bbR)$ is called a conjugacy (see Hudson and Mason \cite{hudson:mason:1982}, p.\ 285). Denote the classes of all possible domain groups, range groups, and pairs of domain and range groups, respectively, by
\begin{equation}
\begin{split}
{\Bbb G}^{\textnormal{dom}}_m &= \hspace{1mm} \{G^{\textnormal{dom}}_{1}(X): \textnormal{for some $n \in \bbN$, $X = \{X(t)\}_{t \in \bbR^m}$ is an $\bbR^n$-valued OFBF}\},\\
{\Bbb G}^{\textnormal{ran}}_n &= \hspace{1mm} \{G^{\textnormal{ran}}_{1}(X): \textnormal{for some $m \in \bbN$, $X = \{X(t)\}_{t \in \bbR^m}$ is an $\bbR^n$-valued OFBF}\},\\
{\Bbb G}_{m,n} \hspace{2mm} &= \hspace{1mm} \{(G^{\textnormal{dom}}_{1}(X),G^{\textnormal{ran}}_{1}(X)): \textnormal{$X = \{X(t)\}_{t \in \bbR^m}$ is an $\bbR^n$-valued OFBF}\}.
\end{split}
\end{equation}
Note that ${\Bbb G}^{\textnormal{dom}}_m \times {\Bbb G}^{\textnormal{ran}}_n \supseteq {\Bbb G}_{m,n}$, but the converse, in principle, may not hold (indeed, see Theorem \ref{t:OFBF_m,2} below). For notational simplicity, we will drop the subscripts and write
\begin{equation}\label{e:def_groups}
{\Bbb G}^{\textnormal{dom}}, \hspace{2mm}{\Bbb G}^{\textnormal{ran}} , \hspace{2mm}{\Bbb G}.
\end{equation}

Let $X = \{X(t)\}_{t \in \bbR^m}$ be an OFBF satisfying
\begin{equation}\label{e:minReeig(H)=<maxReeig(H)<minReeig(E*)}
0 < \min \Re \hspace{0.5mm} \textnormal{eig}(H) \leq \max \Re \hspace{0.5mm}\textnormal{eig}(H) < \min \Re \hspace{0.5mm}\textnormal{eig}(E^*),
\end{equation}
where
\begin{equation}\label{e:eig(M)}
\textnormal{eig}(M)
\end{equation}
denotes the set of eigenvalues of a matrix $M$. Throughout the paper, we will assume that the matrices $E$ and $H$ satisfy condition \eqref{e:minReeig(H)=<maxReeig(H)<minReeig(E*)}. Moreover, since the relation \eqref{e:o.s.s.} can be rewritten with $E/\min \Re \hspace{0.5mm}\textnormal{eig}(E^*)$ and $H/\min \Re \hspace{0.5mm}\textnormal{eig}(E^*)$ in place of $E$ and $H$, respectively, we will further assume without loss of generality that the normalization
\begin{equation}\label{e:minReE*=1}
\min \Re \hspace{0.5mm}\textnormal{eig}(E^*) = 1
\end{equation}
holds. Under \eqref{e:minReeig(H)=<maxReeig(H)<minReeig(E*)}, in Baek et al.\ \cite{baek:didier:pipiras:2014}, Theorem 3.1, it is shown that the OFBF $X$ admits a harmonizable representation of the form
\begin{equation}\label{e:OFBF_harmonizable_representation}
\{X(t)\}_{t \in \bbR^m} \stackrel{{\mathcal L}}= \Big\{ \int_{\bbR^m} (e^{i \langle t, x \rangle} - 1)\widetilde{B}_F(dx)\Big\}_{t \in \bbR^m}.
\end{equation}
The term $\widetilde{B}_F(dx)$ is a Hermitian Gaussian random measure whose ${\mathcal S}_{\geq 0}(n,\bbC)$-valued control measure
\begin{equation}\label{e:F_X(dx)=EB(dx)B(dx)}
F_X(dx) ={\Bbb E}\widetilde{B}_F(dx)\widetilde{B}_F(dx)^*
\end{equation}
satisfies the integrability condition
\begin{equation}\label{e:int_x^2/(1+x^2)Fx(dx)}
\int_{\bbR^m} \frac{\| x\|^2}{1 + \| x\|^2} F_X(dx) < \infty,
\end{equation}
where $\| \cdot \|$ is the Euclidean norm. Moreover, the spectral measure $F_X(dx)$ is $(E^*,-2H)$-homogeneous, i.e.,
\begin{equation}\label{e:FX(dx)_is_homogeneous}
F_X(c^{E^*}dx) =c^{-H}F_X(dx)c^{-H^*}, \quad c > 0,
\end{equation}
and
\begin{equation}\label{e:operator_scaling_under_a.c.}
{\mathcal S}_{\geq 0} (n,\bbC) \ni f_X(c^{E^*} x) = c^{-H_E}f_X(x)c^{-H^*_E} \hspace{2mm}dx\textnormal{-a.e.}, \quad c > 0,
\end{equation}
whenever a spectral density $f_X(x) = \frac{F_X(dx)}{dx}$ exists, where $\textnormal{tr}(\cdot)$ denotes the trace. In \eqref{e:operator_scaling_under_a.c.}, we define
\begin{equation}\label{e:H_E}
H_E = H + \frac{\textnormal{tr}(E)}{2}I.
\end{equation}
The existence of a spectral density is guaranteed in the particular case of an operator fractional Brownian motion (OFBM), i.e., an OFBF in dimension $(1,n)$ with $E = 1$, for which
\begin{equation}\label{e:fX(x)_OFBM}
f_X(x) = x^{-(H + I/2)}_+ AA^* x^{-(H^* + I/2)}_+ + x^{-(H + I/2)}_- \overline{AA^*} x^{-(H^* + I/2)}_- \quad dx\textnormal{--a.e.}
\end{equation}
for some $A \in M(n,\bbC)$ (see Didier and Pipiras \cite{didier:pipiras:2011}, Theorem 3.1). In particular, the OFBM class is parametrized by the triplet of real matrices
\begin{equation}\label{e:(H,ReAA*,ImAA*)}
(H, \Re (AA^*), \Im (AA^*)),
\end{equation}
which we call scaling ($H$) and spherical parameters ($AA^*$, or $\Re (AA^*)$ and $\Im (AA^*)$). It can be shown (Didier and Pipiras \cite{didier:pipiras:2011}, Theorem 3.1) that an OFBM is time-reversible, namely, $\{B_{H}(-t)\}_{t \in \bbR} \stackrel{{\mathcal L}}=  \{B_H(t)\}_{t \in \bbR}$, if and only if
\begin{equation}\label{e:time-reversible}
\Im (AA^*)= {\mathbf 0}.
\end{equation}

\begin{remark}\label{r:OFBM_helps_with_OFBF}
Beyond being a particular case of OFBF when $m=1$ and $n \geq 1$, OFBM is of direct interest in this paper since, based on the polar-harmonizable representation of OFBF (Proposition \ref{p:akin_to_Prop2.3_Bierme_etal2007} below), arguments for OFBF can often be reduced to an argument for OFBM by fixing the so-named spherical component of OFBF (see also Remark \ref{r:polar} below). OFBM is also used to illustrate some of the results in the paper.
\end{remark}

Let ${\mathcal R}^m$ denote either ${\mathbb R}^m$ or ${\mathbb R}^m \backslash\{0\}$. Also let
\begin{equation}\label{e:mu:B(Rm)->S+(n,C)}
\mu: {\mathcal B}({\mathcal R}^m) \rightarrow {\mathcal S}_{\geq 0}(n,\bbC)
\end{equation}
be an (entry-wise $\bbC$-valued) measure whose measure induced by the maximum eigenvalue is $\sigma$-finite. Given a linear operator $A \in M(m,\bbR)$, we define the measure
\begin{equation}\label{e:muA(B)=Mu(A^(-1)B)}
\mu_A(B)=\mu(A^{-1}B)
\end{equation}
on Borel sets $B$, where $A^{-1}B$ is the preimage of the set $B$. Equivalently, we can write $\mu_A(dx)=\mu(A^{-1}dx)$. In the analysis of symmetry groups \eqref{e:def_G1}, we shall use the symmetry groups of a measure $\mu$ as defined next.
\begin{definition}\label{def:symm_group_FX}
Let $\mu$ be as in \eqref{e:mu:B(Rm)->S+(n,C)}. The domain and range symmetry sets of $\mu$ are, respectively,
\begin{equation}\label{e:domain_symmetry_group_mu(dx)}
{\mathcal S}^{\textnormal{dom}}(\mu) = \{C \in M(m,\bbR): \mu_C(B) = \mu(B), \hspace{1mm}B \in {\mathcal B}({\mathcal R}^m)\},
\end{equation}
\begin{equation}\label{e:range_symmetry_group_mu(dx)}
{\mathcal S}^{\textnormal{ran}}(\mu) = \{C \in M(n,\bbR): C \mu(B)C^* = \mu(B), \hspace{1mm}B \in {\mathcal B}({\mathcal R}^m)\}.
\end{equation}
\end{definition}


We will need the notion of support of a matrix-valued measure, which is laid out next.
\begin{definition}
Consider the measure $\mu$ as in \eqref{e:mu:B(Rm)->S+(n,C)}. We define
$$
\textnormal{supp}\hspace{0.5mm}\{\mu\}  = \{x \in S: \textnormal{for any open set $U$, $x \in U$ $\Rightarrow$ $\mu(U) \neq {\mathbf 0}$}\}.
$$
\end{definition}
We will use changes of variables into polar coordinates, as discussed in Bierm\'{e} et al.\ \cite{bierme:meerschaert:scheffler:2007}, p.\ 314. Let $E$ be as in \eqref{e:minReeig(H)=<maxReeig(H)<minReeig(E*)}. Then, there exists a norm $\| \cdot \|_0$ on $\bbR^m$ for which
\begin{equation}\label{e:Phi_change-of-variables}
\Psi:(0,\infty) \times S_0 \rightarrow \bbR^m \backslash\{0\},
\quad \Psi(r,\theta):=r^{E}\theta,
\end{equation}
is a homeomorphism, where
\begin{equation}\label{e:S0}
S_0 = \{x \in \bbR^m: \norm{x}_{0}=1\}.
\end{equation}
Then, one can uniquely write the polar representation
\begin{equation}\label{e:x=tau(x)E*l(x)}
\bbR^m \backslash\{0\} \ni x = \tau(x)^E l(x),
\end{equation} where the functions $\tau(x)$, $l(x)$ -- which depend on $E$ -- are called the radial and directional parts of $x$, respectively. One such norm $\| \cdot \|_0$ may be calculated explicitly by means of the expression
\begin{equation}\label{e:E0_norm}
\norm{x}_{0} = \int^{1}_{0} \|t^{E}x\|\frac{dt}{t},
\end{equation}
where $\| \cdot \|$ is any norm in $\bbR^m$.
The uniqueness of the representation (\ref{e:x=tau(x)E*l(x)}) yields
\begin{equation}\label{e:tau_E}
\tau(c^E x) = c \tau(x), \quad l(c^E x) = l(x).
\end{equation}
In particular, if $\|\cdot\|_0$ is the Euclidean norm, then $S_0 = S^{m-1}$, where the latter denotes the ordinary Euclidean sphere.

\section{Main results}\label{s:structural_results}

\subsection{On the characterization of ${\Bbb G}^{\textnormal{ran}}$}\label{s:Gran1}

Recall that $G^{\textnormal{ran}}_1(X)$ and ${\Bbb G}^{\textnormal{ran}}$ are defined in \eqref{e:def_G1} and \eqref{e:def_groups}, respectively. In this section, we find necessary and sufficient conditions for a group to be the range symmetry group $G^{\textnormal{ran}}_1(X)$ of some OFBF $X$, and explicitly describe all possible range groups (i.e., ${\Bbb G}^{\textnormal{ran}}$) in dimension 2.

The next lemma relates the range symmetries of an OFBF $X$ to those of the spectral measure $F_X(dx)$ in \eqref{e:F_X(dx)=EB(dx)B(dx)}, or of the spectral density $f_X(x) = \frac{F_X(dx)}{dx}$ when it exists. It is more convenient to work with and characterize the domain symmetries of the spectral measure. To state the lemma, recall that a measure defined on the Borel $\sigma$-algebra of a Hausdorff space is called a Radon measure when it is both locally finite and inner regular (Bauer \cite{bauer:2001}, p.\ 155).
\begin{lemma}\label{l:Gran1_via_specmeas}
Let $X$ be an OFBF with harmonizable representation \eqref{e:OFBF_harmonizable_representation} and spectral measure $F_X(dx)$ as in \eqref{e:F_X(dx)=EB(dx)B(dx)}. Then,
\begin{equation}\label{e:Gran1=S(FX)*}
G^{\textnormal{ran}}_1(X) = {\mathcal S}^{\textnormal{ran}}(F_X),
\end{equation}
where the term on the right-hand side of \eqref{e:Gran1=S(FX)*} is the range symmetry set defined by \eqref{e:range_symmetry_group_mu(dx)}. In particular, ${\mathcal S}^{\textnormal{ran}}(F_X)$ is a compact group. Let $H_E$ be as in \eqref{e:H_E}. If, in addition, $F_X(dx)$ is absolutely continuous (a.c.) with density $f_X(x)$, then $B$ is a range symmetry of $X$ (i.e., $B \in G^{\textnormal{ran}}_1(X)$) if and only if
\begin{equation}\label{e:Gran1_in_terms_of_fX}
Br^{-H_{E}}f_X(\theta)r^{-H^*_{E}}B^* = r^{-H_{E}}f_X(\theta)r^{-H^*_{E}}, \quad \theta \in \Theta^c_B, \quad r > 0
\end{equation}
($\Theta^c_B$ denoting the complement of $\Theta_B$). In \eqref{e:Gran1_in_terms_of_fX}, $\Theta_B$ is some set in ${\mathcal B}(S_0)$ depending on $B$, $S_0$ is the sphere \eqref{e:S0} associated with $E^*$, and $\sigma(\Theta_B) = 0$ for some finite Radon measure on ${\mathcal B}(S_0)$ that does not depend on $B$. If, in addition, $f_X $ is continuous on $S_0$, then the condition \eqref{e:Gran1_in_terms_of_fX} is equivalent to
\begin{equation}\label{e:Gran1_in_terms_of_fX_continuous_on_S0}
Bf_X(x)B^* = f_X(x), \quad x \neq 0.
\end{equation}
\end{lemma}
\begin{proof}
Since $X$ is Gaussian,
\begin{equation}\label{e:Gran1<=>commutativity}
B \in {\Bbb G}^{\textnormal{ran}}_1(X) \Leftrightarrow B \hspace{0.5mm}{\Bbb E}[X(s)X(t)^*]B^* = {\Bbb E}[X(s)X(t)^*], \quad s,t \in \bbR^m.
\end{equation}
Then, the claim \eqref{e:Gran1=S(FX)*} is a consequence of the harmonizable representation \eqref{e:OFBF_harmonizable_representation} and the definition \eqref{e:domain_symmetry_group_mu(dx)} (the compactness of ${\mathcal S}^{\textnormal{ran}}(F_X)$ then results from that of $G^{\textnormal{ran}}_1(X)$).

So, fix $B \in GL(n,\bbR)$. Assuming $F_X(dx)$ has a density $f_X(x)$, then again by \eqref{e:OFBF_harmonizable_representation}, $B$ is a range symmetry of $X$ if and only if
\begin{equation}\label{e:Gran1_in_terms_of_fX_dx-a.e.}
Bf_X(x)B^* = f_X(x) \quad dx\textnormal{-a.e.}
\end{equation}
By Lemma 2.1 in Baek et al.\ \cite{baek:didier:pipiras:2014}, we can assume that the density $f_X$ is $(E^*,-2 H_E)$-homogeneous for $H_E$ as in \eqref{e:H_E}, namely, it satisfies the relation
\begin{equation}\label{e:specdens_homogeneous}
f_X(c^{E^*}x) = c^{-H_{E}}f_X(x)c^{-H^*_{E}}, \quad x \in \bbR^{m}\backslash\{0\}, \quad c > 0.
\end{equation}
Let
$$
N_B = \{x = r^{E^*}\theta \in \bbR^{m}\backslash \{0\}: B r^{-H_E}f_X(\theta)r^{-H^*_E}B^* \neq r^{-H_E}f_X(\theta) r^{-H^*_E}\},
$$
i.e., $N_B$ is the set of points (expressed in polar coordinates \eqref{e:Phi_change-of-variables}) where the equality in \eqref{e:Gran1_in_terms_of_fX_dx-a.e.} does not hold. For any fixed $\theta \in S_0$, define the set
$$
R_{\theta} = \{r > 0: B r^{-H_E}f_X(\theta)r^{-H^*_E}B^* \neq r^{-H_E}f_X(\theta) r^{-H^*_E}\}.
$$
Further define
$$
\Theta_{0} = \{\theta \in S_{0}: R_{\theta} \textnormal{ has positive $\bbR$-Lebesgue measure}\}.
$$
By \eqref{e:Gran1_in_terms_of_fX_dx-a.e.}, the $\bbR^m$-Lebesgue measure of $N_B$ is zero. Hence, by using Proposition 2.3 in Bierm\'{e} et al.\ \cite{bierme:meerschaert:scheffler:2007} in the second equality below,
$$
0 = \int_{\bbR^m} 1_{N_B}(x) dx = \int_{S_0}\int^{\infty}_0 1_{\{r^{E^*}\theta \in N_B\}}r^{\textnormal{tr}(E^*)-1}dr \sigma(d \theta)
= \int_{S_0} \int^{\infty}_0  1_{\Theta_0}(\theta) 1_{R_{\theta}}(r) r^{\textnormal{tr}(E^*)-1} dr \sigma(d \theta)
$$
for some finite Radon measure $\sigma(d \theta)$ on ${\mathcal B}(S_0)$. Therefore, $\sigma(\Theta_{0}) = 0$. Now consider $x_* = r^{E^*}_0 \theta_0\neq 0$ such that $\theta_0 := l(x_*) \in \Theta^c_0$. Then, $R_{\theta_*}$ has $\bbR$-Lebesgue measure zero. Therefore, by using \eqref{e:Gran1_in_terms_of_fX_dx-a.e.} and \eqref{e:specdens_homogeneous} we can choose a sequence $\{x_n\}_{n \in \bbN}$, $x_n = r^{E^*}_n \theta_0$, such that, as $n \rightarrow \infty$,
$$
B f_X(x_*)B^* = B r^{-H_E}_0 f_X(\theta_0)r^{-H^*_E}_0 B \leftarrow B r^{-H_E}_n f_X(\theta_0)r^{-H^*_E}_n B =B  f_X(x_n) B
$$
\begin{equation}\label{e:Bf(x)B*=f(x)}
= f_X(x_n)= r^{-H_E}_n f_X(\theta_0)r^{-H^*_E}_n \rightarrow r^{-H_E}_0 f_X(\theta_0)r^{-H^*_E}_0 = f_X(x_*).
\end{equation}
Thus, \eqref{e:Gran1_in_terms_of_fX} holds.

Now assume, in addition, that $f_X $ is continuous on $S_0$. The argument leading to \eqref{e:Bf(x)B*=f(x)} can be extended to establish \eqref{e:Gran1_in_terms_of_fX_continuous_on_S0}. In fact, pick $x_0 = r^{E^*}_0 \theta_0 \in N_B$. Since the $\bbR^m$-Lebesgue measure of $N_B$ is zero, there is a sequence $\{x_n\}_{n \in \bbN} = \{r^{E^*}_n \theta_n\}_{n \in \bbN}  \subseteq N^c_B$ such that $x_n \rightarrow x_0$ as $n \rightarrow \infty$. Thus, $r_n \rightarrow r_0$ and $\theta_n \rightarrow \theta_0$, since the function $\Psi$ in \eqref{e:Phi_change-of-variables} is a homeomorphism. Therefore, by replacing $r^{-H_E}_n f_X(\theta_0)r^{-H^*_E}_n$ with $r^{-H_E}_n f_X(\theta_n)r^{-H^*_E}_n$, and $x_*$ with $x_0$ in \eqref{e:Bf(x)B*=f(x)}, \eqref{e:Gran1_in_terms_of_fX_continuous_on_S0} holds. $\Box$
%
\end{proof}
\begin{example}
In the case of an OFBM $X$, by \eqref{e:fX(x)_OFBM} we can assume that the density is continuous except at zero. Therefore, by \eqref{e:Gran1_in_terms_of_fX_continuous_on_S0}, we have that $B \in G^{\textnormal{ran}}_1(X)$ if and only if, for $x \neq 0$,
$$
B x^{-(H + I/2)}_{+} AA^* x^{-(H^* + I/2)}_{+} B^* = x^{-(H + I/2)}_{+} AA^* x^{-(H^* + I/2)}_{+}
$$
$$
\textnormal{and} \quad B x^{-(H + I/2)}_{-} \overline{AA^*} x^{-(H^* + I/2)}_{-} B^* = x^{-(H + I/2)}_{-} \overline{AA^*} x^{-(H^* + I/2)}_{-},
$$
where the two conditions are equivalent by taking complex conjugates.
\end{example}

\begin{remark}
When $G^{\textnormal{ran}}_1(X) \subseteq O(n)$ (i.e., if we can assume that $W = I$ in \eqref{e:G=WOW^(-1)}), the statement \eqref{e:Gran1<=>commutativity} can be rewritten as a set of commutativity relations, i.e.,
$$
G^{\textnormal{ran}}_1(X) = \bigcap_{s,t \hspace{0.5mm}\in \hspace{0.5mm}\bbR} \{O \in O(n): O \hspace{0.5mm}{\Bbb E}X(s)X(t)^* = {\Bbb E}X(s)X(t)^*O\}
$$
$$
= \bigcap_{B \hspace{0.5mm}\subseteq \hspace{0.5mm}\textnormal{supp}\hspace{0.5mm}F_X} \{O \in O(n): O \hspace{0.5mm}F_X(B) = F_X(B)O\},
$$
where the second equality is a consequence of \eqref{e:Gran1=S(FX)*}. Furthermore, in the absolutely continuous case and assuming $f_X$ is continuous on $S_0$, by \eqref{e:Gran1_in_terms_of_fX_continuous_on_S0} we can write
$$
G^{\textnormal{ran}}_1(X)= \bigcap_{x \neq 0} \{O \in O(n): O \hspace{0.5mm}f_X(x) = f_X(x)O\}.
$$
\end{remark}

Proposition \ref{p:akin_to_Prop2.3_Bierme_etal2007}, to be stated and shown next, establishes a formula for a change of measure into (anisotropic) polar coordinates, which in turn yields a polar-harmonizable representation for OFBF. In this reinterpretation of the covariance function of OFBF, the domain exponent $E$ influences the spherical component of OFBF, whereas $H$ determines the decay of the spectral measure in each spherical direction. Before stating and proving the proposition, it is useful to revisit the case of OFBM, characterized by \eqref{e:fX(x)_OFBM} and \eqref{e:(H,ReAA*,ImAA*)}. Expression \eqref{e:HFX(1,infty)+FX(1,infty)H*=AA*} below, involving the spectral measure and parameters of OFBM, will be used in the proof of the ensuing proposition.
\begin{example}\label{ex:HF_X+F_XH*=AA*}
By the homogeneity relation \eqref{e:FX(dx)_is_homogeneous} with $E = 1$ (see \eqref{e:minReE*=1}), $c = x$ and $dx = [1,\infty)$, the spectral measure of an OFBM $X$ satisfies the relation
$$
F_X[x,\infty) = x^{-H}F_X[1,\infty)x^{-H^*}, \quad x > 0.
$$
Consequently, by \eqref{e:fX(x)_OFBM},
$$
x^{-H}\{HF_X[1,\infty) + F_X[1,\infty)H^*\}x^{-H^*} x^{-1} = - \frac{d}{dx}F_X[x,\infty)= f_X(x) = x^{-H} AA^*  x^{-H^*} x^{-1} \quad dx \textnormal{--a.e.}
$$
for $x > 0$. Thus,
\begin{equation}\label{e:HFX(1,infty)+FX(1,infty)H*=AA*}
HF_X[1,\infty) + F_X[1,\infty)H^* = AA^*.
\end{equation}
In particular, the left-hand side of \eqref{e:HFX(1,infty)+FX(1,infty)H*=AA*} is Hermitian positive semidefinite (n.b.: if $B \in {\mathcal S}_{\geq 0}(n,\bbC)$ and $H$ has eigenvalues with positive real parts, it is not generally true that $HB+BH^* \in {\mathcal S}_{\geq 0}(n,\bbC)$).
\end{example}

In the following proposition, $\overline{{\mathcal S}_{\geq 0}(n,\bbC)}$ denotes the cone of extended Hermitian positive semidefinite matrices, obtained from ${\mathcal S}_{\geq 0}(n,\bbC)$ by including matrix limits with infinite maximum eigenvalues.
\begin{proposition}\label{p:akin_to_Prop2.3_Bierme_etal2007}
Let $F_X(dx):{\mathcal B}(\bbR^m) \rightarrow \overline{{\mathcal S}_{\geq 0}(n,\bbC)}$ be the spectral measure \eqref{e:F_X(dx)=EB(dx)B(dx)} under the assumption that
\begin{equation}\label{e:maxeigLambda(Theta)>0=>minReeigLambda(Theta)>0}
\infty \geq \max \textnormal{eig} \hspace{0.5mm} F_X(B) > 0 \Rightarrow \min  \textnormal{eig} \hspace{0.5mm}F_X(B) > 0, \quad B \in {\mathcal B}(\bbR^m ).
\end{equation}
Then,
\begin{equation}\label{e:change_of_variables_into_polar_coordinates}
F_X(B) = \int^{\infty}_{0}\int_{S_0} 1_{\{r^{E^*}\theta \in  B\}}r^{-H}\Delta(d \theta) r^{-H^*}r^{-1}dr, \quad  B \in {\mathcal B}(\bbR^m),
\end{equation}
for some entry-wise finite, Hermitian Borel measure $\Delta: {\mathcal B}(S_0)\rightarrow {\mathcal S}_{\geq 0}(n,\bbC)$. In particular, the covariance function $\Gamma(t_1,t_2)= {\Bbb E}X(t_1)X(t_2)^*$ of an OFBF $X$ has a harmonizable representation in polar coordinates
\begin{equation}\label{e:int-rep-spectral-F}
\Gamma(t_1,t_2) = \int^{\infty}_{0}\int_{S_0}(e^{i\langle t_1,r^{E^*}\theta \rangle} - 1)(e^{-i\langle t_2,r^{E^*}\theta \rangle} - 1)r^{-H}\Delta(d \theta)r^{-H^*}r^{-1}dr, \quad t_1,t_2 \in \bbR^m.
\end{equation}
Conversely, if the function $\Gamma(t_1,t_2)$ as defined in \eqref{e:int-rep-spectral-F} is such that
\begin{equation}\label{e:Gamma(t,t)_is_pos_def}
\Gamma(t,t) \textnormal{ is a positive definite matrix for all $t \neq 0$},
\end{equation}
then it is the covariance function of an OFBF with exponents $(E,H)$.
\end{proposition}
\begin{proof}
Consider the decomposition \eqref{e:x=tau(x)E*l(x)} induced by $E^*$, namely, with the latter in place of $E$. Let
\begin{equation}\label{e:A(r,Theta)}
A(r,\Theta) = \{x \in \bbR^m: \tau(x) \geq r, \hspace{1mm}l(x) \in \Theta \}, \quad r > 0, \quad \Theta \in {\mathcal B}(S_0),
\end{equation}
where $\Theta \in {\mathcal B}(S_0)$ is a set such that
\begin{equation}\label{e:Theta_in_B(S0)}
\max \textnormal{eig} \hspace{0.5mm} F_X(A(1,\Theta)) > 0.
\end{equation}
Define
$$
G_{X,\Theta}[a,b) = \left\{\begin{array}{cc}
F_X (A(a, \Theta)) - F_X (A(b, \Theta)), & 0 < a < b;\\
\overline{G_{X,\Theta}[-b,-a)}, & a < b < 0.
\end{array}\right.
$$
By \eqref{e:tau_E} and \eqref{e:x=tau(x)E*l(x)} with $E^*$ in place of $E$, $A(cr,\Theta) = c^{E^*}A(r,\Theta)$. Hence, by \eqref{e:FX(dx)_is_homogeneous}, the measure $G_{X,\Theta}$ satisfies $G_{X,\Theta}(c[a,b)) = F_{X}(A(ca,\Theta)) - F_{X}(A(cb,\Theta))= c^{-H}G_{X,\Theta}[a,b)c^{-H^*}$, $c > 0$. Moreover, by \eqref{e:Theta_in_B(S0)} and \eqref{e:maxeigLambda(Theta)>0=>minReeigLambda(Theta)>0},
\begin{equation}\label{e:G_{X,Theta}(B)_has_full_rank}
G_{X,\Theta}(B) \in {\mathcal S}_{> 0}(n,\bbC), \quad B \in \{ [a,b): -\infty < a < b < \infty, \hspace{1mm} 0 \notin [a,b)\}.
\end{equation}
We can extend the $\sigma$-finite measure $G_{X,\Theta}$ to ${\mathcal B}(\bbR)$ so that
\begin{equation}\label{e:GX,Theta_on_Borel}
G_{X,\Theta}(-ds) = \overline{G_{X,\Theta}(ds)}, \quad G_{X,\Theta}(cds) = c^{-H}G_{X,\Theta}(ds)c^{-H^*}, \hspace{1mm} c > 0.
\end{equation}
Let $\lambda^{\max}_{X,\Theta}(ds) = \sup_{v \in S^{n-1}_{\bbC}}v^*G_{X,\Theta}(ds)v$ be the measure induced by the maximum eigenvalue of $G_{X,\Theta}(ds)$. For $\rho \geq 1$, by \eqref{e:FX(dx)_is_homogeneous},
\begin{equation}\label{e:v*G[1,r)v}
v^*G_{X,\Theta}[1,\rho) v = v^* \{F_X(A(1,\Theta)) - \rho^{-H}F_X(A(1,\Theta))\rho^{-H^*}\}v, \quad v \in S^{n-1}_{\bbC}.
\end{equation}
By taking $\rho \rightarrow \infty$ in \eqref{e:v*G[1,r)v}, we conclude that
\begin{equation}\label{e:lambdamaxX,Theta[r,infty)<infty}
0 \leq \lambda^{\max}_{X,\Theta}[r,\infty) < \infty, \quad r \geq 1.
\end{equation}
On the other hand, fix $0 < r < 1$ and rewrite $H = P J_H P^{-1}$, where $J_H$ is the Jordan form of $H$ and $P \in GL(n,\bbC)$. Then, for $v \in S^{n-1}_{\bbC}$ and any small $\delta > 0$, by using the relation $r^{-H} = P r^{-J_H}P^{-1}$,
$$
v^*G_{X,\Theta}[r,1) v = v^* r^{-H}F_{X}(A(1,\Theta))r^{-H^*} v - v^* F_{X}(A(1,\Theta)) v
$$
\begin{equation}\label{e:v^*GX,Theta[r,1)v=<r(-2h)}
\leq (v^*P) r^{-J_H}\{P^{-1}F_{X}(A(1,\Theta))(P^{*})^{-1}\} r^{-J^*_H} P^*v \leq C r^{-2h_{\max}- \delta},
\end{equation}
where $h_{\max} := \max \Re\hspace{0.5mm} \textnormal{eig}(H)$ and the last inequality follows by the explicit form for $r^{J_H}$ (e.g., Didier and Pipiras \cite{didier:pipiras:2011}, Appendix D). Thus, the bound \eqref{e:v^*GX,Theta[r,1)v=<r(-2h)} implies that
\begin{equation}\label{e:lambdamaxX,Theta[r,1)<integral}
0 \leq \lambda^{\max}_{X,\Theta}[r,1) \leq C' \int^{1}_{r}s^{-(2h_{\max}+\delta+1)}ds.
\end{equation}
Recall that, by conditions \eqref{e:minReeig(H)=<maxReeig(H)<minReeig(E*)} and \eqref{e:minReE*=1}, $h_{\max} \leq 1$. Thus, together with the bound $|e^{its}-1|^2 \leq \min\{4, C s^2 \}$ for an appropriate $C > 0$, the expressions \eqref{e:lambdamaxX,Theta[r,infty)<infty} and \eqref{e:lambdamaxX,Theta[r,1)<integral} imply that
$$
v^*\Big\{\int^{\infty}_{0}|e^{its}-1|^2 \hspace{1mm}G_{X,\Theta}(ds)\Big\}v \leq \|v\|^2 \int^{\infty}_0 |e^{its}-1|^2 \lambda^{\max}_{X,\Theta}(ds) < \infty, \quad t \in \bbR, \quad v \in S^{n-1}_{\bbC}.
$$
In particular, by \eqref{e:GX,Theta_on_Borel}, $\Gamma_{\Theta}(t_1,t_2) := \int_{\bbR}(e^{i t_1 s}-1)(e^{-i t_2 s}-1)G_{X,\Theta}(ds)$ is the covariance function of an OFBM with range (Hurst) exponent $H$, where properness stems from \eqref{e:G_{X,Theta}(B)_has_full_rank}. By the proof of Theorem 3.1 in Didier and Pipiras \cite{didier:pipiras:2011}, there is $A_{\Theta} \in M(n,\bbC)$ such that
\begin{equation}\label{e:G_{X,Theta}(ds)=density}
G_{X,\Theta}(ds) = (s^{-H}_{+}A_{\Theta}A^*_{\Theta}s^{-H^*}_{+} + s^{-H}_{-}\overline{A_{\Theta}A^*_{\Theta}}s^{-H^*}_{-})s^{-1} ds.
\end{equation}
Since $G_{X,\Theta}[1,\infty) = F_X ( A(1,\Theta))$, then by \eqref{e:HFX(1,infty)+FX(1,infty)H*=AA*} we obtain
\begin{equation}\label{e:HFX(A)+FX(A)H*=AA*}
H \hspace{1mm}G_{X,\Theta}[1,\infty) + G_{X,\Theta}[1,\infty)\hspace{1mm} H^* = A_{\Theta} A^*_{\Theta} \in {\mathcal S}_{\geq 0}(n,\bbC).
\end{equation}
Thus, for $r > 0$, it results from \eqref{e:G_{X,Theta}(ds)=density} that
$$
 - \frac{d}{dr}F_X(A(r,\Theta)) = \frac{G_{X,\Theta}(dr)}{dr} = r^{-H} \{H F_X(A(1,\Theta))+ F_X(A(1,\Theta))H^*\}r^{-H^*}r^{-1}
 $$
 $$
 = r^{-H}\int_{\Theta}\Delta(d \theta) r^{-H^*}r^{-1},
$$
where $\Delta(d \theta) := H F_X(A(1,d\theta)) + F_X(A(1,d\theta))H^*$. Note that, by \eqref{e:HFX(A)+FX(A)H*=AA*}, $\Delta(\Theta) \in {\mathcal S}_{\geq 0}(n,\bbC)$ and $v^* \Delta(\Theta) v < \infty$, $\Theta \in {\mathcal B}(S_0)$, $v \in S^{n-1}_{\bbC}$. By integrating from $r$ to $\infty$, we arrive at the relation \eqref{e:change_of_variables_into_polar_coordinates} for the class ${\mathcal A}$ of sets of the form \eqref{e:A(r,Theta)}. Because $F_X(A(r,S_0)) = r^{-H}F_X(A(1,S_0))r^{-H^*}$, $r > 0$, the measure induced by the maximum eigenvalue of $F_X(dx)$ is $\sigma$-finite. Since, in addition, the class ${\mathcal A}$ is a $\pi$-system that generates the Borel sets, an entry-wise (real and imaginary parts) application of Theorem 1.1.3 in Meerschaert and Scheffler \cite{meerschaert:scheffler:2001} implies \eqref{e:change_of_variables_into_polar_coordinates}.

The polar-harmonizable representation \eqref{e:int-rep-spectral-F} is an immediate consequence of \eqref{e:change_of_variables_into_polar_coordinates}. Conversely, let $\Gamma(t_1,t_2)$ be the function defined by the expression \eqref{e:int-rep-spectral-F}. Then, it corresponds to the covariance function of an OFBF as a consequence of the change of variables formula \eqref{e:change_of_variables_into_polar_coordinates} and the condition \eqref{e:Gamma(t,t)_is_pos_def}, the latter ensuring properness. $\Box$
\end{proof}
\begin{example}
The covariance function of an OFBM (i.e., $m=1$), characterized by its associated spectral density \eqref{e:fX(x)_OFBM}, satisfies \eqref{e:int-rep-spectral-F} with $E = 1$ and
\begin{equation}\label{e:OFBM_polar_harmonizable}
\Delta(d \theta) = AA^*\delta_{\{1\}}(d\theta) + \overline{AA^*}\delta_{\{-1\}}(d\theta),
\end{equation}
where $\delta_{\bullet}$ is a Dirac delta measure.
\end{example}

\begin{remark}\label{r:polar}
As anticipated in Remark \ref{r:OFBM_helps_with_OFBF}, in light of \eqref{e:int-rep-spectral-F} and \eqref{e:OFBM_polar_harmonizable} we can interpret the spectral (covariance) structure of OFBF as that of an OFBM for each fixed spherical direction.

Moreover, let $X = \{X(t)\}_{t \in \bbR^m}$ be an OFBF whose spectral measure $F_X(dx)$ satisfies the assumptions of Proposition \ref{p:akin_to_Prop2.3_Bierme_etal2007}. Then, the representation \eqref{e:int-rep-spectral-F} leads to the polar-harmonizable integral representation
\begin{equation}\label{e:harmonizable-polar}
\{X(t)\}_{t \in \bbR^m} \stackrel{{\mathcal L}}= \Big\{ \int^{\infty}_0 \int_{S_0}(e^{i \langle t, r^{E^*}\theta\rangle}-1)r^{- H - I/2}\widetilde{B}_{H,\Delta}(dr,d\theta)\Big\}_{t \in \bbR^m},
\end{equation}
where $\widetilde{B}_{H,\Delta}(dr,d\theta)$ is a Hermitian Gaussian random measure with ${\mathcal S}_{\geq 0}(n,\bbC)$-valued control measure
$$
{\Bbb E}\widetilde{B}_{H,\Delta}(dr,d\theta)\widetilde{B}_{H,\Delta}(dr,d\theta)^* = dr \Delta(d\theta) = \Delta(d\theta) dr.
$$
The representation \eqref{e:harmonizable-polar} is of independent interest. As noted before Example \ref{ex:HF_X+F_XH*=AA*}, in \eqref{e:harmonizable-polar} the domain exponent $E$ influences the spherical component of OFBF, whereas $H$ determines the decay of the spectral measure in each spherical direction. The representation \eqref{e:harmonizable-polar} generalizes that of OFBM (i.e., $m = 1$) and will allow us to apply the arguments developed for OFBM to OFBF, e.g., as in the proof of Proposition \ref{p:Gran_structure} below.
\end{remark}

\begin{remark}\label{r:suff_condition_Delta_is_Hermitian_posdef}
In terms of the spectral measure expressed in polar coordinates, a sufficient condition for \eqref{e:Gamma(t,t)_is_pos_def} is that, for some basis $\{u_k\}_{k=1,\hdots,m} \subseteq S_0$ of $\bbR^m$, and pairwise disjoint vicinities $\Theta_k \ni u_k$, $k = 1,\hdots,m$,
\begin{equation}\label{e:Lambda(Theta0)_is_proper}
\Re \Delta(\Theta_k) \in {\mathcal S}_{>0}(n,\bbR), \quad k = 1,\hdots,m.
\end{equation}
Indeed, let $v \in \bbR^n \backslash\{0\}$. For $t \neq 0$,
$$
v^* \Gamma(t,t) v  \geq \int^{\infty}_{0}v^* r^{H} \Big\{2\int_{\cup^{m}_{k=1}\Theta_k} |e^{i\langle t,r^{E^*}\theta \rangle} - 1 |^2 \hspace{0.5mm}\Re\Delta(d \theta)\Big\}r^{H^*}v \hspace{1mm}r^{-1}dr
$$
$$
= 2\int^{\infty}_{0}\Big\{\int_{\cup^{m}_{k=1}\Theta_k} |e^{i\langle t,r^{E^*}\theta \rangle} - 1 |^2  \frac{v^* r^{H}}{\|v^* r^{H}\|}\hspace{0.5mm}\Re \Delta(d \theta)\frac{r^{H^*}v}{\|r^{H^*}v\|} \Big\}\| r^{H^*}v\|^2 \hspace{1mm}r^{-1}dr
$$
\begin{equation}\label{e:suff_condition_Delta_is_Hermitian_posdef}
\geq 2\int^{\infty}_{0}\Big\{\int_{\cup^{m}_{k=1}\Theta_k} |e^{i\langle t,r^{E^*}\theta \rangle} - 1 |^2  \min \textnormal{eig}(\Re \Delta(d \theta))\Big\}\| r^{H^*}v\|^2 \hspace{1mm}r^{-1}dr >0.
\end{equation}
The first inequality in \eqref{e:suff_condition_Delta_is_Hermitian_posdef} results from the Hermitian property of the measure $\Delta(d \theta)$, i.e., $\Delta(-d \theta) = \overline{\Delta(d \theta)}$. In turn, the last inequality is a consequence of the condition \eqref{e:Lambda(Theta0)_is_proper}, since $\int_{\cup^{m}_{k=1}\Theta_k} |e^{i\langle t,r^{E^*}\theta \rangle} - 1 |^2  \min \textnormal{eig}(\Re \Delta(d \theta)) > 0$ for all $r> 0$.
\end{remark}
Given a matrix $M$, we denote by ${\mathcal C}_{O(n)}(M)$ the centralizer of $M$ restricted to the orthogonal group, namely, the set of orthogonal matrices that commute with $M$. Centralizers appear naturally in the characterization of the range symmetry group of OFBM, as shown in Didier and Pipiras \cite{didier:pipiras:2012}. We shall use centralizers also with OFBF and resort to arguments in the latter reference whenever needed. The next proposition shows that the range symmetry group of an OFBF $X$ can be decomposed into the intersection of the (orthogonal) centralizers of the spectral measure $F_X(dx)$ expressed in polar coordinates.
\begin{proposition}\label{p:Gran_structure}
Let $X = \{X(t)\}_{t \in \bbR^m}$ be an OFBF satisfying the conditions \eqref{e:minReeig(H)=<maxReeig(H)<minReeig(E*)} and \eqref{e:maxeigLambda(Theta)>0=>minReeigLambda(Theta)>0}. Let
\begin{equation}\label{e:class_U}
{\mathcal U} = \{\Theta \in {\mathcal B}(S_0): \Theta \cap \textnormal{int}(\textnormal{supp}\hspace{0.5mm}\Re\Delta(d \theta)) \neq \emptyset\},
\end{equation}
where the measure $\Delta(d \theta)$ appears in \eqref{e:change_of_variables_into_polar_coordinates} and \eqref{e:int-rep-spectral-F}. Then,
\begin{itemize}
\item [$(i)$]
\begin{equation}\label{e:Gran_structure}
G^{\textnormal{ran}}_1(X) = \bigcap_{\Theta \in {\mathcal U}} G_{H,\Theta}
:=\bigcap_{\Theta \in {\mathcal U}}  W_{\Theta } \Big( \bigcap_{r > 0} {\mathcal C}_{O(n)}(\Pi_{r, \Theta}) \cap {\mathcal C}_{O(n)}(\Pi_{I,\Theta})\Big)W^{-1}_{\Theta },
\end{equation}
where, for $\Theta \in {\mathcal U}$,
\begin{equation}\label{e:Pi-rTheta_Pi-I,Theta}
\Pi_{r,\Theta} = r^{- W^{-1}_{\Theta}H W_{\Theta}}r^{- W_{\Theta}H^* W^{-1}_{\Theta} }, \quad r > 0, \quad
\Pi_{I,\Theta} = W^{-1}_{\Theta} \Im (\Delta(\Theta)) W^{-1}_{\Theta},
\end{equation}
and
\begin{equation}\label{e:W_Theta}
W_{\Theta} := \Re(\Delta( \Theta))^{1/2} \in {\mathcal S}_{>0}(n,\bbR), \quad \Theta \in {\mathcal B}(S_0);
\end{equation}
\item [$(ii)$] when $n = 2$,
\begin{equation}\label{e:Gran_structure_n=2}
G_{H,\Theta} = \left\{\begin{array}{cc}
W_{\Theta } \Big( \bigcap_{r > 0} {\mathcal C}_{O(2)}(\Pi_{r, \Theta}) \cap SO(2) \Big)W^{-1}_{\Theta }, & \textnormal{if }\Im \Delta(\Theta) \neq {\mathbf 0} ;\\
W_{\Theta } \Big( \bigcap_{r > 0} {\mathcal C}_{O(2)}(\Pi_{r, \Theta}) \Big)W^{-1}_{\Theta }, & \textnormal{if }\Im \Delta(\Theta) = {\mathbf 0} ,
\end{array}\right.
\end{equation}
and
\begin{equation}\label{e:C(O(2))(Pir,Theta)=D2_or_O(2)}
{\mathcal C}_{O(2)}(\Pi_{r, \Theta}) = {\mathcal D}_2 \textnormal{ or } O(2), \quad \Theta \in {\mathcal U}, \quad r > 0.
\end{equation}
\end{itemize}
\end{proposition}
\begin{proof} By the change of measure $F_X(dx)$ into polar coordinates (Proposition \ref{p:akin_to_Prop2.3_Bierme_etal2007}) and Lemma \ref{l:Gran1_via_specmeas},
$$
G^{\textnormal{ran}}_1(X) = \{C \in GL(n,\bbR): C F_X(dx) C^*= F_X(dx)\}
$$
$$
= \{C \in GL(n,\bbR): C r^{-H}\Delta(d \theta) r^{-H^*} C^*= r^{-H}\Delta(d \theta) r^{-H^*} , \hspace{1mm}r > 0\}
$$
$$
= \{C \in GL(n,\bbR): C r^{-H} \Re\Delta(d \theta) r^{-H^*} C^*= r^{-H}\Re\Delta(d \theta) r^{-H^*} , \hspace{1mm}r > 0\}
$$
$$
\cap \{C \in GL(n,\bbR): C r^{-H} \Im\Delta(d \theta) r^{-H^*} C^*= r^{-H}\Im\Delta(d \theta) r^{-H^*} , \hspace{1mm}r > 0\}
$$
$$
= \bigcap_{\Theta \in {\mathcal U}}\{C \in GL(n,\bbR): C r^{-H} \Re\Delta(\Theta) r^{-H^*} C^*= r^{-H}\Re\Delta(\Theta) r^{-H^*} , \hspace{1mm}r > 0\}
$$
$$
\cap \{C \in GL(n,\bbR): C r^{-H} \Im\Delta(\Theta) r^{-H^*} C^*= r^{-H}\Im\Delta(\Theta) r^{-H^*} , \hspace{1mm}r > 0\}
$$
$$
=: \bigcap_{\Theta \in {\mathcal U}} G^{\textnormal{ran}}_{1,\Theta} \cap G^{\textnormal{ran}}_{2, \Theta}.
$$
So, consider $W_{\Theta}$ as in \eqref{e:W_Theta}. To establish claim $(i)$, we can apply the same argument as in the proof of Theorem 3.1 in Didier and Pipiras \cite{didier:pipiras:2012}, pp.\ 362--364, with $W_{\Theta}$ in place of $W$ and $\Delta(\Theta) \in {\mathcal S}_{\geq 0}(n,\bbC) \cap GL(n,\bbC)$, $\Theta \in {\mathcal U}$ (under condition \eqref{e:maxeigLambda(Theta)>0=>minReeigLambda(Theta)>0}), in place of the OFBM spectral matrix $AA^*$. In particular, by the argument in Didier and Pipiras \cite{didier:pipiras:2012}, one can express $G^{\textnormal{ran}}_{2, \Theta} = \{C \in GL(n,\bbR): C \Im\Delta(\Theta) C^*= \Im\Delta(\Theta)\}$; i.e., the centralizer ${\mathcal C}_{O(n)}(\Pi_{I,\Theta})$ does not depend on $r$. Claim ($ii$) is a consequence of the following two facts. First, if $\Pi_{I,\Theta} \neq {\mathbf 0}$, then ${\mathcal C}_{O(2)}(\Pi_{I,\Theta}) = SO(2)$, since $\Pi_{I,\Theta}$ is skew-symmetric, namely, $\Pi_{I,\Theta}^* = - \Pi_{I,\Theta}$ (see Lemma 5.1 in Didier and Pipiras \cite{didier:pipiras:2012}, p.\ 376, where ${\mathcal C}_{O(2)}(\Pi_{I,\Theta})$ is denoted by $G(\Pi_{I,\Theta})$). Second, \eqref{e:C(O(2))(Pir,Theta)=D2_or_O(2)} follows from the analysis in Didier and Pipiras \cite{didier:pipiras:2012}, p.\ 377. $\Box$\\
\end{proof}

We are now in a position to describe ${\Bbb G}^{\textnormal{ran}}$ in dimension 2.
\begin{corollary}\label{c:range_groups_dim2}
Consider the class of OFBFs taking values in $\bbR^2$ and satisfying the conditions \eqref{e:minReeig(H)=<maxReeig(H)<minReeig(E*)} and \eqref{e:maxeigLambda(Theta)>0=>minReeigLambda(Theta)>0}. Then, up to a positive definite conjugation (see \eqref{e:G=WOW^(-1)}), the elements of ${\Bbb G}^{\textnormal{ran}}$ are
\begin{equation}\label{C2,D2,SO(2),O(2)}
\textnormal{${\mathcal C}_2$, ${\mathcal D}_2$, $SO(2)$ or $O(2)$}.
\end{equation}
\end{corollary}
\begin{proof}
First, note that
\begin{equation}\label{e:O(2)andSO(2),D2andSO(2)}
O(2) \cap SO(2) = SO(2), \quad {\mathcal D}_2 \cap SO(2) = {\mathcal C}_2.
\end{equation}
So, let $G_{H,\Theta}$ be as in \eqref{e:Gran_structure}. By \eqref{e:Gran_structure_n=2}, \eqref{e:C(O(2))(Pir,Theta)=D2_or_O(2)} and \eqref{e:O(2)andSO(2),D2andSO(2)}, $G_{H,\Theta}$ has one of the forms in \eqref{C2,D2,SO(2),O(2)} up to a positive definite conjugation $W_{\Theta}$. Moreover, fix any pair $\Theta, \Theta'  \in {\mathcal B}(S_0)$ and consider their associated groups $G_{H,\Theta}$, $G_{H,\Theta' }$. By looking at each subcase and using Lemma \ref{l:W1=wW2}, the intersection group $G_{H,\Theta} \cap G_{H,\Theta' }$ also has one of the forms \eqref{C2,D2,SO(2),O(2)}, up to a conjugacy, as summed up in Table \ref{table:intersection_groups_dim=2}. In fact, to obtain the first entry of Table \ref{table:intersection_groups_dim=2}, suppose $G_{H,\Theta} = W_{\Theta} O(2) W^{-1}_{\Theta} $ and $G_{H,\Theta'} = W_{\Theta'}  O(2) W^{-1}_{\Theta'} $. Then, ${\mathcal C}_2 \subseteq G_{H,\Theta'} \cap G_{H,\Theta'}$. If, in addition, the latter set inclusion is strict, then Lemma \ref{table:intersection_groups_dim=2} implies that $W_{\Theta} = w W_{\Theta'}$ for some $w > 0$, and thus $G_{H,\Theta'} \cap G_{H,\Theta'}= W_{\Theta} O(2) W^{-1}_{\Theta}$, as stated in Table \ref{table:intersection_groups_dim=2}. For another case, if $G_{H,\Theta} = W_{\Theta} SO(2) W^{-1}_{\Theta} $ and $G_{H,\Theta'} = W_{\Theta'}  {\mathcal D}_2 W^{-1}_{\Theta'} $, then ${\mathcal C}_2 \subseteq G_{H,\Theta'} \cap G_{H,\Theta'}$ but $W_{\Theta'}\{\pm \textnormal{diag}(1,-1)\} W^{-1}_{\Theta'} \nsubseteq G_{H,\Theta'} \cap G_{H,\Theta'}$, since the matrices $\pm \textnormal{diag}(1,-1)$ have negative eigenvalues. Therefore, $G_{H,\Theta'} \cap G_{H,\Theta'} = {\mathcal C}_2$, as described on the sixth entry of Table \ref{table:intersection_groups_dim=2}. The remaining entries of the table can be obtained in a similar fashion.

Since for any subgroup ${\mathcal G}$ given by \eqref{C2,D2,SO(2),O(2)} there is an OFBM $X$ ($m = 1$) such that $G^{\ran}_1(X) \cong {\mathcal G}$ (see Didier and Pipiras \cite{didier:pipiras:2012}, Theorem 5.1), then up to conjugacies the family ${\Bbb G}^{\textnormal{ran}}$ is given by \eqref{C2,D2,SO(2),O(2)}, as claimed. $\Box$
\end{proof}

\begin{remark}
For the sake of illustration, in Examples \ref{ex:Gdom1=O(2)_Gran1=O(2)} and \ref{ex:Gdom1=D3_Gran1=SO(2)} below we construct two OFBFs with given domain and range symmetry groups.
\end{remark}

\vspace{-0.5cm}
\begin{center}
\begin{table}[h]
\centering
\begin{tabular}{ccc}\hline
 $G_{H,\Theta} \cong \hdots$ & $G_{H,\Theta'} \cong \hdots$ & $G_{H,\Theta} \cap G_{H,\Theta'} \cong \hdots$ \\\hline
 $O(2)$ & $O(2)$ & $O(2)$ \textnormal{ or } ${\mathcal C}_2$ \\
 $O(2)$ & $SO(2)$ & $SO(2)$ \textnormal{ or } ${\mathcal C}_2$ \\
 $O(2)$ & ${\mathcal D}_2$ & ${\mathcal D}_2$ \textnormal{ or }${\mathcal C}_2$\\
 $O(2)$ & ${\mathcal C}_2$ & ${\mathcal C}_2$\\
 $SO(2)$ & $SO(2)$ & $SO(2)$\textnormal{ or } ${\mathcal C}_2$ \\
 $SO(2)$ & ${\mathcal D}_2$ &  ${\mathcal C}_2$  \\
 $SO(2)$ & ${\mathcal C}_2$ & ${\mathcal C}_2$\\
 ${\mathcal D}_2$ & ${\mathcal D}_2$ &  ${\mathcal D}_2$ \textnormal{ or } ${\mathcal C}_2$   \\
 ${\mathcal D}_2$ & ${\mathcal C}_2$ & ${\mathcal C}_2$\\
 ${\mathcal C}_2$ & ${\mathcal C}_2$ & ${\mathcal C}_2$\\\hline
\end{tabular}\caption{The intersection of range groups for different spherical sets $\Theta$, $\Theta'$ ($n = 2$), where
$\cong $ denotes conjugacy (see \eqref{e:G=O}).}\label{table:intersection_groups_dim=2}
\end{table}
\end{center}
\vspace{-1cm}

\subsection{On the characterization of ${\Bbb G}^{\textnormal{dom}}$}\label{s:domain_symm_groups_are_maximal}

This section is dedicated to domain symmetry groups, where $G^{\textnormal{dom}}_1(X)$ for an OFBF $X$ is defined in \eqref{e:def_groups}. However, the arguments are quite different from those in Section \ref{s:Gran1}, on range symmetry groups, and build upon the framework developed in Meerschaert and Veeh \cite{meerschaert:veeh:1995}. In particular, the operator self-similarity of OFBF will not play a role, and will only reappear in the subsequent Section \ref{s:on_G}. The two main results of this section are the following. First, we find a necessary condition for a group to be the domain symmetry group of some OFBF $X$. Second, and within the same mathematical framework, we show how a scalar-valued measure can be built that has a given domain symmetry group. Such measure will be used in Section \ref{s:on_G} to help build OFBFs with given domain and range symmetry groups, and thus different types of anisotropy. An explicit description of ${\Bbb G}^{\textnormal{dom}}$ in dimension $m = 2$ is postponed to Corollary \ref{c:Gdom1,Gran1_dim_m=2,n=2} in Section \ref{s:on_G}.

The next lemma relates the domain symmetries of $X$ to those of the spectral measure $F_X(dx)$ in \eqref{e:F_X(dx)=EB(dx)B(dx)}, or of the spectral density $f_X(x)= F_X(dx)/dx$ when it exists. As with range symmetries, it is more convenient to work with the spectral measure in the study of domain symmetries.
\begin{lemma}\label{l:G1_via_specmeas}
Let $X$ be an OFBF with harmonizable representation \eqref{e:OFBF_harmonizable_representation} and spectral measure $F_X(dx)$ as in \eqref{e:F_X(dx)=EB(dx)B(dx)}. Then,
\begin{equation}\label{e:Gdom1=S(FX)*}
G^{\textnormal{dom}}_1(X) = {\mathcal S}^{\textnormal{dom}}(F_X)^*,
\end{equation}
where ${\mathcal S}^{\textnormal{dom}}(F_X)$ is the domain symmetry set defined by \eqref{e:domain_symmetry_group_mu(dx)}. In particular, ${\mathcal S}^{\textnormal{dom}}(F_X)$ is a compact group. If, in addition, $F_X(dx)$ is absolutely continuous, then $A \in G^{\textnormal{dom}}_1(X)$ if and only if
\begin{equation}\label{e:Gdom1_in_terms_of_fX}
f_X(x) = |\det (A^*)^{-1}| f_X((A^*)^{-1}x) \quad dx\textnormal{--a.e.}
\end{equation}
\end{lemma}
\begin{proof}
A matrix $A$ in $GL(m,\bbR)$ satisfies $A \in G^{\textnormal{dom}}_1(X)$ if and only if ${\Bbb E}X(As)X(At)^* = {\Bbb E} X(s)X(t)^*$ for all $s,t\in{\mathbb R}^m$.  From the representation \eqref{e:OFBF_harmonizable_representation},
\begin{equation}
{\Bbb E}X(As)X(At)^* =\int_{\bbR^m} (e^{i \langle s,x \rangle}-1)(e^{-i \langle t,x \rangle}-1)(F_X)_{A^*}(dx),\\
\end{equation}
by a change of variables $x = A^*y$. It follows that $A \in G^{\textnormal{dom}}_1(X)$ if and only if
$(F_X)_{A^*}( dx) = F_X(dx)$, as we wanted to show. Proposition 2.1 in Didier et al.\ \cite{didier:meerschaert:pipiras:2016:exponents} then implies that ${\mathcal S}^{\textnormal{dom}}(F_X)$ is a compact group, and \eqref{e:Gdom1_in_terms_of_fX} follows promptly. $\Box$\\
\end{proof}

To characterize domain symmetry groups, recall the group equivalence relation laid out in Meerschaert and Veeh \cite{meerschaert:veeh:1995}. For two subgroups ${\mathcal G}, {\mathcal K} \subseteq GL(m,\bbR)$, we write that
\begin{equation}\label{e:equiv_class_based_on_orbits}
{\mathcal G} \sim {\mathcal K} \Leftrightarrow \{Gx: G \in {\mathcal G}\} = \{Kx: K \in {\mathcal K}\} \textnormal{ for all } x \in \bbR^m.
\end{equation}
Let $[{\mathcal G}]$ be the equivalence class of the group ${\mathcal G}$. We partial order the subsets of $GL(m,\bbR)$ by set inclusion and call a group \textit{maximal }when it contains all other groups in its equivalence class. For example, $O(2)$ is maximal, whereas $SO(2)$ is not, since $[O(2)] = [SO(2)]$. The next proposition shows that maximality is a necessary condition for a group to be the domain symmetry group of an OFBF.

\begin{proposition}\label{p:Gdom1=>maximality}
Let $X$ be an OFBF with harmonizable representation \eqref{e:OFBF_harmonizable_representation}. Then, the domain symmetry group ${\mathcal G} = G^{\textnormal{dom}}_1(X)$ is maximal with respect to its equivalence class $[{\mathcal G}]$.
\end{proposition}
\begin{proof}
Since the group $G^{\textnormal{dom}}_1(X)$ is compact, then by Lemma \ref{l:G1_via_specmeas} we can write $G^{\textnormal{dom}}_1(X)^* = W {\mathcal O}_0 W^{-1} = {\mathcal S}^{\textnormal{dom}}(F_X)$, where $W$ is a positive definite matrix (see \eqref{e:G=WOW^(-1)}) and ${\mathcal O}_0$ is a subgroup of $O(m)$. Define the measure $G_X(dx) = F_X(Wdx)$. Then, $G_X( {\mathcal O}_0 dx) = F_X(W {\mathcal O}_0 W^{-1}W dx) = F_X(W dx) = G_X(dx)$. Therefore, $ {\mathcal O}_0 \subseteq {\mathcal S}^{\textnormal{dom}}(G_X)$. Now let $A \in {\mathcal S}^{\textnormal{dom}}(G_X)$. Then, $F_X(WAdx) = G_X(Adx) = G_X(dx) = F_X(Wdx)$, whence $F_X(WAW^{-1}dy) = F_X(WW^{-1}dy) = F_X(dy)$. Thus, $A \in {\mathcal O}_0$. In other words, ${\mathcal O}_0 = {\mathcal S}^{\textnormal{dom}}(G_X)$, and thus without loss of generality we can assume that $W = I$ so that we can conveniently use the Euclidean norm in the ensuing argument.

Let $\phi(z)$ be the $M(n,\bbC)$-valued transform
$$
\phi(z) = \int_{\bbR^m} e^{i \langle z,x\rangle} \frac{\|x\|^2}{1 + \|x\|^2} F_X(dx), \quad z \in \bbR^m.
$$
By condition \eqref{e:int_x^2/(1+x^2)Fx(dx)}, $\phi(z)$ is well-defined pointwise. For a pair $k,l=1,\hdots,n$, let $\mu(dx)_{k l} = \frac{\|x\|^2}{1 + \|x\|^2} F_X(dx)_{k l}$ and $\mu(dx) = (\mu(dx)_{k l})_{k,l=1,\hdots,n}$. Now, consider the decomposition $\mu(dx)_{k l} = \Re\mu(dx)_{k l} + i \Im \mu(dx)_{k l}$.
The Borel measures $\Re \mu(dx)_{k l}, \Im \mu(dx)_{k l}$ are real-valued and finite. Hence, they can be broken up into positive and negative parts. So, for simplicity we can suppose that $\Re \mu(dx)_{k l}$, $\Im \mu(dx)_{k l}$ are positive measures. As a consequence,
$$
\phi(z)_{k l} = \alpha_{k l,1}\int_{\bbR^m} e^{i \langle z,x\rangle} \frac{\Re \mu(dx)_{k l}}{\alpha_{k l,1}} + i \alpha_{k l,2} \int_{\bbR^m} e^{i \langle z,x\rangle}  \frac{\Im \mu(dx)_{k l}}{\alpha_{k l,2}} =: \alpha_{k l,1} \phi(z)_{k l,1} + i \alpha_{k l,2} \phi(z)_{k l,2} \in \bbC,
$$
where the constants $\alpha_{k l,1}, \alpha_{k l,2} >0 $ make $\Re \mu(dx)_{k l}/ \alpha_{k l,1}$, $\Im \mu(dx)_{k l}/ \alpha_{k l,2}$ into probability measures. Therefore, $\phi(z)_{k l,j}$, $j=1,2$, are ($\bbC$-valued) characteristic functions. Thus, for $O \in O(m)$,
$$
\phi(O^* z)_{k l} = \phi(z)_{k l} \Leftrightarrow \phi(O^* z)_{k l,j} = \phi(z)_{k l,j}, \hspace{1mm}j=1,2 \Leftrightarrow \mu(O^* dx)_{k l,j} = \mu(dx)_{k l,j}, \hspace{1mm} j=1,2
$$
\begin{equation}\label{e:equivalence_symmetry_phi_mu}
\Leftrightarrow \mu(O^* dx) = \mu(dx) \Leftrightarrow F_X(O^* dx) = F_X(dx)
\end{equation}
for all $z \in \bbR^m $, where the last equivalence follows from $\frac{\|Ox\|^2}{1 + \|Ox\|^2} = \frac{\|x\|^2}{1 + \|x\|^2}$. Therefore, for $O \in O(m)$,
$$
\phi(O^*z) = \phi(z), \hspace{2mm}z \in \bbR^m \Leftrightarrow O \in {\mathcal S}^{\textnormal{dom}}(F_X),
$$
where $\phi(z) = ( \phi(z)_{kl} )_{k,l=1,\hdots,n}$. So, we can focus on the characteristic function matrix $\phi$ (for all entries $k,l$ simultaneously).

Suppose ${\mathcal G} \subseteq {\mathcal K} \in [{\mathcal G}]$ and recall that $W = I$. This implies that ${\mathcal K} \subseteq O(m)$; otherwise, for some non-orthogonal matrix $K \in {\mathcal K}$, $Kx \notin S^{m-1}$ for some $x \in S^{m-1}$. This contradicts the fact that ${\mathcal G}S^{m-1} \subseteq S^{m-1}$. So, for all $K \in {\mathcal K}$ and all $z \in \bbR^m$, $K z = G_{z,K} z$ for some $G_{z,K} \in {\mathcal G}$. But then, for this $K$, we have that for every $z \in \bbR^m $, $\phi(K^* z)  = \phi(G^*_{z,K} z) = \phi(z) $. Hence, \eqref{e:equivalence_symmetry_phi_mu} implies that $K \in {\mathcal S}^{\textnormal{dom}}(F_X) = G^{\textnormal{dom}}_1(X)^* = {\mathcal G}$, by Lemma \ref{l:G1_via_specmeas}, i.e., ${\mathcal K} \subseteq {\mathcal G}$. However, since ${\mathcal G} \subseteq {\mathcal K}$, we conclude that ${\mathcal K} = {\mathcal G}$. Hence ${\mathcal G}$ is maximal, as claimed. $\Box$
\end{proof}

\begin{remark}
Table \ref{table:OFBF_m=2_n=2_symmgroups} in Section \ref{s:on_G} contains an explicit description of maximal compact subgroups of $O(2)$, i.e., $O(2)$, ${\mathcal C}_\nu$, ${\mathcal D}_\nu$ for $\nu \in \bbN$, and ${\mathcal D}^*_1$. Corollary \ref{c:Gdom1,Gran1_dim_m=2,n=2} shows that, indeed, these types of subgroup make up ${\Bbb G}^{\textnormal{dom}}$ in dimension $m = 2$.
\end{remark}

\begin{remark}
As mentioned in Section \ref{s:Gran1}, in Examples \ref{ex:Gdom1=O(2)_Gran1=O(2)} and \ref{ex:Gdom1=D3_Gran1=SO(2)} below we construct two OFBFs with given domain and range symmetry groups.
\end{remark}

The next natural question is whether the converse of Proposition \ref{p:Gdom1=>maximality} is true, namely, given a compact maximal subgroup ${\mathcal G}$ of $GL(m,\bbR)$ one can build an OFBF whose domain symmetry group is ${\mathcal G}$. We answer this question in the negative for $(m,1)$ and in the affirmative for $(m,n) = (2,2)$ in Theorem \ref{t:OFBF_m,2} below. Starting from ${\mathcal G}$, the construction of the OFBF amounts to defining an appropriate spectral measure, expressed in polar coordinates \eqref{e:FX(dx)=x^-H_Xi(dx)_x^-H_x^(-1)}. The first step in this direction consists of defining a scalar-valued measure in the fashion of Meerschaert and Veeh \cite{meerschaert:veeh:1995}, p.\ 3, which draws upon the Haar measure of the compact group ${\mathcal G}$. Recall that, by \eqref{e:G=WOW^(-1)}, every compact subgroup of $GL(m,\bbR)$ is a subgroup of $O(m)$ up to a conjugacy. Therefore, to construct a measure with a given symmetry group, it suffices to directly consider subgroups of $O(m)$, instead.
\begin{definition}\label{def:LambdaD(dx)}
Let ${\mathcal G}$ be a maximal compact subgroup of $O(m)$, and let $D = \{x_1,\hdots,x_J\}$ be a set of points (pivots) in $S^{m-1}$ such that their respective orbits ${\mathcal G}x_1,\hdots, {\mathcal G}x_J$ are distinct, i.e., ${\mathcal G}x_{j_1} \neq {\mathcal G}x_{j_2}$ for $j_1 \neq j_2$. For $j = 1,\hdots, J$, denote by $n_j \in \bbN$ the (finite) number of connected components of the orbit ${\mathcal G}x_j$, where connectedness is defined in the topology induced by any matrix norm. We define the $\bbR^m$-Borel measure
\begin{equation}\label{e:Lambda_D}
\Lambda_{D}(dx) = \sum^{J}_{j=1}\int_{{\mathcal G}} j n_j \hspace{1mm}\delta_{x_j}(G dx) {\mathbf H}(dG) \geq 0,
\end{equation}
where, for $j = 1,\hdots,J$, $\delta_{x_j}$ is the Dirac measure concentrated on the pivot $x_j$, and ${\mathbf H}(dG)$ is the unique Haar probability measure on the group ${\mathcal G}$ (see Halmos \cite{halmos:2000}, pp.\ 254 and 263).
\end{definition}
Whenever there is no risk of ambiguity, we will drop the subscript $D$ and simply write $\Lambda(dx)$. Lemma \ref{l:Lambda,Lambda_j_properties} in the Appendix sums up some of the properties of the measure $\Lambda(dx)$ (see also Meerschaert and Veeh \cite{meerschaert:veeh:1995}, p.\ 3, proof of Theorem 1).

In the next proposition, we show that the symmetry group of the scalar-valued measure $\Lambda_D(dx)$ becomes exactly ${\mathcal G}$ after a sufficiently large, but finite, number of pivots is added to the set $D$. In the statement and proof of the proposition, $\textnormal{span}\{{\mathcal G}x_1,\hdots,{\mathcal G}x_J\}$ for a set of pivots $x_1,\hdots, x_J$ is understood as the space generated by the vectors in the orbits ${\mathcal G}x_1, \hdots,{\mathcal G}x_J$. Before stating and showing the proposition, we give a simple example of a measure $\Lambda_D(dx)$ and briefly discuss its properties. This will be useful when proving the proposition.
\begin{example}\label{ex:LambdaD(dx)}
Consider $\Lambda_D(dx)$ as in Definition \ref{def:LambdaD(dx)} with $J = 1$, $D = \Big\{e^{i \pi/4} \equiv \Big(\frac{\sqrt{2}}{2},\frac{\sqrt{2}}{2}\Big)^* \Big\}\subseteq S^1$ and ${\mathcal G} = {\mathcal C}_4 = \{O_{0}, O_{\pi/2}, O_{\pi}, O_{3 \pi/2}\}$. Then, the orbit associated with the pivot $e^{i \pi/4}$ is ${\mathcal G} e^{i \pi/4} = \{e^{i \pi/4}, e^{i 3\pi/4}, e^{i 5\pi/4}, e^{i 7\pi/4}\}$, where the complex exponentials, interpreted as vectors in $S^{1}$, are the connected components of the orbit. Since the Haar measure ${\mathbf H}(dG)$ of ${\mathcal C}_4$ assigns equal weight 1/4 to each element (connected component) of the group, we further obtain that
$$
\int_{\mathcal G} \hspace{1mm}\delta_{e^{i \pi/4}}(G e^{i \pi/4}){\mathbf H}(dG) = \delta_{e^{i \pi/4}}(I e^{i \pi/4}){\mathbf H}(I) = 1 \times \frac{1}{4} = \frac{1}{\# \{\textnormal{connected components of ${\mathcal G}e^{i \pi/4}$}\}}.
$$
In addition, note that $\int_{\mathcal G} \hspace{1mm}\delta_{e^{i \pi/4}}(G\{y\}){\mathbf H}(dG) = 0$ when the orbit ${\mathcal G}\{y\}$ does not include the vector $e^{i\pi/4} \in S^1$.
\end{example}

\begin{proposition}\label{p:existence_singular_measure}
Let ${\mathcal G}$ be a maximal compact subgroup of $O(m)$ and let $\Lambda_{D}(dx)$ be the measure \eqref{e:Lambda_D} for a given set of pivots $D =\{ x_1,\hdots,x_J \} \subseteq S^{m-1}$. Then, there is a finite set of pivots $D_k \subseteq S^{m-1}$, $D \subseteq D_k$, satisfying the conditions of Definition \ref{def:LambdaD(dx)}, such that the corresponding scalar-valued, $\bbR^m$-Borel measure $\Lambda_{D_k}(dx)$ as defined in \eqref{e:Lambda_D} has symmetry group
\begin{equation}\label{e:S_Lambda(dx)=G}
{\mathcal S}^{\textnormal{dom}}(\Lambda_{D_k}) = {\mathcal G}.
\end{equation}
Moreover, there is a set of pivots $D_* = \{x_1,\hdots,x_{J_*}\} \supseteq D_k$ such that
\begin{equation}\label{e:span=Rm}
\textnormal{span}\{{\mathcal G}x_1,\hdots,{\mathcal G}x_{J_*}\} = \bbR^m
\end{equation}
and
\begin{equation}\label{e:S_Lambda(dx)=G_extended_to_Dstar}
{\mathcal S}^{\textnormal{dom}}(\Lambda_{D_*}) = {\mathcal G}.
\end{equation}
\end{proposition}
\begin{proof}
In this proof, $\subset$ denotes proper set inclusion, whereas $\subseteq$ denotes weak set inclusion.

Let ${\mathbf H}$ be the Haar probability measure on ${\mathcal G}$. Then, ${\mathbf H}({\mathcal G})=1$ and ${\mathbf H}(GA)={\mathbf H}(AG)={\mathbf H}(A)$ for any Borel subset $A$ of ${\mathcal G}$ and any $G \in {\mathcal G}$. By Lemma \ref{l:Lambda,Lambda_j_properties}, ($ii$), each of the orbits ${\mathcal G}x_j$, $x_j \in S^{m-1}$, is a compact set, and the number of connected components of an orbit is no more than the (finite) number of connected components of ${\mathcal G}$. By Lemma \ref{l:Lambda,Lambda_j_properties}, ($i$), two orbits are either disjoint or coincide. Suppose that the orbit ${\mathcal G}x_j$ has $n_j$ connected components. Since $G{\mathcal G} = {\mathcal G}$, each of these components satisfies
\begin{equation}\label{e:integ_over_orbit=1/nj}
\int_{{\mathcal G}}\delta_{x_j}(G{\mathcal G}x_j){\mathbf H}(dG) = \frac{1}{n_j}
\end{equation}
(see Meerschaert and Veeh \cite{meerschaert:veeh:1995}, p.\ 4; cf.\ Example \ref{ex:LambdaD(dx)}). For a given $D = \{x_1,\hdots,x_J\} \subseteq S^{m-1}$ associated with distinct orbits ${\mathcal G}x_1,\hdots,{\mathcal G}x_J$ and the corresponding measure $\Lambda_D(dx)$ in \eqref{e:Lambda_D}, note that Lemma \ref{l:Lambda,Lambda_j_properties}, ($iii$) and ($v$) imply that $\Lambda_D(dx)$ is supported on the compact set ${\mathcal G}D$, and that it assigns different values to each orbit.

By Lemma \ref{l:Lambda,Lambda_j_properties}, ($iv$), each $G \in {\mathcal G}$ is a symmetry of $\Lambda_D$, i.e., ${\mathcal S}^{\textnormal{dom}}(\Lambda_{D}) \supseteq {\mathcal G}$. If ${\mathcal S}^{\textnormal{dom}}(\Lambda_{D})x = {\mathcal G}x$ for all $x \in \bbR^m$, then ${\mathcal S}^{\textnormal{dom}}(\Lambda_{D}) \sim {\mathcal G}$. Since ${\mathcal G}$ is maximal, then ${\mathcal S}^{\textnormal{dom}}(\Lambda_D)\subseteq {\mathcal G}$, whence ${\mathcal G} = {\mathcal S}^{\textnormal{dom}}(\Lambda_{D})$. Otherwise,
\begin{equation}\label{e:there_is_x_st_Gx_contained_Sdomx}
\textnormal{there exists some element $x \in S^{m-1} \backslash\{0\}$ such that ${\mathcal G}x \subset {\mathcal S}^{\textnormal{dom}}(\Lambda_{D})x$.}
\end{equation}
Set $D_1=D \cup \{x\}$, and consider the measure $\Lambda_{D_1}$. By Lemma \ref{l:Lambda,Lambda_j_properties}, ($vii$), $K \in {\mathcal S}^{\textnormal{dom}}(\Lambda_{D_1}) \Rightarrow  K {\mathcal G}y = {\mathcal G}y$, $y \in D_1$. Therefore, ${\mathcal S}^{\textnormal{dom}}(\Lambda_{D_1}){\mathcal G}y = {\mathcal G}y$, $y \in D_1$. Since ${\mathcal G} \subseteq {\mathcal S}^{\textnormal{dom}}(\Lambda_{D_1})$, and in view of \eqref{e:there_is_x_st_Gx_contained_Sdomx}, this yields
\begin{equation}\label{e:S(Lambda(D1))x=Gdom.x}
{\mathcal S}^{\textnormal{dom}}(\Lambda_{D_1})x={\mathcal G}x \subset {\mathcal S}^{\textnormal{dom}}(\Lambda_{D})x.
\end{equation}
Lemma \ref{l:Lambda,Lambda_j_properties}, ($vi$) with $D'=D_1$ and expression \eqref{e:S(Lambda(D1))x=Gdom.x} imply that ${\mathcal S}^{\textnormal{dom}}(\Lambda_{D_1}) \subset {\mathcal S}^{\textnormal{dom}}(\Lambda_{D})$, since the connected components of ${\mathcal S}^{\textnormal{dom}}(\Lambda_{D_1})x$ must be strictly contained in those of ${\mathcal S}^{\textnormal{dom}}(\Lambda_{D})x$. Continue in this manner to obtain a decreasing nested sequence of symmetry groups $\{ {\mathcal S}^{\textnormal{dom}}(\Lambda_{D_k}) \}_{k \in \bbN \cup \{0\}}$, all of which contain ${\mathcal G}$. By Lemma \ref{l:Lambda,Lambda_j_properties}, $(viii)$, ${\mathcal G}={\mathcal S}^{\textnormal{dom}}(\Lambda_{D_k})$ for some $k$.

Let $D_k = \{x_1,\hdots,x_k\}$ be the set of pivot vectors that we arrive at by following the procedure above. If the ensemble of points in the orbits $\{{\mathcal G}x_1,\hdots,{\mathcal G}x_k\}$ does not contain a basis of $\bbR^m$, then since $I \in {\mathcal G}$ there is a vector $x_{k+1}$ such that $\textnormal{span}\{{\mathcal G}x_1,\hdots,{\mathcal G}x_k\} \subset \textnormal{span}\{{\mathcal G}x_1,\hdots,{\mathcal G}x_{k+1}\}$. By Lemma \ref{l:Lambda,Lambda_j_properties}, ($iv$) and ($vi$), ${\mathcal G} \subseteq {\mathcal S}^{\textnormal{dom}}(\Lambda_{D_{k+1}} ) \subseteq {\mathcal S}^{\textnormal{dom}}(\Lambda_{D_{k}} )$. Thus,
\begin{equation}\label{e:G=Sdom(LambdaDk+1)}
{\mathcal G} = {\mathcal S}^{\textnormal{dom}}(\Lambda_{D_{k+1}}).
\end{equation}
Proceeding in this fashion, for a finite $J_* $ we obtain a set of pivots $D_* = \{x_1,\hdots,x_{J_*}\}$ satisfying \eqref{e:span=Rm}. Expression \eqref{e:S_Lambda(dx)=G_extended_to_Dstar} is a consequence of \eqref{e:G=Sdom(LambdaDk+1)}. $\Box$
\end{proof}

\subsection{On the characterization of ${\Bbb G}$}\label{s:on_G}

The following theorem is the main result of this paper, and concerns the set ${\Bbb G}$ of the possible pairs of domain and range symmetry groups defined in \eqref{e:def_groups}. It consists of three statements. Two of them characterize ${\Bbb G}$ (in dimensions $(m,n) = (m,1)$ and $(2,2)$) and the other establishes a subset of ${\Bbb G}$ (for dimension $(m,2)$). In particular, the theorem settles in the negative a central issue, namely, whether in general ${\Bbb G}^{\textnormal{dom}} \times {\Bbb G}^{\textnormal{ran}} \subseteq {\Bbb G}$ (the opposite set inclusion being straightforward). This can be illustrated in dimension $(m,n) = (2,2)$, Table \ref{table:OFBF_m_n=2_symmgroups}; indeed, Lemma \ref{l:(m,2)-OFBFs:range_symm_vs_-I} shows that some choices of range symmetry groups imply restrictions on the choice of domain symmetry groups.
\begin{table}[h]
\centering
\begin{tabular}{ccc}\hline
   &  $\textnormal{types of }G^{\textnormal{ran}}_1(X)$ & $\textnormal{restriction on }G^{\textnormal{dom}}_1(X)$ \\\hline
   & ${\mathcal C}_2$ & -- \\
\textnormal{($i$)} &  ${\mathcal D}_2$ & $-I \in G^{\textnormal{dom}}_1(X)$\\
\textnormal{($ii$)} &  $SO(2)$ & $-I \notin G^{\textnormal{dom}}_1(X)$ \\
\textnormal{($iii$)} &  $O(2)$ & $-I \in G^{\textnormal{dom}}_1(X)$ \\ \hline
\end{tabular}\caption{OFBF, $(m,n)  = (2,2)$: restrictions on the domain group imposed by the range symmetry group.}\label{table:OFBF_m_n=2_symmgroups}
\end{table}

\begin{remark}
The presence of the restrictions $-I \in G^{\textnormal{dom}}_1(X)$ and $-I \notin G^{\textnormal{dom}}_1(X)$ should not be surprising. Note that $-I \in G^{\textnormal{dom}}_1(X)$ is equivalent to the law of the OFBF $X$ being reversible in the sense that
\begin{equation}\label{e:X(-t)=X(t)}
\{X(-t)\}_{t \in \bbR^m} \stackrel{{\mathcal L}}= \{X(t)\}_{t \in \bbR^m}.
\end{equation}
For example, for $m \geq 1$ and $n=1$, the condition \eqref{e:X(-t)=X(t)} always holds. This can be seen by noting that (for $n = 1$) the control measure $F_X(dx)$ in \eqref{e:F_X(dx)=EB(dx)B(dx)} satisfies $F_X(dx) = \overline{F_X(dx)} = F_X(-dx)$, and hence that
$$
{\Bbb E}X(t_1)X(t_2) = \int_{\bbR^m} (e^{i \langle t_1,x\rangle}-1)(e^{-i \langle t_2,x\rangle}-1) F_X(dx)
$$
$$
= \int_{\bbR^m} (e^{-i \langle t_1,x\rangle}-1)(e^{i \langle t_2,x\rangle}-1) F_X(dx) = {\Bbb E}X(-t_1)X(-t_2).
$$
In particular, in this case all domain symmetry groups contain the element $-I$ (see also part $(i)$ of Theorem \ref{t:OFBF_m,2} below).
\end{remark}

The theorem's proof consists of constructing OFBF spectral measures in polar form $F_X(x) = r^{-H}\Delta(d \theta) r^{-H^*}r^{-1}dr$ (see \eqref{e:int-rep-spectral-F}) which display attainable pairs of domain and range symmetries. It draws upon the characterization of range symmetries in Proposition \ref{p:Gran_structure}, as well as on the class of spectral measures on ${\mathcal B}(S^{m-1})$ with given domain symmetries, provided in Proposition \ref{p:existence_singular_measure}. Especially in dimension $(m,n) = (2,2)$, where both domain and range symmetries can be non-trivial, the argument boils down to reducing the construction of the OFBF spectral measure to that of building appropriate OFBM spectral measures in every spherical direction, where the spherical measure $\Delta(d \theta)$ has the desired domain symmetry group. In the theorem's statement, we denote by
\begin{equation}\label{e:Gdom_restricted}
{\Bbb G}_{\max}, \quad {\Bbb G}_{\max}|_{-I \hspace{0.5mm} \bullet \hspace{0.5mm}{\mathcal G} },
\end{equation}
respectively, and up to conjugacies, the class of maximal subgroups of $O(m)$ and the subclass of those which satisfy the restriction $- I \hspace{0.5mm}\bullet \hspace{0.5mm}{\mathcal G}$, where $\bullet$ stands for $\notin$ or $\in$.

\begin{theorem}\label{t:OFBF_m,2}
Consider the class of OFBFs in dimension ($m,n$), and satisfying the conditions \eqref{e:minReeig(H)=<maxReeig(H)<minReeig(E*)} and \eqref{e:maxeigLambda(Theta)>0=>minReeigLambda(Theta)>0}.
\begin{itemize}
\item [$(i)$] For $m \geq 2$ and $n = 1$, the set of possible pairs of domain and range symmetry groups is given by
\begin{equation}\label{e:Psi_n=1}
{\Bbb G} = {\Bbb G}_{\max}|_{-I \in {\mathcal G} } \times \{\pm 1\};
\end{equation}
\item [$(ii)$] for $m \geq 2$ and $n =2$, the set of possible pairs of domain and range symmetry groups satisfies
\begin{equation}\label{e:Psi_n=2}
{\Bbb G} \supseteq  {\Bbb G}_{\max}|_{-I \in {\mathcal G}} \times \{{\mathcal C}_2, {\mathcal D}_2, O(2)\};
\end{equation}
\item [$(iii)$] for $(m,n) = (2,2)$, the set of possible pairs of domain and range symmetry groups is given by
\begin{equation}\label{e:Psi}
{\Bbb G} = {\Bbb G}_{\max}|_{-I \in {\mathcal G}} \times \{{\mathcal C}_2, {\mathcal D}_2, O(2)\} \hspace{3mm}\bigcup \hspace{3mm}{\Bbb G}_{\max}|_{-I \notin {\mathcal G} }\times \{{\mathcal C}_2, SO(2)\},
\end{equation}
where ${\Bbb G}_{\max}$ consists of the maximal groups described in the middle column of Table \ref{table:OFBF_m=2_n=2_symmgroups}.
\end{itemize}
In \eqref{e:Psi_n=1}, \eqref{e:Psi_n=2} and \eqref{e:Psi}, equalities and set inclusion hold up to conjugacies.
\end{theorem}
\begin{proof}
Recall that Proposition \ref{p:akin_to_Prop2.3_Bierme_etal2007} provides the general representation \eqref{e:change_of_variables_into_polar_coordinates} of any spectral measure $F_X(dx)$ in polar coordinates. We will construct suitable OFBF spectral measures with the desired domain and range symmetry groups. In polar coordinates, these measures will have the form
\begin{equation}\label{e:FX(dx)=x^-H_Xi(dx)_x^-H_x^(-1)}
F_X(dx) = r^{-H} \Xi(d\theta) r^{-H^*} r^{-1} dr,
\end{equation}
where the spherical component is given by the measure
\begin{equation}\label{e:Xi(dtheta)}
\Xi(d\theta) = AA^* \Lambda(d \theta) + \overline{AA^*}\Lambda(-d\theta), \quad S_0 = S^{m-1},
\end{equation}
as constructed in Lemma \ref{l:Xi(dx)}. In \eqref{e:Xi(dtheta)}, the matrix $A$ will be appropriately chosen (together with $H$) so that $X$ has the desired range symmetry group, and the measure $\Lambda(d \theta)$ will be obtained from Proposition \ref{p:existence_singular_measure}.

Given the spectral measure \eqref{e:FX(dx)=x^-H_Xi(dx)_x^-H_x^(-1)}, we claim that $G^{\textnormal{dom}}_1(X)$ and $G^{\textnormal{ran}}_1(X)$ are determined, respectively, by the component $\Xi(d\theta)$, and by the latter combined with the parameter $H$. Indeed, for a given $F_X(dx)$ of the form \eqref{e:FX(dx)=x^-H_Xi(dx)_x^-H_x^(-1)}, by Lemmas \ref{l:G1_via_specmeas} and \ref{l:S(FX(dx))=G}, $(i)$,
\begin{equation}\label{e:Gdom(X)=SdomXi}
G^{\textnormal{dom}}_1(X) = {\mathcal S}^{\textnormal{dom}}(F_X)^* =  {\mathcal S}^{\textnormal{dom}}(\Xi)^*.
\end{equation}
Moreover, by Lemmas \ref{l:Gran1_via_specmeas} and \ref{l:S(FX(dx))=G}, $(ii)$,
\begin{equation}\label{e:Gran(X)=Gran(BH)}
G^{\textnormal{ran}}_1(X) = {\mathcal S}^{\textnormal{ran}}(F_X) = G^{\textnormal{ran}}_1(B_{H}),
\end{equation}
where
\begin{equation}\label{e:OFBM_main_theorem}
B_{H}, \quad F_{B_H}(dx) = r^{-H }\{AA^* \delta_{\{1\}}(d \theta) + \overline{AA^*} \delta_{\{-1\}}(d \theta)\}r^{-H^*}r^{-1}dr
\end{equation}
are an OFBM with parameters $(H, \Re(AA^*), \Im (AA^*))$ and its spectral measure expressed in polar coordinates notation. Relations \eqref{e:Gdom(X)=SdomXi} and \eqref{e:Gran(X)=Gran(BH)} show that the domain and range symmetry groups of the associated OFBF $X$ are determined, respectively, by $\Xi(d \theta)$ and by the latter and $H$, as claimed.

We now show $(ii)$. Fix ${\mathcal G}_2 \in \{{\mathcal C}_2, {\mathcal D}_2, O(2)\}$. By Corollary 5.1 in Didier and Pipiras \cite{didier:pipiras:2012}, ${\mathcal G}_2$ is a range symmetry group attainable by a time-reversible OFBM (see \eqref{e:time-reversible}). This means that we can choose $H$ and $AA^*$ such that $A_2 = 0$ in $A = A_1 + i A_2$ and the OFBM \eqref{e:OFBM_main_theorem} has range symmetry group $G^{\textnormal{ran}}_1(B_H) = {\mathcal G}_2$. Now pick ${\mathcal G}_1 \in {\Bbb G}_{\max}|_{-I \in {\mathcal G}}$. Then, by Lemma \ref{l:Xi(dx)}, ($ii$), the associated measure $\Xi(d \theta)$ in \eqref{e:Xi(dtheta)} has domain symmetry group ${\mathcal G}_1$. Then, relations \eqref{e:FX(dx)=x^-H_Xi(dx)_x^-H_x^(-1)}, \eqref{e:Gdom(X)=SdomXi} and \eqref{e:Gran(X)=Gran(BH)} ensure that the induced random field $X$ satisfies $(G^{\textnormal{dom}}_1(X),G^{\textnormal{ran}}_1(X)) = ({\mathcal G}_1,{\mathcal G}_2)$. Moreover, $X$ is proper as a consequence of \eqref{e:Lambda(Theta0)_is_proper} and \eqref{e:span=Rm}, as explained in Remark \ref{r:suff_condition_Delta_is_Hermitian_posdef} Therefore, $X$ is an OFBF (with exponents $E = I$ and $H$).

To show $(i)$, note that the constraint $F_X(-dx) = \overline{F_X(dx)}$ for any spectral measure boils down to $F_X(-dx) = F_X(dx)$ when $n = 1$. By Lemma \ref{l:G1_via_specmeas}, this is equivalent to $-I$ being in the domain symmetry group of the associated OFBF $X$, i.e., ${\Bbb G} \subseteq {\Bbb G}_{\max}|_{-I \in {\mathcal G} } \times \{\pm 1\}$. To establish the converse, the same procedure for showing $(ii)$ can be applied with added simplicity stemming from scalar-valued parameters $H$ and $AA^*$.

The statement $(iii)$ is an immediate consequence of Proposition \ref{p:OFBF_m,n=2,2}, shown below. $\Box$\\
\end{proof}

The following result complements Corollary \ref{c:range_groups_dim2}. It states that for every \textit{individual} group (in contrast with a pair thereof) described in Table \ref{table:OFBF_m=2_n=2_symmgroups}, there is an OFBF exhibiting that domain or range symmetry group.
\begin{corollary}\label{c:Gdom1,Gran1_dim_m=2,n=2}
For the class of OFBFs satisfying the conditions \eqref{e:minReeig(H)=<maxReeig(H)<minReeig(E*)} and \eqref{e:maxeigLambda(Theta)>0=>minReeigLambda(Theta)>0}, the classes ${\Bbb G}^{\textnormal{dom}}$ for $m = 2$ and ${\Bbb G}^{\textnormal{ran}}$ for $n = 2$ can be described as in the middle and right columns of Table \ref{table:OFBF_m=2_n=2_symmgroups}, respectively.
\end{corollary}
\begin{proof}
In view of Corollary \ref{c:range_groups_dim2}, we only need to describe the middle column in Table \ref{table:OFBF_m=2_n=2_symmgroups}. The latter is a consequence of Theorem \ref{t:OFBF_m,2}, $(iii)$, and the complete description of compact maximal groups in dimension $m= 2$  (see Cohen et al.\ \cite{cohen:meerchaert:rosinski:2010}, p.\ 2404). $\Box$\\
\end{proof}

\vspace{-0.5cm}
\begin{center}
\begin{table}[h]
\centering
\begin{tabular}{lcc}\hline
\textsl{type} & $G^{\textnormal{dom}}_1(X) \cong \hdots$ & $G^{\textnormal{ran}}_1(X)\cong \hdots$\\\hline
full & $O(2)$ & $O(2)$\\
rotational & $-$ & $SO(2)$ \\
cyclic & ${\mathcal C}_{\nu}$, $\nu \in \bbN$ & ${\mathcal C}_2$\\
dihedral & ${\mathcal D}_{\nu}$, $\nu \in \bbN$, ${\mathcal D}^*_1$ & ${\mathcal D}_2$\\\hline
\end{tabular}\caption{OFBF: description of the (individual) domain (${\Bbb G}^{\textnormal{dom}}$, $m = 2$) and range (${\Bbb G}^{\textnormal{ran}}$, $n = 2$) symmetry groups. In the middle column, isotropy corresponds to $G^{\textnormal{dom}}_1(X) = O(2)$ (with $W = I$; see Section \ref{s:isotropy}), the remaining cases describing all types of anisotropy.}\label{table:OFBF_m=2_n=2_symmgroups}
\end{table}
\end{center}
\vspace{-1cm}

The next proposition pertains to the case of dimension $(m,n) = (2,2)$. It shows that (almost) every possible combination of domain and range symmetry groups can be attained by some OFBF whose spectral density is either singular or has a density (i.e., is absolutely continuous with respect to the Lebesgue measure). The special case not covered by singular measures is that of pairs including the domain group $O(2)$; indeed, in Proposition \ref{p:isotropy_m,n_case} below, it is shown that isotropy implies that the spectral measure $F_X(dx)$ is absolutely continuous. Figure \ref{f:asymm_sphere} is provided to help visualize part of the argument (see also Examples \ref{ex:Gdom1=O(2)_Gran1=O(2)} and \ref{ex:Gdom1=D3_Gran1=SO(2)}).

\begin{proposition}\label{p:OFBF_m,n=2,2}
Let $({\mathcal G}_1,{\mathcal G}_2)$ be a pair of domain and range symmetry groups as described in \eqref{e:Psi}. Then,
\begin{itemize}
\item [(i)] if ${\mathcal G}_1$ is not conjugate to $O(2)$, there is an OFBF $X$ with singular spectral measure $F_X(dx)$ such that
    \begin{equation}\label{e:Gdom,Gran=G1,G2}
    (G^{\textnormal{dom}}_1(X),G^{\textnormal{ran}}_1(X)) = ({\mathcal G}_1,{\mathcal G}_2);
    \end{equation}
\item [(ii)] there is an OFBF $X$ with absolutely continuous spectral measure $F_X(dx) = f_X(x)dx$ such that \eqref{e:Gdom,Gran=G1,G2} holds.
\end{itemize}
\end{proposition}
\begin{proof}
Throughout this proof, we denote by
\begin{equation}\label{e:Gdom1,Gran1_generic}
G^{\textnormal{dom}}_1 \textnormal{ and }G^{\textnormal{ran}}_1
\end{equation}
generic domain and range symmetry groups, respectively, of an OFBF $X$ being constructed. In the end, we are able to write $G^{\textnormal{dom}}_1 = G^{\textnormal{dom}}_1(X)$ and $G^{\textnormal{ran}}_1 = G^{\textnormal{ran}}_1(X)$.

In both cases $(i)$ and $(ii)$, the proof is by construction, but based on different techniques. We will make use of the representation \eqref{e:FX(dx)=x^-H_Xi(dx)_x^-H_x^(-1)} in polar coordinates, where the choice of the pair of domain and range symmetry groups has to account for the restrictions described in Table \ref{table:OFBF_m_n=2_symmgroups}.

To show $(i)$, we will apply the same technique for showing $(ii)$ in Theorem \ref{t:OFBF_m,2}. Fix a group ${\mathcal G}_2 \in {\Bbb G}^{\textnormal{ran}}_1 = \{{\mathcal C}_2,{\mathcal D}_2, SO(2),O(2)\}$, where, without loss of generality, we can disregard conjugacies $W \in {\mathcal S}_{>0}(2,\bbR)$. Recall the notation \eqref{e:Gdom_restricted} for maximal subgroups. In light of Table \ref{table:OFBF_m_n=2_symmgroups}, choose the parameters $(H,\Re(AA^*),\Im(AA^*))$ for the OFBM \eqref{e:OFBM_main_theorem} according to the following recipe.
\begin{itemize}
\item If ${\mathcal G}_2$ implies $-I \in G^{\textnormal{dom}}_1$ by Table \ref{table:OFBF_m_n=2_symmgroups}, then choose
\begin{enumerate}
\item [(a)] any ${\mathcal G}_1 \in {\Bbb G}_{\max}|_{-I \in {\mathcal G}}$;
\item [(b)] a parametrization $(H,\Re(AA^*),\Im(AA^*))$ from the OFBM \eqref{e:OFBM_main_theorem} such that $\Re (AA^*)$ is positive definite, $A_2 = {\mathbf 0}$ and $G^{\textnormal{ran}}_1(B_H) = {\mathcal G}_2$.
\end{enumerate}
\item If ${\mathcal G}_2$ implies $-I \notin G^{\textnormal{dom}}_1$ by Table \ref{table:OFBF_m_n=2_symmgroups}, then choose
\begin{enumerate}
\item [(a)] any ${\mathcal G}_1 \in {\Bbb G}_{\max}|_{-I \notin {\mathcal G}}$;
\item [(b)] a parametrization $(H,\Re(AA^*),\Im(AA^*))$ from the OFBM \eqref{e:OFBM_main_theorem} such that $\Re (AA^*)$ is positive definite and $G^{\textnormal{ran}}_1(B_H) = {\mathcal G}_2 $.
\end{enumerate}
\item If ${\mathcal G}_2$ is compatible with either $-I \in G^{\textnormal{dom}}_1$ or $I \in G^{\textnormal{dom}}_1$ as described in Table \ref{table:OFBF_m_n=2_symmgroups}, then choose
\begin{enumerate}
\item [(a)] any ${\mathcal G}_1 \in {\Bbb G}_{\max}$;
\item [(b)] a parametrization $(H,\Re(AA^*),\Im(AA^*))$ from the OFBM \eqref{e:OFBM_main_theorem} such that $\Re (AA^*)$ is positive definite, $G^{\textnormal{ran}}_1(B_H) = {\mathcal G}_2 $
and $A_2 \neq {\mathbf 0}$ or $= {\mathbf 0}$ according to whether $-I \notin {\mathcal G}_1 $ or $-I \in {\mathcal G}_1$, respectively
\end{enumerate}
\end{itemize}
(see Didier and Pipiras \cite{didier:pipiras:2012}, Section 5.1, on how to choose $H$ and $AA^*$). Then, as in the proof of Theorem \ref{t:OFBF_m,2}, $(ii)$, relations \eqref{e:FX(dx)=x^-H_Xi(dx)_x^-H_x^(-1)}, \eqref{e:Gdom(X)=SdomXi} and \eqref{e:Gran(X)=Gran(BH)} ensure that the induced random field $X$ satisfies $(G^{\textnormal{dom}}_1(X),G^{\textnormal{ran}}_1(X)) = ({\mathcal G}_1,{\mathcal G}_2)$. In particular, depending on whether $-I \notin {\mathcal G}_1 $ or $-I \in {\mathcal G}_1$, Lemma \ref{l:Xi(dx)} ensures that the associated measure $\Xi(d \theta)$ in \eqref{e:Xi(dtheta)} has domain symmetry group ${\mathcal G}_1$. Moreover, $X$ is, indeed, proper, and thus an OFBF (with exponents $E = I$ and $H$), which is a consequence of \eqref{e:Lambda(Theta0)_is_proper} and \eqref{e:span=Rm}.

Because, by assumption, ${\mathcal G}_1$ is not (conjugate to) $O(2)$, the list of the possible domain groups displayed in Table \ref{table:OFBF_m=2_n=2_symmgroups} shows that for every $x \neq 0$, the orbit ${\mathcal G}_1x$ consists of finitely many points. Moreover, by Lemma \ref{l:Lambda,Lambda_j_properties} and expression \eqref{e:def_Xi(d_theta)}, the support of the measure $\Xi(d\theta)$ consists of the orbits that enter into the construction of the measure, namely, a finite number of points in $S^{1}$. Therefore, the resulting spectral measure $F_X(dx) = r^{-H}\Xi(d\theta)r^{-H^*}r^{-1}dr$ is singular.\\

To show $(ii)$, it will suffice to take $E = I$. Let $H \in M(2,\bbR)$ be a matrix whose eigenvalues satisfy \eqref{e:maxeigLambda(Theta)>0=>minReeigLambda(Theta)>0}. Consider the OFBF class whose harmonizable representation is
\begin{equation}\label{e:OFBF_from_an_OFBM}
X = \{X(t)\}_{t \in \bbR^m} = \Big \{\int_{\bbR^2} (e^{i \langle t, x \rangle}- 1)\|x\|^{-H_{E}} \Delta^{1/2} \Big( \frac{x}{\|x\|}\Big)\widetilde{B}(dx)
\Big \}_{t \in \bbR^2},
\end{equation}
where $H_E$ is as in \eqref{e:operator_scaling_under_a.c.}, $\Delta^{1/2} \in {\mathcal S}_{\geq 0}(2,\bbC)$ is a Hermitian function whose real parts' maximal and minimal eigenvalues are bounded and bounded away from zero, respectively (cf.\ Remark \ref{r:suff_condition_Delta_is_Hermitian_posdef}). By Theorem 3.1 in Baek et al.\ \cite{baek:didier:pipiras:2014}, the random field \eqref{e:OFBF_from_an_OFBM} is a well-defined OFBF with exponents $(I ,H)$ and spectral density
\begin{equation}\label{e:specdens_OFBF_2,2}
f_X(x) = \|x\|^{-H_{E}} \Delta \Big( \frac{x}{\|x\|} \Big)\|x\|^{-H^*_{E}}.
\end{equation}
Then,
$$
{\Bbb E}X(s)X(t)^* = \int_{\bbR^2} (e^{i \langle s, x \rangle}- 1) (e^{-i \langle t, x \rangle}- 1) \|x\|^{- H_E}\Delta\Big( \frac{x}{\|x\|}\Big)\|x\|^{- H^*_E}dx
$$
$$
= \int^{2 \pi}_{0}\int^{\infty}_{0} (e^{i \langle s, r\theta \rangle}- 1) (e^{-i \langle t, r\theta \rangle}- 1)  r^{- H}\Delta( (\cos \theta , \sin \theta)^* ) r^{-H^*} r^{-1} dr d \theta,
$$
where the equality is a consequence of making a change of variables into (Euclidean) polar coordinates. It will suffice to define the spherical function $\Delta$ appropriately.

In light of Table \ref{table:OFBF_m=2_n=2_symmgroups}, we will break up the construction according to the types of domain symmetry groups, i.e., groups of the form ${\mathcal D}_{\nu}$, ${\mathcal C}_{\nu}$ or $O(2)$ (Cases 1, 2 or 3, respectively; see also \eqref{e:Dv_Cv}).\\

\noindent \textbf{Case 1}: Fix $\nu \geq 1$, and set ${\mathcal G}_1 = {\mathcal D}_{\nu}$. Now pick a range group ${\mathcal G}_2$ that is compatible with ${\mathcal G}_1$ according to \eqref{e:Psi}. In other words, depending on whether $\nu$ is odd or even, then $-I \notin {\mathcal D}_{\nu}$ or $-I \in {\mathcal D}_{\nu}$, respectively. Consider a parameter $H$ and spherical parameters
\begin{equation}\label{e:Delta1D_Delta2D}
\Delta_{1,{\mathcal D}}, \Delta_{2,{\mathcal D}} \in {\mathcal S}_{>0}(2,\bbC)
\end{equation}
(i.e., $\Delta_{\cdot,{\mathcal D}}$ stands for $AA^*$ in \eqref{e:OFBM_main_theorem}) such that
\begin{equation}\label{e:Lambda1D,Lambda2D}
\Delta_{1,{\mathcal D}} =
\left\{\begin{array}{cc}
\Re \Delta_{1,{\mathcal D}} + i \Im \Delta_{1,{\mathcal D}}, & \textnormal{$\nu$ is odd},\\
\Re \Delta_{1,{\mathcal D}}, & \textnormal{$\nu$ is even},
\end{array}\right.
\quad
\Delta_{2,{\mathcal D}} =
\left\{\begin{array}{cc}
\overline{\Delta_{1,{\mathcal D}}}, & \textnormal{$\nu$ is odd},\\
\Re \Delta_{2,{\mathcal D}} \neq \Re \Delta_{1,{\mathcal D}}, & \textnormal{$\nu$ is even},
\end{array}\right.
\end{equation}
and which yield the same group
\begin{equation}\label{e:G(BH)=G2}
G^{\textnormal{ran}}_1(B_H) = {\mathcal G}_2
\end{equation}
as the range symmetry group of two OFBMs of the form \eqref{e:OFBM_main_theorem}, both with $H$ as the Hurst exponent, whereas, for one, $AA^* =  \Delta_{1,{\mathcal D}}$, and for the other, $AA^* =  \Delta_{2,{\mathcal D}}$ (see Didier and Pipiras \cite{didier:pipiras:2012}, Section 5.1, on how to choose $H$ and the spherical parameters). Note that, depending on whether $\nu$ is odd or even, we have $-I \notin G^{\textnormal{dom}}_1$ or $-I \in G^{\textnormal{dom}}_1$, respectively. Since $-I$ corresponds to a $\pi$ rotation, when $\nu$ is even such a rotation must take a slice of the sphere to another slice where it takes the same value, the opposite holding for when $\nu$ is odd (this can be visualized in Figure \ref{f:asymm_sphere}). In addition, when $\nu$ is even, we can always suppose
\begin{equation}\label{e:Delta2=cDelta1}
\Re \Delta_{2,{\mathcal D}} = c \Re \Delta_{1,{\mathcal D}}
\end{equation}
for some $c \in (0,\infty) \backslash \{1\}$ since multiplication by a nonzero constant does not alter the domain symmetry group of an OFBF, and
$$
\Re \Delta_{1,{\mathcal D}}, \Re \Delta_{2,{\mathcal D}} \in {\mathcal S}_{> 0}(2,\bbR).
$$

For $x \in S_0 = S^{1}$ and its angular component $\theta(x)$, define the matrix-valued function $\Delta(x)$ appearing in \eqref{e:OFBF_from_an_OFBM} and \eqref{e:specdens_OFBF_2,2} as
\begin{equation}\label{e:angular_region_Dv}
\Delta( x) = \left\{\begin{array}{cc}
\Delta_{1,{\mathcal D}}, & \theta ( x) \in \frac{2\pi}{\nu}\Big[\frac{1}{4}+(k-1) , \frac{3}{4}+(k-1) \Big);\\
\Delta_{2,{\mathcal D}}, & \theta( x) \in \frac{2\pi}{\nu}\Big[k-1, \frac{1}{4}+(k-1) \Big)\cup \frac{2\pi}{\nu}\Big[\frac{3}{4}+(k-1) , k \Big),\\
\end{array}\right.
\end{equation}
for $k = 1,2,3,\hdots,\nu$. In other words, we can interpret the function $\Delta(\cdot)$ as dividing up the sphere $S^1$ into slices of angular size $\frac{1}{4}\frac{2\pi}{\nu}$, where it takes values $\Delta_{1,{\mathcal D}}$ or $\Delta_{2,{\mathcal D}}$. In particular, each consecutive pair of slices associated with the value $\Delta_{1,{\mathcal D}}$ is followed by a pair associated with the value $\Delta_{2,{\mathcal D}}$ (cf.\ Figure \ref{f:asymm_sphere}, left column). Moreover,
\begin{equation}\label{e:Lambda(Opi_x)}
\Delta(-x) =  \left\{\begin{array}{cc}
\Delta(x) \in {\mathcal S}_{\geq 0}(2,\bbR)   & \nu \textnormal{ is even}; \\
\overline{\Delta(x)}    & \nu \textnormal{ is odd},
\end{array}\right. \quad x \in S^1.
\end{equation}
Therefore, for $\nu \in \bbN$,
\begin{equation}\label{e:Lambda(-x)=Lambda(x)conj}
\Delta(-x) = \overline{\Delta(x)}, \quad x \in S^1.
\end{equation}

We now study the symmetries of the resulting OFBF spectral measure interpreted in terms of polar coordinates as in \eqref{e:OFBM_polar_harmonizable}, where the spherical measure is given by $\Delta(d \theta) = \Delta(x)dx$ for $x \in S^{1}$ and $\Delta(x)$ is defined by \eqref{e:angular_region_Dv}. In regard to range symmetries, \eqref{e:Delta2=cDelta1} implies that for any $\Theta \in \textnormal{supp}\{\Delta(d \theta)\}$ we can write
$$
\Delta (\Theta) = c_{\Theta}\Re \Delta_{1,{\mathcal D}} + i \hspace{1mm}d_{\Theta}\Im \Delta_{1,{\mathcal D}}
$$
for some pair $c_{\Theta}> 0$ and $d_{\Theta} \geq 0$, where $d_{\Theta}$ is $>$ or $= 0$ when $\nu$ is odd or even, respectively. Since
\begin{equation}\label{e:W_Theta=c^(1/2)_Theta*ReDelta1^(1/2)}
W_{\Theta} = \Re(\Delta(\Theta))^{1/2} = c^{1/2}_{\Theta}\Re(\Delta_{1,{\mathcal D}})^{1/2}
\end{equation}
in \eqref{e:W_Theta}, then
$$
W^{-1}_{\Theta}H W_{\Theta} = \Re \Delta_{1,{\mathcal D}}^{-1/2} \hspace{1mm}H \hspace{1mm}\Re \Delta_{1,{\mathcal D}}^{1/2}
$$
and
$$
\Pi_{r,\Theta} = r^{- \Re \Delta_{1,{\mathcal D}}^{-1/2}\hspace{0.5mm}H \hspace{0.5mm}\Re \Delta_{1,{\mathcal D}}^{1/2}}r^{- \Re \Delta_{1,{\mathcal D}}^{1/2}\hspace{0.5mm}H^* \hspace{0.5mm}\Re \Delta_{1,{\mathcal D}}^{-1/2}}, \quad r > 0, \quad \Pi_{I,\Theta} = \Re \Delta_{1,{\mathcal D}}^{-1/2}\hspace{1mm}d_{\Theta}\Im \Delta_{1,{\mathcal D}} \hspace{1mm}\Re \Delta_{1,{\mathcal D}}^{-1/2}.
$$
In particular,
$$
{\mathcal C}_{O(2)}(\Pi_{I,\Theta}) =
\left\{\begin{array}{cc}
SO(2), & \textnormal{if } \Im \Delta_{1,{\mathcal D}} \neq {\mathbf 0} \textnormal{ ($\nu$ is odd)},\\
O(2), & \textnormal{if }\Im \Delta_{1,{\mathcal D}} = {\mathbf 0} \textnormal{ ($\nu$ is even)}.\\
\end{array}\right.
$$
This holds because, for odd $\nu$, by \eqref{e:Delta1D_Delta2D} the matrix $\Im \Delta_{1,{\mathcal D}}$ is skew-symmetric, and thus so is $\Re \Delta_{1,{\mathcal D}}^{-1/2}\Im \Delta_{1,{\mathcal D}}\Re \Delta_{1,{\mathcal D}}^{-1/2}$ (cf.\ Lemma 5.1 in Didier and Pipiras \cite{didier:pipiras:2012}). Therefore, \eqref{e:Gran_structure_n=2} can be rewritten as
\begin{equation}\label{e:G_H,Theta_specdens_Case1}
G_{H,\Theta} = \left\{\begin{array}{cc}
\Re(\Delta_{1,{\mathcal D}})^{1/2} \Big( \bigcap_{r > 0} {\mathcal C}_{O(2)}(\Pi_{r}) \cap SO(2) \Big) \Re(\Delta_{1,{\mathcal D}})^{-1/2}, & \textnormal{$\nu$ is odd},\\
\Re(\Delta_{1,{\mathcal D}})^{1/2} \Big( \bigcap_{r > 0} {\mathcal C}_{O(2)}(\Pi_{r}) \Big) \Re(\Delta_{1,{\mathcal D}})^{-1/2}, & \textnormal{$\nu$ is even},\\
\end{array}\right.
\end{equation}
where $\Pi_{r}$ is the scaling function \eqref{e:Pi-rTheta_Pi-I,Theta} (not dependent on $\Theta$) of an OFBM with parameters $H$ and $AA^* = \Delta_{1,{\mathcal D}}$, namely, $B_H$ in \eqref{e:G(BH)=G2}. From \eqref{e:Gran_structure}, \eqref{e:W_Theta=c^(1/2)_Theta*ReDelta1^(1/2)} and \eqref{e:G_H,Theta_specdens_Case1}, we obtain
$$
G^{\textnormal{ran}}_1(X) = G^{\textnormal{ran}}_1(B_H) = {\mathcal G}_2.
$$

In regard to domain symmetries, by Lemmas \ref{l:G1_via_specmeas} and \ref{l:(m,n)=(2,2)_construction=>group_contained_O(2)}, we know that $G^{\textnormal{dom}}_1(X)^* = {\mathcal S}^{\textnormal{dom}}(F_X) \subseteq O(2)$. We first look at reflection matrices. A matrix $F_{k \frac{2 \pi}{\nu}} \in O(2) \backslash SO(2)$ determines a reflection axis at the angle $\frac{k}{2}\frac{2 \pi}{\nu}$. If $k$ is odd, this angle can be rewritten as
\begin{equation}\label{e:AA*}
\frac{k}{2}\frac{2 \pi}{\nu} = \Big( \frac{k+1}{2}- \frac{1}{2}\Big) \frac{2 \pi}{\nu}, \quad \frac{k+1}{2} \in \bbN,
\end{equation}
In other words, the reflection axis splits a pair of angular slices where $\Delta(x)$ takes the value $\Delta_{1,{\mathcal D}}$. Alternatively, if $k$ is even, then
\begin{equation}\label{e:AA*conj}
\frac{k}{2} \frac{2\pi}{\nu}, \quad \frac{k}{2} \in \bbN.
\end{equation}
In this case, the reflection axis splits a pair of angular slices where $\Delta(x)$ takes the value $\Delta_{2,{\mathcal D}}$. Combined with the fact that $F_X(dx)$ is a Hermitian measure, in view of \eqref{e:Gdom1_in_terms_of_fX} this implies that $F_X(F_{k \frac{2 \pi}{\nu}}dx) = F_X(dx)$, i.e., $F_{k \frac{2 \pi}{\nu}} \in G^{\textnormal{dom}}_1(X)$, $k = 1,\hdots,\nu$. Turning to rotation matrices, it is clear that, by \eqref{e:Gdom1_in_terms_of_fX} and the construction of $F_X(dx)$,
$$
F_X\Big(O_{k \frac{2\pi}{\nu}}dx \Big) = F_X(dx), \quad k = 1,\hdots,\nu,
$$
Hence, we also have $O_{k \frac{2 \pi}{\nu}} \in G^{\textnormal{dom}}_1(X)$, $k = 1,\hdots,\nu$. Moreover, by construction, no other rotation or reflection matrices appear in $G^{\textnormal{dom}}_1(X)$. Therefore, $G^{\textnormal{dom}}_1(X) = {\mathcal D}_{\nu}$. \\

\noindent \textbf{Case 2}: Fix $\nu \geq 1$, and set ${\mathcal G}_1 = {\mathcal C}_{\nu}$. Again pick a range group ${\mathcal G}_2$ that is compatible with ${\mathcal G}_1$ according to \eqref{e:Psi}, i.e., in other words, depending on whether $\nu$ is odd or even, then $-I \notin {\mathcal C}_{\nu}$ or $-I \in {\mathcal C}_{\nu}$, respectively. As in Case 1, by analogy to \eqref{e:Lambda1D,Lambda2D}, we pick an appropriate $H$ and define the matrices $\Delta_{i,{\mathcal C}}$, $i=1,2,3,4$, so that their imaginary parts are zero or not depending on whether $\nu$ is even or odd. More specifically,
$$
\Delta_{1,{\mathcal C}} =
\left\{\begin{array}{cc}
\Re \Delta_{1,{\mathcal C}} + i \Im \Delta_{1,{\mathcal C}}, & \textnormal{$\nu$ is odd},\\
\Re \Delta_{1,{\mathcal C}}, & \textnormal{$\nu$ is even},
\end{array}\right.
\quad
\Delta_{2,{\mathcal C}} =
\left\{\begin{array}{cc}
\Re \Delta_{2,{\mathcal C}} + i \Im \Delta_{2,{\mathcal C}}, & \textnormal{$\nu$ is odd},\\
\Re \Delta_{2,{\mathcal C}}, & \textnormal{$\nu$ is even},
\end{array}\right.
$$
\begin{equation}\label{e:Delta1C,Delta2C,Delta3C,Delta4C}
\Delta_{3,{\mathcal C}} =
\left\{\begin{array}{cc}
\overline{\Delta_{1,{\mathcal C}}}, & \textnormal{$\nu$ is odd},\\
\Re \Delta_{3,{\mathcal C}}, & \textnormal{$\nu$ is even},
\end{array}\right.
\quad
\Delta_{4,{\mathcal C}} =
\left\{\begin{array}{cc}
\overline{\Delta_{2,{\mathcal C}}}, & \textnormal{$\nu$ is odd},\\
\Re \Delta_{4,{\mathcal C}}, & \textnormal{$\nu$ is even}.
\end{array}\right.
\end{equation}
In \eqref{e:Delta1C,Delta2C,Delta3C,Delta4C},
$$
\left\{\begin{array}{cc}
\Re \Delta_{1,{\mathcal C}} =  \Re \Delta_{3,{\mathcal C}} \neq \Re \Delta_{2,{\mathcal C}} =  \Re \Delta_{4,{\mathcal C}}, & \textnormal{if $\nu$ is odd},\\
\Re \Delta_{1,{\mathcal C}}, \Re \Delta_{2,{\mathcal C}}, \Re \Delta_{3,{\mathcal C}}, \Re \Delta_{4,{\mathcal C}} \textnormal{ are pairwise distinct}, & \textnormal{if $\nu$ is even},
\end{array}\right.
$$
and $\Delta_{1,{\mathcal D}} \in M(n,\bbC)$, $i = 1,2,3,4$, correspond to the spherical parameters with positive definite real parts associated with OFBMs \eqref{e:OFBM_main_theorem} displaying the same range symmetry group ${\mathcal G}_2$. For $x \in S^{1}$, define the matrix-valued function
\begin{equation}\label{e:angular_region_Cv}
\Delta( x) = \left\{\begin{array}{cc}
\Delta_{1,{\mathcal D}}, & \theta ( x) \in \frac{2\pi}{\nu}\Big[k-1 , \frac{1}{4}+(k-1) \Big);\\
\Delta_{2,{\mathcal D}}, & \theta( x) \in \frac{2\pi}{\nu}\Big[\frac{1}{4}+(k-1), \frac{1}{4}+(k-1) \Big);\\
\Delta_{3,{\mathcal D}}, & \theta( x) \in \frac{2\pi}{\nu}\Big[\frac{1}{2}+(k-1), \frac{3}{4}+(k-1) \Big);\\
\Delta_{4,{\mathcal D}}, & \theta( x) \in \frac{2\pi}{\nu}\Big[\frac{3}{4}+(k-1), 1+(k-1) \Big),\\
\end{array}\right.
\end{equation}
for $k = 1,2,3,\hdots,\nu$. In other words, we can interpret the function $\Delta(\cdot)$ as dividing up the sphere $S^1$ into slices of angular size $\frac{1}{4}\frac{2\pi}{\nu}$, where it takes values $\Delta_{1,{\mathcal D}}$, $\Delta_{2,{\mathcal D}}$, $\Delta_{3,{\mathcal D}}$ or $\Delta_{4,{\mathcal D}}$ (cf.\ Figure \ref{f:asymm_sphere}, right column).

We now study the symmetries of the spectral measure. In regard to range symmetries, the same type of argument as in Case 1 can be used. In regard to domain symmetries, again by Lemma \ref{l:(m,n)=(2,2)_construction=>group_contained_O(2)}, we know that $G^{\textnormal{dom}}_1(X)^* = {\mathcal S}^{\textnormal{dom}}(F_X) \subseteq O(2)$. Irrespective of whether $\nu$ is odd or even, $F_X(O_{k2\pi/\nu}dx)= F_X(dx)$, $k = 1,\hdots,\nu$. Moreover, the matrix value of $\Delta(x)$ on each of four consecutive slices are, by construction, pairwise distinct. This implies that there is no reflection in the symmetry group of $F_X(dx)$. Moreover, both conditions \eqref{e:Lambda(Opi_x)} and \eqref{e:Lambda(-x)=Lambda(x)conj} hold. Therefore, $G^{\dom}_1(X) = {\mathcal G}_1$.\\

\noindent \textbf{Case 3}: It remains to consider the domain symmetry group $O(2)$. So, pick a compatible range symmetry group ${\mathcal G}_2$ and consider the spectral density \eqref{e:specdens_OFBF_2,2} with
\begin{equation}\label{e:Delta_associated_with_O(2)}
\Delta((\cos\theta, \sin \theta)^*) \equiv \Delta.
\end{equation}
In \eqref{e:Delta_associated_with_O(2)} and \eqref{e:specdens_OFBF_2,2}, $\Delta \in {\mathcal S}_{\geq 0}(n,\bbC)$ is chosen so that $\Re \Delta$ is symmetric positive definite and $\Delta$ and $H$ correspond to the spectral parametrization of an OFBM \eqref{e:OFBM_main_theorem} with (range) symmetry group ${\mathcal G}_2$. $\Box$\\
\end{proof}

\begin{figure}[h]
\centerline{\includegraphics[height=2in,width=2.8in]{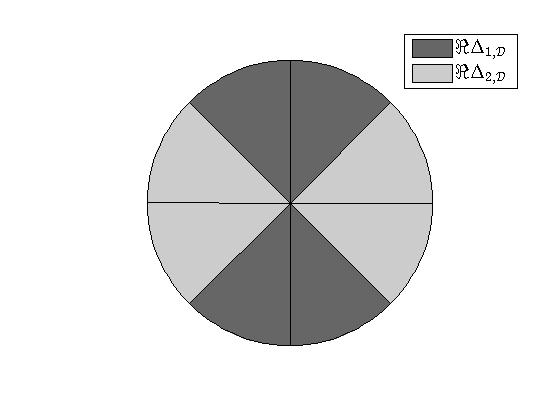}
\includegraphics[height=2in,width=2.7in]{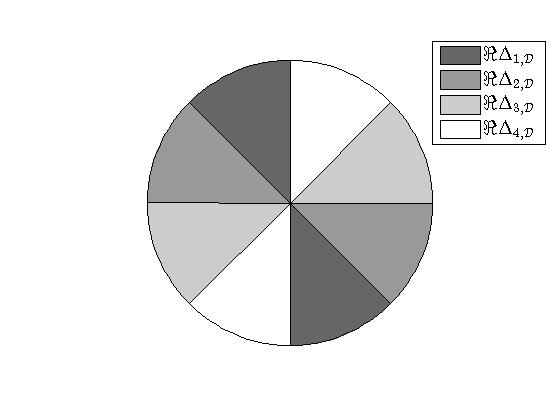}}
	\centerline{\includegraphics[height=2in,width=2.8in]{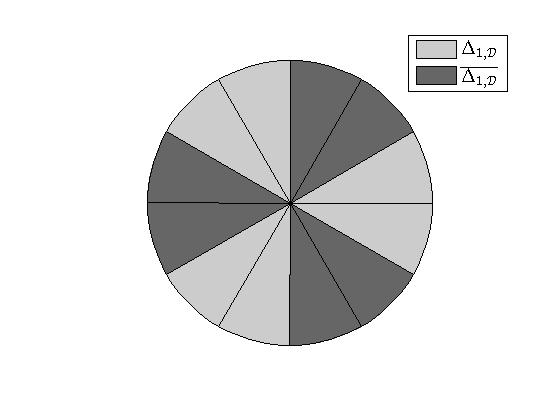}
	\includegraphics[height=2in,width=2.7in]{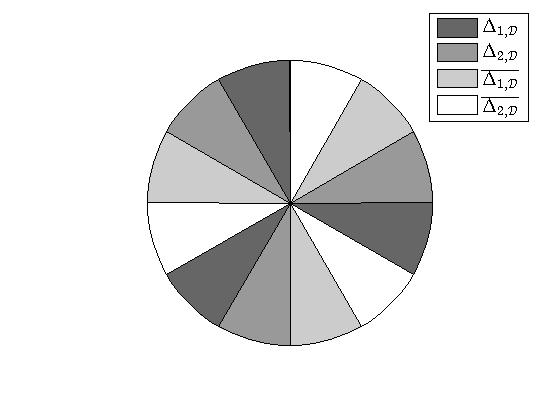}}
	\caption{\label{f:asymm_sphere} \textbf{The spectral density of anisotropic OFBFs with Euclidean spherical component} ($G^{\dom}_1$). The shading of each arc is extended to the corresponding disk slice for ease of visualization. From left to right: top row, ${\mathcal D}_2$ and ${\mathcal C}_2$; bottom row, ${\mathcal D}_3$ and ${\mathcal C}_3$. In the panels for ${\mathcal D}_2$ and ${\mathcal D}_3$, the reflection axes appear as black lines splitting a slice of a given shade.}
\end{figure}

The following examples illustrate the study of the structures of range and domain symmetry groups provided in Propositions \ref{p:Gran_structure} and \ref{p:Gdom1=>maximality} (as well as in Theorem \ref{t:OFBF_m,2} and Proposition \ref{p:OFBF_m,n=2,2}). The first one is taken from Didier et al.\ \cite{didier:meerschaert:pipiras:2016:exponents}; in this case, the domain and range symmetry groups can be obtained based on a direct computation. In the second example, we make use of the construction in the proof of Proposition \ref{p:OFBF_m,n=2,2}. An application to the problem of the identifiability of the exponents of OFBF is given in Example \ref{ex:sets_of_exponents} below.
\begin{example}\label{ex:Gdom1=O(2)_Gran1=O(2)}
Let $X = \{X(t)\}_{t \in \bbR^2}$ be an $\bbR^2$-valued OFBF with spectral density $f_X(x) = \|x\|^{-\gamma} I $, $x \in \bbR^2 \backslash \{0\}$, $2 < \gamma < 4$, where $\|\cdot\|$ denotes the Euclidean norm and $I$ is the identity matrix. This means that its covariance function can be written as
\begin{equation}\label{e:OFBF}
\Gamma(s,t) = I \int_{\bbR^2} (e^{i \langle s,x\rangle}-1)(e^{-i \langle t,x\rangle}-1)\frac{1}{\|x\|^{\gamma}}\hspace{1mm} dx,
\end{equation}
where $\langle \cdot, \cdot \rangle$ is the Euclidean inner product. By \eqref{e:OFBF} and a change of variables, $X$ is $(E,H)$-o.s.s.\ with $E = I$, $H = h I$, where $h = (\gamma-2)/2$. Since $\Gamma(s,t)$ is a scalar matrix (i.e., a scalar times the identity) for $s,t \in \bbR^2$, then the condition
\begin{equation}\label{e:AGamma(s,t)A*=Gamma(s,t)}
A\Gamma(s,t)A^* = \Gamma(s,t)
\end{equation}
for $A \in GL(2,\bbR)$ implies that $AA^* = I$, namely, $A \in O(2)$. Moreover, any $A \in O(2)$ satisfies \eqref{e:AGamma(s,t)A*=Gamma(s,t)}. Hence, $G^{\textnormal{ran}}_1(X) = O(2)$. Now note that, by a change of variables in \eqref{e:OFBF} and the continuity of the spectral density except at zero, $A \in G^{\textnormal{dom}}_1(X) \Leftrightarrow \|A^* x\| = \|x\|$, $x \in \bbR^{m}\backslash\{0\}$, i.e., $A \in O(2)$. As a consequence, $G^{\textnormal{dom}}_1(X) = O(2)$.
\end{example}

\begin{example}\label{ex:Gdom1=D3_Gran1=SO(2)}
Let $X = \{X(t)\}_{t \in \bbR^2}$ be an $\bbR^2$-valued OFBF with spectral density
$$
f_X(x) = \|x\|^{-H_{E}} \Delta \Big( \frac{x}{\|x\|} \Big)\|x\|^{-H^*_{E}}
$$
(see \eqref{e:H_E}, \eqref{e:OFBF_from_an_OFBM} and \eqref{e:specdens_OFBF_2,2}). For the sake of illustration, we look at a subcase, namely, we want to construct an OFBF $X$ with symmetry groups
\begin{equation}\label{e:choice_groups_example}
G^{\textnormal{dom}}_1(X) = {\mathcal D}_3, \quad G^{\textnormal{ran}}_1(X) = SO(2).
\end{equation}
So, choose the parameters $H$, $\Delta_{1,{\mathcal D}}$, $\Delta_{2,{\mathcal D}}$ such that
$$
\Delta_{1,{\mathcal D}} = \Re \Delta_{1,{\mathcal D}} + i \Im \Delta_{1,{\mathcal D}}, \quad \Delta_{2,{\mathcal D}} = \overline{\Delta_{1,{\mathcal D}} }.
$$
where $(H,\Re \Delta_{1,{\mathcal D}},\Im \Delta_{1,{\mathcal D}})$ corresponds to the parametrization \eqref{e:(H,ReAA*,ImAA*)} of an OFBM with range symmetry group $SO(2)$. In particular, $\Im \Delta_{1,{\mathcal D}} \neq {\boldsymbol 0}$. The function $\Delta(\cdot)$ then breaks up the sphere $S^1$ into slices of angular size $\frac{1}{4}\frac{2\pi}{3}$, where it takes values $\Delta_{1,{\mathcal D}}$ or $\Delta_{2,{\mathcal D}}$. This is depicted in Figure \ref{f:asymm_sphere}, bottom left panel. A detailed justification of why \eqref{e:choice_groups_example} holds is provided in the proof of Proposition \ref{p:OFBF_m,n=2,2}. Intuitively, since the spectral density $f_X(x)$ of the OFBF coincides, in every direction, with that of an OFBM with (range) symmetry group $SO(2)$, then $G^{\textnormal{ran}}_1(X) = SO(2)$. Moreover, of all the possible domain groups in Table \ref{table:OFBF_m=2_n=2_symmgroups}, only the application of ${\mathcal D}_3$ leaves the sphere in Figure \ref{f:asymm_sphere} unaltered. Therefore, $G^{\textnormal{dom}}_1(X) = {\mathcal D}_3$.
\end{example}

\begin{remark}\label{r:difficulties_in_extending}
The description of all pairs of symmetry groups in general dimension $(m,n)$, $m,n \in \bbN$, remains an open problem. In regard to range symmetries, solving commutativity relations of the type involved in \eqref{e:Gran_structure} is algebraically intense in dimension $n \geq 3$ (cf.\ Didier and Pipiras \cite{didier:pipiras:2012}). Remark \ref{r:-I_on_mathcalG} below describes the technical difficulties surrounding the construction of a spectral measure for general $m \in \bbN$ when the domain symmetries include $-I$.
\end{remark}

\subsection{Applications}

In this section, we provide two applications of the analysis in the preceding sections: one is a parametric characterization of isotropic OFBF, and the other is the set of exponents of OFBF in dimension $(m,n) = (2,2)$. Throughout this section, ${\mathcal E}^{\textnormal{dom}}_{H}(X)$ denotes the set of domain exponents given some range exponent $H$, and likewise, ${\mathcal E}^{\textnormal{ran}}_{E}(X)$ denotes the set of range exponents given some domain exponent $E$ (see also Didier et al.\ \cite{didier:meerschaert:pipiras:2016:exponents}).

\subsubsection{On the parametric characterization of isotropy}\label{s:isotropy}

Recall that a random field $X = \{X(t)\}_{t \in \bbR^m}$ is called \textit{isotropic} when its law is invariant under orthogonal transformations, namely,
\begin{equation}\label{e:def_isotropy}
\{ X(Ot)\}_{t \in \bbR^m} \stackrel{{\mathcal L}}= \{ X(t) \}_{t \in \bbR^m}, \quad O \in O(m).
\end{equation}
In other words, $G^{\textnormal{dom}}_1(X) = O(m)$. The existence of a commuting domain exponent of the form $E_0 = \eta I$ is not generally sufficient for isotropy. For example, an OFBM, for which the domain exponent is just a scalar, may not be time-reversible (isotropic; see Didier and Pipiras \cite{didier:pipiras:2011}, Theorem 6.1). The inequivalence between isotropy and Euclidean spherical coordinates is further illustrated in Figure \ref{f:asymm_sphere}, which depicts the Fourier spectrum of anisotropic OFBFs with Euclidean spherical components. In fact, in the next proposition we show that, even though a scalar matrix-valued domain exponent is a necessary condition for isotropy, sufficiency is only attained in the presence of the spherical symmetry of the measure $\Delta(d \theta)$ on the Euclidean sphere.

\begin{proposition}\label{p:isotropy_m,n_case}
Let $X = \{X(t)\}_{t \in \bbR^m}$ be an $\bbR^n$--valued OFBF with exponents $(E,H)$. Suppose $X$ satisfies the condition \eqref{e:minReeig(H)=<maxReeig(H)<minReeig(E*)}, and recall that $\| \cdot \|$ denotes the Euclidean norm. Then, $X$ is isotropic if and only if the following two conditions hold:
\begin{itemize}
\item [(i)] there exists $\eta > 0$ such that $E_0 = \eta I \in {\mathcal E}^{\textnormal{dom}}_H(X)$;
\item [(ii)] based on the norm $\| \cdot\|_0$ induced by $E_0$ via the relation \eqref{e:E0_norm} for $\|\cdot\|$,
\begin{equation}\label{e:isotropy_at_sphere}
\Delta(d \theta) = \Delta(O d \theta), \quad O \in O(m), \quad S_0 = c^{-1}_0 S^{m-1},
\end{equation}
\end{itemize}
for some $c_0 > 0$, where $\Delta(d \theta)$ is the spherical measure in \eqref{e:int-rep-spectral-F}. Moreover, if $X$ is isotropic, its spectral measure has a density $f_X(x)= \frac{F_X(dx)}{dx}$.
\end{proposition}
\begin{proof} Suppose $X$ is isotropic. By Theorem 2.6 in Didier et al.\ \cite{didier:meerschaert:pipiras:2016:exponents}, there exists an exponent $E_0$ that commutes with $G^{\textnormal{dom}}_1(X)$. Because the domain symmetry group is the full orthogonal group $O(m)$, the exponent has the form $E_0 = \eta I$, $\eta > 0$. This, in turn, yields $\|\cdot\|_{0}$ based on the Euclidean norm via \eqref{e:E0_norm}, i.e., $\|x\|_0 = c_0 \|x\|$ for some $c_0 > 0$. From \eqref{e:S0}, we obtain
\begin{equation}\label{e:l(x)_when_E0=eta*I}
l(x) = \frac{x}{c_0 \norm{x}}.
\end{equation}
Since \eqref{e:Phi_change-of-variables} is a homeomorphism,
\begin{equation}\label{e:tau(x)_when_E0=eta*I}
\tau(x) = (c_0\norm{x})^{1/\eta}.
\end{equation}
Under \eqref{e:l(x)_when_E0=eta*I} and \eqref{e:tau(x)_when_E0=eta*I}, the relation \eqref{e:int-rep-spectral-F} holds with $E = \eta I$ and the induced measure $\Delta(d \theta)$.
Moreover, let $O \in O(m)$. By isotropy and a change of variables $O^*\theta = \theta'$,
$$
{\Bbb E}X(s)X(t)^* = \int^{\infty}_0 \int_{S_0} (e^{i \langle s, r^{\eta I} O^* \theta \rangle}-1)(e^{-i \langle t, r^{\eta I} O^* \theta \rangle}-1)r^{-H} \Delta(d \theta) r^{-H^*}r^{-1}dr
$$
$$
= \int^{\infty}_0 \int_{S_0} (e^{i \langle s, r^{\eta I} \theta ' \rangle}-1)(e^{-i \langle t, r^{\eta I}\theta ' \rangle}-1)r^{-H} \Delta(O d \theta ' ) r^{-H^*}r^{-1}dr.
$$
This gives the equality of measures $r^{-H} \Delta(d \theta) r^{-H^*}r^{-1} = r^{-H} \Delta(O^* d \theta ' ) r^{-H^*}r^{-1}$, $r > 0$. Hence, \eqref{e:isotropy_at_sphere} holds. The converse, i.e., ${\Bbb E}X(Os)X(Ot)^* = {\Bbb E}X(s)X(t)^*$, $s,t \in \bbR^m$, $O \in O(m)$, can be established in the same fashion by means of \eqref{e:int-rep-spectral-F}.

Now note that \eqref{e:isotropy_at_sphere} implies that the measure $\Delta(d \theta)$ is uniform on $c^{-1}_0 S^{m-1}$. In view of the polar representation \eqref{e:change_of_variables_into_polar_coordinates}, this yields the absolute continuity of $F_X(dx)$. $\Box$
\end{proof}
\begin{remark}
It is well known that the covariance function
\begin{equation}\label{e:fBm_cov}
{\Bbb E}X(s)X(t) = \frac{\sigma^2}{2}\{|t|^{2H} + |s|^{2H} - |t-s|^{2H}\}, \quad s, t \in \bbR, \quad 0 < H \leq 1,
\end{equation}
characterizes the univariate FBM. The equivalence between the covariance function and a closed-form formula such as \eqref{e:fBm_cov} breaks down in case of vector processes. Generally speaking, and assuming that $X(0) = 0$ a.s., the stationarity of the increments of OFBF leads to the expression
$$
{\Bbb E}X(t)X(s)^* + {\Bbb E}X(s)X(t)^* = {\Bbb E}X(t)X(t)^* + {\Bbb E}X(s)X(s)^* - {\Bbb E}X(t-s)X(t-s)^*.
$$
If
\begin{equation}\label{e:EX(t)X(s)*=EX(s)X(t)*}
{\Bbb E}X(t)X(s)^* = {\Bbb E}X(s)X(t)^*, \quad s,t \in \bbR^m,
\end{equation}
then operator self-similarity based on exponents $(E,H)$ yields
$$
{\Bbb E}X(s)X(t)^* = \frac{1}{2} \Big\{\tau(t)^{H}{\Bbb E}X(l(t))X(l(t))^* \tau(t)^{H^*}+ \tau(s)^{H}{\Bbb E}X(l(s))X(l(s))^* \tau(s)^{H^*}
$$
\begin{equation}\label{e:EX(s)X(t)*_OFBF_closed_formula_after_change-of-variables}
- \tau(t-s)^{H}{\Bbb E}X(l(t-s))X(l(t-s))^* \tau(t-s)^{H^*} \Big\}, \quad s,t \in \bbR^m
\end{equation}
(cf.\ the relations (4.6) and (4.7) in Bierm\'{e} et al.\ \cite{bierme:meerschaert:scheffler:2007}, p.\ 325). Conversely, starting from \eqref{e:OFBF_harmonizable_representation} and by making use of the fact that ${\Bbb E}X(-t)X(-t)^* = {\Bbb E}X(t)X(t)^*$, $t \in \bbR^m$, we can see that \eqref{e:EX(s)X(t)*_OFBF_closed_formula_after_change-of-variables} implies \eqref{e:EX(t)X(s)*=EX(s)X(t)*}. Under the assumption \eqref{e:minReeig(H)=<maxReeig(H)<minReeig(E*)}, the polar-harmonizable representation \eqref{e:int-rep-spectral-F} can be used to extend this statement. In other words, the relation \eqref{e:EX(t)X(s)*=EX(s)X(t)*}, the existence of the closed form formula \eqref{e:EX(s)X(t)*_OFBF_closed_formula_after_change-of-variables} and the relation $\Delta(d\theta) = \overline{\Delta(d\theta)}$ are all equivalent. Furthermore, by a simple adaptation of the argument in Didier and Pipiras \cite{didier:pipiras:2011}, Proposition 5.1, these relations can in turn be shown to be equivalent to ${\Bbb E}X(-s)X(-t)^* = {\Bbb E}X(s)X(t)^*$, $s, t \in \bbR^m$.
\end{remark}

\subsubsection{On the identifiability of OFBF}\label{s:identifiability}
For an OFBF $X$ with exponents $E$ and $H$, one of or both its exponents may be non-identifiable, i.e., its sets of domain or range exponents may comprise more than one element. As a consequence of Didier et al.\ \cite{didier:meerschaert:pipiras:2016:exponents}, Theorems 2.4 and 2.5, if $X$ satisfies \eqref{e:minReeig(H)=<maxReeig(H)<minReeig(E*)} and \eqref{e:maxeigLambda(Theta)>0=>minReeigLambda(Theta)>0} we can write
\begin{equation}\label{e:didier:meerschaert:pipiras:2016:exponents}
{\mathcal E}^{\textnormal{dom}}_{H}(X) =  E + T(G^{\textnormal{dom}}_1(X)), \quad {\mathcal E}^{\textnormal{ran}}_{E}(X) = H + T(G^{\textnormal{ran}}_1(X)),
\end{equation}
where, for any closed group ${\mathcal G}$ such as $G^{\textnormal{dom}}_1(X)$ or $G^{\textnormal{ran}}_1(X)$, we define its tangent space by
\begin{equation}\label{e:T(G)}
T({\mathcal G}) = \Big\{A \in M(n,\bbR): A = \lim_{n \rightarrow \infty} \frac{G_n - I}{d_n}, \quad
\textnormal{for some } \{G_n\} \subseteq {\mathcal G} \textnormal{ and some }
0 \neq d_n \rightarrow 0 \Big\}.
\end{equation}
The following result is a consequence of \eqref{e:didier:meerschaert:pipiras:2016:exponents} and Theorem \ref{t:OFBF_m,2}, and of the fact that $T(O(2)) = T(SO(2)) = so(2)$, where $so(2)$ is the space of $2 \times 2$ skew-symmetric matrices.
\begin{corollary}\label{c:exponents}
Let $X$ be an OFBF in dimension $(m,n) = (2,2)$ with exponents $E$ and $H$, and satisfying the conditions \eqref{e:minReeig(H)=<maxReeig(H)<minReeig(E*)} and \eqref{e:maxeigLambda(Theta)>0=>minReeigLambda(Theta)>0}. Then, the sets of exponents of $X$ are given by, respectively,
$$
{\mathcal E}^{\textnormal{dom}}_{H}(X) =  \left\{\begin{array}{cc}
E + W_{\textnormal{dom}} so(2) W^{-1}_{\textnormal{dom}}, & G^{\textnormal{dom}}_1(X) \cong O(2);\\
E, & G^{\textnormal{dom}}_1(X) \ncong O(2),
\end{array}\right.
$$
$$
{\mathcal E}^{\textnormal{ran}}_{E}(X) =  \left\{\begin{array}{cc}
H + W_{\textnormal{ran}} so(2) W^{-1}_{\textnormal{ran}}, & G^{\textnormal{ran}}_1(X) \cong SO(2) \textnormal{ or }O(2);\\
H, & G^{\textnormal{ran}}_1(X) \ncong SO(2) \textnormal{ or }O(2),
\end{array}\right.
$$
for a pair of matrices $W_{\textnormal{dom}}, W_{\textnormal{ran}} \in {\mathcal S}_{>0}(2,\bbR)$.
\end{corollary}

\begin{example}\label{ex:sets_of_exponents}
In Example \ref{ex:Gdom1=O(2)_Gran1=O(2)}, since $G^{\textnormal{dom}}_1(X) = O(2)=G^{\textnormal{ran}}_1(X)$, then
$$
{\mathcal E}^{\textnormal{ran}}_I (X) = h I + so(2) , \quad {\mathcal E}^{\textnormal{dom}}_{hI}(X) = I + so(2).
$$
In Example \ref{ex:Gdom1=D3_Gran1=SO(2)},
since $G^{\textnormal{dom}}_1(X) = {\mathcal D}_3$ and $G^{\textnormal{ran}}_1(X) = SO(2)$, then
$$
{\mathcal E}^{\textnormal{ran}}_I (X) = h I + so(2) , \quad {\mathcal E}^{\textnormal{dom}}_{hI}(X) = I.
$$
\end{example}

\appendix

\section{Auxiliary results}

The following lemma is used in the proof of Corollary \ref{c:range_groups_dim2}.
\begin{lemma}\label{l:W1=wW2}
Let $W_1, W_2 \in {\mathcal S}_{> 0}(2,\bbR)$. Also, let $O_1,O_2 \in O(2) \backslash {\mathcal C}_2$. If
\begin{equation}\label{e:W1O1W1inv=W2O2W2inv}
W_1 O_1 W^{-1}_1 = W_2 O_2 W^{-1}_2,
\end{equation}
then for some $w > 0$,
\begin{equation}
W_1 = w W_2 \quad \textnormal{and}\quad O_1 = O_2.
\end{equation}
\end{lemma}
\begin{proof}
We first show that
\begin{equation}\label{e:O_2=AO1A*}
O_2 = AO_1A^*, \quad A \in O(2).
\end{equation}
By \eqref{e:W1O1W1inv=W2O2W2inv},
\begin{equation}\label{e:eig(O1)=eig(O2)}
\textnormal{eig}(O_1) =  \textnormal{eig}(O_2)
\end{equation}
(see \eqref{e:eig(M)}). If $O_1 = O_2$, then \eqref{e:O_2=AO1A*} trivially holds. So, suppose $O_1 \neq O_2$. By \eqref{e:W1O1W1inv=W2O2W2inv} and the uniqueness of the Jordan spectrum, the matrices $O_1$, $O_2$ have the same eigenvalues. If $O_1 \in SO(2) \backslash {\mathcal C}_2$, then we can write $O_1 = U_2 \textnormal{diag}(e^{-i \theta},e^{i \theta})U^*_2$, $\theta \in (0,2\pi)\backslash\{\pi\}$, where $U_2$ is given by \eqref{e:U2}. Note that $O_2 \in SO(2) \backslash{\mathcal C_2}$ by \eqref{e:eig(O1)=eig(O2)} and the fact that $O_2 \in O(2)$. Then, \eqref{e:W1O1W1inv=W2O2W2inv} implies that $O_2 = U_2 \textnormal{diag}(e^{i \theta},e^{-i \theta})U^*_2 = AO_1A^*$, where $A = \textnormal{diag}(1,-1)$. Alternatively, if $O_1 \in O(2) \backslash SO(2)$, then the eigenvalues of $O_1$ are $-1$, 1, with real eigenvectors. By \eqref{e:eig(O1)=eig(O2)}, the same must be true for $O_2$, whence $O_1$ and $O_2 $ only differ by a rotation of their eigenvectors, i.e., there is $O_3 \in SO(2)$ such that $O_2 = O_3 O_1 O^*_3$. By setting $A = O_3$, we establish the relation \eqref{e:O_2=AO1A*} in all cases. Therefore, \eqref{e:W1O1W1inv=W2O2W2inv} and \eqref{e:O_2=AO1A*} imply that
\begin{equation}\label{e:(A*W^(-1)2W1)O1=O1(A*W^(-1)2W1)}
(A^*W^{-1}_2W_1)O_1 = O_1  (A^*W^{-1}_2W_1).
\end{equation}
Assume that $O_1 \in O(2) \backslash {\mathcal C}_2$, i.e., $O_1$ has distinct eigenvalues. By Theorem 2.1 in Didier and Pipiras \cite{didier:pipiras:2012} or Gantmacher \cite{gantmacher:1959}, p.\ 219, there is some (possibly real-valued) unitary matrix $U$ and some $\eta \in \bbC$ such that $A^*W^{-1}_2W_1 = U \textnormal{diag}(\eta,\overline{\eta})U^*$, where the conjugate eigenvalues are a consequence of the fact that $A^*W^{-1}_2W_1$ is real. Thus,
$A^*W^{-1}_2W^2_1W^{-1}_2 A = |\eta|^2 I$, i.e., $W_1 = |\eta|W_2$ since $A \in O(2)$ and $W_1 \in {\mathcal S}_{>0}(2,\bbR)$. By \eqref{e:W1O1W1inv=W2O2W2inv}, this implies that $O_1 = O_2$. $\Box$\\
\end{proof}

The following lemma is used in the proof of Proposition \ref{p:existence_singular_measure}.
\begin{lemma}\label{l:Lambda,Lambda_j_properties}
Let ${\mathcal G}$ be a maximal compact subgroup of $GL(m,\bbR)$, and let $\Lambda_D(dx)$ be the measure \eqref{e:Lambda_D}, $D = \{x_1,\hdots,x_J\}$. Then,
\begin{itemize}
\item [($i$)] two orbits ${\mathcal G}x_{j_1}$, ${\mathcal G}x_{j_2}$ either are disjoint or coincide;
\item [($ii$)] each orbit ${\mathcal G}x_j$, $j = 1,\hdots,J$, is a compact set, and the number of connected components of an orbit ${\mathcal G}x_j$ is no more than the (finite) number of connected components of ${\mathcal G}$;
\item [($iii$)] $\textnormal{supp}\hspace{1mm}\{\Lambda\} = \bigcup^{J}_{j=1}{\mathcal G}x_j$;
\item [($iv$)] $G_0 \in {\mathcal G} \Rightarrow G_0 \in {\mathcal S}^{\textnormal{dom}}(\Lambda)$;
\item [($v$)] the measure $\Lambda(dx)$ assigns different (positive) values to distinct orbits ${\mathcal G}x_{k_1}$ and ${\mathcal G}x_{k_2}$;
\item [($vi$)] $D = \{x_1, \hdots, x_J\} \subseteq \{x_1, \hdots, x_{J'}\} = D' \Rightarrow {\mathcal S}^{\textnormal{dom}}(\Lambda_{D'}) \subseteq {\mathcal S}^{\textnormal{dom}}(\Lambda_{D})$;
\item [($vii$)] for $G \in {\mathcal S}^{\textnormal{dom}}(\Lambda_D)$,
\begin{equation}\label{e:H(mathcal_G)xj=(mathcal_G)xj}
G {\mathcal G}x_{j} = {\mathcal G}x_{j}, \quad j = 1,\hdots, J;
\end{equation}
\item [($viii$)] for a decreasing nested sequence of symmetry groups $\{ {\mathcal S}^{\textnormal{dom}}(\Lambda_{D_k}) \}_{k \in \bbN \cup \{0\}}$, all of which containing ${\mathcal G}$, the equality ${\mathcal G}={\mathcal S}^{\textnormal{dom}}(\Lambda_{D_k})$ holds for some $k$.
\end{itemize}
\end{lemma}
\begin{proof}
To show statement ($i$), suppose $x \in {\mathcal G}x_1 \cap {\mathcal G}x_2$ where we set $j_1 = 1$ and $j_2 = 2$ for notational simplicity. Then, there are $G_1$ and $G_2$ such that $G_1 x_1 = G_2 x_2 = x$. So, let $y \in {\mathcal G} x_2$, i.e., $y = G_{y}x_2$ for some $G_y \in {\mathcal G}$. Then, $G_y G^{-1}_2 G_2 x_2 \in {\mathcal G} x_1$, since $G_2 x_2 \in {\mathcal G}x_1$. This shows that ${\mathcal G}x_1 \supseteq {\mathcal G}x_2$. By the same argument, the converse also holds. In regard to statement ($ii$), the compactness of ${\mathcal G}$ and the continuity of the group action imply that each orbit ${\mathcal G}x_j$, $j = 1,\hdots,J$, is a compact set. Therefore, the number of connected components of the orbit ${\mathcal G}x_j$ is no greater than the number of connected components of the group ${\mathcal G}$ (see Meerschaert and Veeh \cite{meerschaert:veeh:1995}, p.\ 4). To show ($iii$), consider a Borel set $B \subseteq ({\mathcal G}x_1 \cup \hdots \cup {\mathcal G}x_J )^c$, and let $G \in {\mathcal G}$. Since $G$ is bijective and $\emptyset = B \cap {\mathcal G}x_j$, then $\emptyset = GB \cap G{\mathcal G}x_j =   GB \cap {\mathcal G}x_j$, $j = 1,\hdots,J$. Therefore, $GB \cap \{x_j\} = \emptyset$ and $\Lambda(B) = \sum^{J}_{j=1} \int_{\mathcal G} j n_j \hspace{1mm} \delta_{x_j}(GB){\mathbf H}(dG) = 0$. By ($ii$), each orbit ${\mathcal G}x_j$ is a closed set. Then, $\textnormal{supp} \hspace{1mm}\{\Lambda\} \subseteq \cup^{J}_{j=1}{\mathcal G}x_j$. Conversely, let $y \in {\mathcal G}x_{j_0}$ for some $j_0$. Then, $\Lambda(\{y\}) \geq \int_{\mathcal G} j_0 n_{j_0} \hspace{1mm}\delta_{j_0}(G\{y\}){\mathbf H}(dG) > 0$, since there is $G_0 \in {\mathcal G}$ and $y \in \bbR^m$ such that $G_0 y = x_{j_0}$. Thus, $\cup^{J}_{j=1}{\mathcal G}x_j \subseteq \textnormal{supp} \hspace{1mm}\{\Lambda\}$. To show ($iv$), note that
$$
\Lambda(G^{-1}_0 dx) = \sum^{J}_{j=1} \int_{{\mathcal G}} j n_j \hspace{1mm}\delta_{x_j}((GG^{-1}_0) dx){\mathbf H}(d(GG^{-1}_0) G_0)
= \sum^{J}_{j=1} \int_{{\mathcal G}} j n_j \hspace{1mm}\delta_{x_j}(K dx){\mathbf H}(dK G_0)
$$
$$
=\sum^{J}_{j=1} \int_{{\mathcal G}}j n_j \hspace{1mm}\delta_{x_j}(K dx){\mathbf H}(dK ) = \Lambda(dx),
$$
where we made the change of variables $GG^{-1}_0 = K$ and used the right-translation invariance of the Haar measure. In regard to ($v$), first note that the orbits ${\mathcal G}x_1, \hdots, {\mathcal G}x_J$ are distinct, hence disjoint by ($i$). Thus, for $k_1 \neq k_2$,
\begin{equation}\label{e:LambdaD(Gxk)=tauk}
\Lambda({\mathcal G}x_{k_1}) = \sum^{J}_{j=1} \int_{\mathcal G} j n_j \hspace{1mm} \delta_{x_j}(G{\mathcal G}x_{k_1}){\mathbf H}(dG) =
\sum^{J}_{j=1} \int_{\mathcal G} j n_j \hspace{1mm}\delta_{x_j}({\mathcal G}x_{k_1}){\mathbf H}(dG) = k_1  \neq k_2  = \Lambda({\mathcal G}x_{k_2}),
\end{equation}
by \eqref{e:integ_over_orbit=1/nj}. \\

To establish ($vi$), pick any $K \in {\mathcal S}^{\textnormal{dom}} (\Lambda_{D'}) $. Then, for fixed $j = 1,\hdots,J'$, $J' \geq J$, and a Borel set $B \subseteq {\mathcal G}x_{j}$,
$$
\int_{{\mathcal G}}j n_j \delta_{x_j}(GK^{-1}B){\mathbf H}(dG) = \Lambda_{D'}(K^{-1}B) =  \Lambda_{D'}(B) = \int_{{\mathcal G}}j n_j \delta_{x_j}(GB){\mathbf H}(dG).
$$
In particular, this also holds for $j = 1,\hdots,J$, i.e., $K \in {\mathcal S}^{\textnormal{dom}}(\Lambda_{D})$.

In regard to $(vii)$, the argument is very similar to that in Meerschaert and Veeh \cite{meerschaert:veeh:1995}, p.\ 3, but we reproduce it here for the reader's convenience. Let $C_1$ be one of the finitely many connected components of one of the orbits ${\mathcal G}x_1,\hdots, {\mathcal G}x_J$. Since $K \in {\mathcal S}^{\textnormal{dom}}(\Lambda_D)$ is a homeomorphism, then $K C_1$ is also connected. Moreover, by Lemma 1 in Meerschaert and Veeh \cite{meerschaert:veeh:1995} and the fact that $\Lambda_D$ is a finite measure, $K$ maps $\supp \hspace{0.5mm}\{\Lambda_D\}$ onto itself. Therefore, there is some connected component $C_2$ of some orbit such that
\begin{equation}\label{e:HC1_contained_C2}
K C_1 \subseteq C_2.
\end{equation}
Since $K^{-1} \in {\mathcal S}^{\textnormal{dom}}(\Lambda_D)$ is also a homeomorphism, then by the same reasoning there is some connected component $C_3$ of some orbit such that $K^{-1}C_2 \subseteq C_3$. By \eqref{e:HC1_contained_C2}, $C_1 \subseteq C_3$; i.e., $C_1 = C_3 = K^{-1}C_2$. Hence, $KC_1 = C_2$. However, again since $K \in {\mathcal S}^{\textnormal{dom}}(\Lambda_D)$, then $\Lambda_D(C_1) = \Lambda_D(KC_1) = \Lambda_D(C_2)$. Note that each connected component of the orbit ${\mathcal G}x_{j}$ has equal mass, namely, $\frac{j}{\sum^{J}_{k=1}k n_{k}}$, where $n_k$ is the number of components of the orbit ${\mathcal G}x_k$ (see Meerschaert and Veeh \cite{meerschaert:veeh:1995}, p.\ 4). Then, $C_1$ and $C_2$ are connected components of the same orbit. Therefore, \eqref{e:H(mathcal_G)xj=(mathcal_G)xj} holds.

Statement $(viii)$ involves the familiar idea that a compact finite-dimensional Lie group does not have an infinite properly nested sequence of closed subgroups, which in turn is a consequence of the fact that any closed subgroup is compact, and therefore has finite dimension and finitely many connected components. For the reader's convenience, we provide a precise argument in the fashion of Meerschaert and Veeh \cite{meerschaert:veeh:1995}, p.\ 4. 
Suppose ${\mathcal G}$, ${\mathcal H}$ are finite-dimensional compact Lie groups. We claim that, if ${\mathcal G} \subset {\mathcal H}$, then either the dimension of ${\mathcal G}$ is strictly less than the dimension of ${\mathcal H}$, or the number of components of ${\mathcal G}$ is strictly less than the number of components of ${\mathcal H}$.  Let us recall that the
dimension $\textnormal{dim}({\mathcal G})$ of a Lie group ${\mathcal G}$ is the dimension of the corresponding Lie algebra (or tangent space) $T{\mathcal G}$ (When ${\mathcal G}$ is
a set of linear operators, the tangent space is the collection of all operators which can be written as $\lim_{k\to\infty}
(G_k-I)/g_k$ where $\{G_k\}$ is a sequence of operators from ${\mathcal G}$ and $\{g_k\}$ is a sequence of real numbers which converges to
0). Since ${\mathcal G} \subset {\mathcal H}$, then $T{\mathcal G} \subseteq T{\mathcal H}$. If $\textnormal{dim}({\mathcal G})=\textnormal{dim}({\mathcal H})$, then the Lie
algebras are equal.  Since the exponential map sends the Lie algebra onto the connected component of the identity, we see
that in this case the connected component of the identity of the two groups is the same, and we can call it ${\mathcal C}$. Now we have ${\mathcal G}/{\mathcal C} \subset
{\mathcal H}/{\mathcal C}$, where both quotient groups are finite groups. If the inclusion were not proper, we would conclude that ${\mathcal G}={\mathcal H}$. However, if $\textnormal{dim}({\mathcal G})=\textnormal{dim}({\mathcal H})$, then the number of components of ${\mathcal G}$ is smaller than that of ${\mathcal H}$, which establishes the claim. It follows that any decreasing sequence of properly nested compact finite-dimensional Lie groups must eventually terminate.  Hence, ${\mathcal G}={\mathcal S}^{\textnormal{dom}}(\Lambda_{D_{k_0}})$ for some $k_0$, which concludes the proof. $\Box$\\
\end{proof}

The next lemma is used in Section \ref{s:on_G}.
\begin{lemma}\label{l:(m,2)-OFBFs:range_symm_vs_-I}
Consider the class of OFBFs taking values in $\bbR^2$ and satisfying the conditions \eqref{e:minReeig(H)=<maxReeig(H)<minReeig(E*)} and \eqref{e:maxeigLambda(Theta)>0=>minReeigLambda(Theta)>0}. Then, a particular $G^{\textnormal{ran}}_1$ implies a restriction on $G^{\textnormal{dom}}_1$ as described in Table \ref{table:OFBF_m_n=2_symmgroups}, where $G^{\textnormal{dom}}_1$ and $G^{\textnormal{ran}}_1$ are understood as in \eqref{e:Gdom1,Gran1_generic}.
\end{lemma}
\begin{proof}
By the polar representation \eqref{e:change_of_variables_into_polar_coordinates} of the measure $F_X$ and Lemma \ref{l:Gran1_via_specmeas}, $- I \notin G^{\textnormal{dom}}_1$ if and only if there is some set $A(s,\Theta_0)$ as in \eqref{e:A(r,Theta)} such that $F_X(-A(s,\Theta_0)) \neq F_X(A(s,\Theta_0))$, i.e., $\overline{\Delta(\Theta_0)} =\Delta(- \Theta_0) \neq \Delta(\Theta_0)$ for some $\Theta_0 \in {\mathcal U}$ (see \eqref{e:class_U}). In other words, $\Im \Delta(\Theta_0) \neq 0$. Equivalently, by \eqref{e:Gran_structure_n=2},
\begin{equation}\label{e:Lambda(-Theta0)=Lambda(Theta0)}
- I \notin G^{\textnormal{dom}}_1  \Leftrightarrow G_{H,\Theta_0} = W_{\Theta_0}\Big( \bigcap_{r > 0}{\mathcal C}_{O(2)} (\Pi_{r,\Theta_0})\cap SO(2)\Big)W^{-1}_{\Theta_0} \textnormal{ for some $\Theta_0 \in {\mathcal U}$}.
\end{equation}
By the contrapositive, if $- I \notin G^{\textnormal{dom}}_1$, then \eqref{e:Lambda(-Theta0)=Lambda(Theta0)} and \eqref{e:C(O(2))(Pir,Theta)=D2_or_O(2)} imply that $G_{H,\Theta_0} \cong SO(2)$ or ${\mathcal C}_2$ (see also Table \ref{table:intersection_groups_dim=2}). By considering further intersections with the groups $G_{H,\Theta}$, $\Theta \in {\mathcal U}$, Table \ref{table:intersection_groups_dim=2} shows that $G^{\textnormal{ran}}_1$ cannot be conjugate to $O(2)$ or ${\mathcal D}_2$. This establishes statements ($i$) and ($iii$).

Now suppose $G^{\textnormal{ran}}_1 \cong SO(2)$. By Table \ref{table:intersection_groups_dim=2} and expression \eqref{e:Gran_structure_n=2}, there is some $\Theta_0 \in {\mathcal U}$ such that $G_{H,\Theta_0}\cong SO(2)$. By \eqref{e:C(O(2))(Pir,Theta)=D2_or_O(2)} and \eqref{e:Lambda(-Theta0)=Lambda(Theta0)}, $- I \notin G^{\textnormal{dom}}_1$, i.e., statement ($ii$) holds. $\Box$\\
\end{proof}

\begin{remark}
As pointed out in Table \ref{table:OFBF_m_n=2_symmgroups}, $G^{\textnormal{ran}}_1 \cong {\mathcal C}_2$ yields no restriction in the sense it is compatible with either $ -I \in$ or $-I \notin G^{\textnormal{dom}}_1$. This is shown by establishing \eqref{e:Psi}, which implies that ${\mathcal C}_2$ can be matched to any possible domain symmetry group.
\end{remark}

The following lemma is used in the proof of Proposition \ref{p:OFBF_m,n=2,2}.
\begin{lemma}\label{l:(m,n)=(2,2)_construction=>group_contained_O(2)}
Let $F_X(dx) = f_X(x) dx$ be a spectral measure built in the proof of Proposition \ref{p:OFBF_m,n=2,2}, $(ii)$, for a domain symmetry group of the form ${\mathcal C}_{\nu}$ or ${\mathcal D}_{\nu}$, $\nu \in \bbN$, as described in Table \ref{table:OFBF_m_n=2_symmgroups}. Then, ${\mathcal S}^{\textnormal{dom}}(F_X) \subseteq O(2)$.
\end{lemma}
\begin{proof}
By contradiction, suppose that ${\mathcal S}^{\textnormal{dom}}(F_X) = W {\mathcal O} W^{-1}$, where ${\mathcal O} \subseteq O(2)$ and  $W O_1 W^{-1} \notin O(2)$ for some $O_1 \in O(2)$. Then, for some $x_0 \in S^1$, $y_0 := W O_1 W^{-1} x_0 \notin S^1$. By the continuity of the transformation $W O_1 W^{-1}$, we can without loss of generality assume that neither $x_0$ nor $y_0$ is a boundary point between the slices \eqref{e:angular_region_Dv} or \eqref{e:angular_region_Cv}. Then, we can choose a small $\varepsilon_0 > 0$ so that for $\theta \in (-\varepsilon_0, \varepsilon_0) \backslash\{0\}$, the perturbed points $O_{\theta} x_0$ and $WO_1W^{-1}O_{\theta} x_0$, with $O_\theta$ as in \eqref{e:Otheta_Ftheta}, are in the same slices containing $x_0$ and $y_0$, respectively. Since, in addition, $WO_1W^{-1}$ is a domain symmetry of $F_X$, $f_X$ is continuous around the points $x_0$ and $y_0$, and $|\det(WO_1W^{-1})| = 1$, then by Lemma \ref{l:G1_via_specmeas} we have $f_X(y_0) = f_X(x_0)$. In addition, the fact that $WO_1W^{-1} \notin O(2)$ implies that it maps any segment in the sphere to a segment not contained in any sphere around zero; in particular,
$$
f_X(x_0) = f_X(O_\theta x_0), \quad \|W O_1 W^{-1}O_\theta x_0\| \neq \|W O_1 W^{-1}x_0\|, \quad \theta \in (-\varepsilon_0, \varepsilon_0) \backslash\{0\}.
$$
However, since $y_0$ and $W O_1 W^{-1}O_\theta x_0$ lie in the same slice,
$$
\Delta\Big(\frac{W O_1 W^{-1}O_\theta x_0}{\|W O_1 W^{-1}O_\theta x_0\|}\Big) = f_X\Big(\frac{W O_1 W^{-1}O_\theta x_0}{\|W O_1 W^{-1}O_\theta x_0\|} \Big) = f_X\Big(\frac{y_0}{\|y_0\|} \Big) = \Delta\Big(\frac{y_0}{\|y_0\|} \Big).
$$
As a consequence, by \eqref{e:minReeig(H)=<maxReeig(H)<minReeig(E*)},
$$
f_X(O_\theta x_0) = f_X(W O_1 W^{-1}O_\theta x_0) = \|W O_1 W^{-1}O_\theta x_0\|^{-H_E} \Delta\Big( \frac{y_0}{\|y_0\|}\Big)\|W O_1 W^{-1}O_\theta x_0\|^{-H^*_E}
$$
$$
\neq \|y_0\|^{-H_E} \Delta\Big( \frac{y_0}{\|y_0\|}\Big)\|y_0\|^{-H^*_E} = f_X(y_0) = f_X(x_0) = f_X(O_\theta x_0)
$$
(contradiction). $\Box$\\
\end{proof}



In the following lemma, we construct a measure on ${\mathcal B}(S^{m-1})$ that goes into the measure \eqref{e:FX(dx)=x^-H_Xi(dx)_x^-H_x^(-1)}, expressed in polar coordinates.
\begin{lemma}\label{l:Xi(dx)}
Let ${\mathcal G}$ be a maximal compact subgroup of $O(m)$. Let $A:= A_1 + i A_2 \in M(n,\bbC)$ be a matrix such that $A_1A^*_1 \in {\mathcal S}_{>0}(n,\bbR)$, and
\begin{equation}\label{e:condition_on_Im_Delta}
\left\{\begin{array}{ccc}
(a) & A_2A^{*}_1 - A_1 A^{*}_2 \neq {\mathbf 0}, & \textnormal{if } - I \notin {\mathcal G}; \\
(b) & A_2 = {\mathbf 0}, & \textnormal{if } - I \in {\mathcal G}.
\end{array}\right.
\end{equation}
If either
\begin{itemize}
\item [$(i)$] $m = 2$; or
\item [$(ii)$] $m \in \bbN$ and $- I \in {\mathcal G}$,
\end{itemize}
then there is a scalar-valued measure $\Lambda(d \theta)$ on ${\mathcal B}(S^{m-1})$ such that the measure
\begin{equation}\label{e:Xi(dtheta)}
\Xi(d\theta) := AA^* \Lambda(d \theta) + \overline{AA^*}\Lambda(-d \theta)
\end{equation}
is Hermitian, ${\mathcal S}_{\geq 0}(n,\bbC)$-valued,
\begin{equation}\label{e:supp_Xi(dtheta)}
\textnormal{supp}\hspace{0.5mm}\{\Xi \} = \bigcup^{k_{\Xi}}_{j=1}C_j
\end{equation}
for disjoint connected components $C_j$, $j=1,\hdots,k_{\Xi}$, and
\begin{equation}\label{e:Sdom(Lambda)=G}
{\mathcal S}^{\textnormal{dom}}(\Xi) = {\mathcal G}.
\end{equation}
\end{lemma}
\begin{proof}
Suppose condition $(i)$ holds, and further assume that $- I \notin {\mathcal G}$. We are interested in a measure $\Lambda_D(dx)$ as in \eqref{e:Lambda_D}, $D = \{x_1,\hdots,x_J\} \subseteq S^{m-1}$, satisfying \eqref{e:span=Rm} and \eqref{e:S_Lambda(dx)=G_extended_to_Dstar}, where in addition each pivot is such that
\begin{equation}\label{e:-xj_notin_Gxj}
-x_j \notin {\mathcal G}x_j, \quad j = 1,\hdots,J.
\end{equation}
For this purpose, we now argue that we can rewrite the proof of Proposition \ref{p:existence_singular_measure} while replacing the statement \eqref{e:there_is_x_st_Gx_contained_Sdomx} with
\begin{equation}\label{e:technical_condition_on_orbits}
\textnormal{there exists some element $x \in \bbR^2 \backslash\{0\}$ such that $- x \notin {\mathcal G}x \subset {\mathcal S}^{\textnormal{dom}}(\Lambda_{D})x$.}
\end{equation}
In fact, right before \eqref{e:there_is_x_st_Gx_contained_Sdomx} in the proof of Proposition \ref{p:existence_singular_measure}, we already know that $ {\mathcal G} \subseteq {\mathcal S}^{\textnormal{dom}}(\Lambda_D)$ and may now assume that
\begin{equation}\label{e:G_contained_S(LambdaD)}
{\mathcal G} \subset {\mathcal S}^{\textnormal{dom}}(\Lambda_{D}).
\end{equation}
Recall that ${\mathcal G}$ is maximal, and note that the orbits of ${\mathcal S}^{\textnormal{dom}}(\Lambda_{D})$ must coincide with those of a maximal subgroup in its equivalence class $[{\mathcal S}^{\textnormal{dom}}_1(\Lambda_D)]$. Since $m = 2$  and $- I \notin {\mathcal G}$, by Lemma \ref{l:orbits_maximal_groups}, ($iii$), with ${\mathcal G}_1 := {\mathcal G}$ and ${\mathcal G}_2$ set to the maximal element in the class $[{\mathcal S}^{\textnormal{dom}}_1(\Lambda_D)]$, there is $x_0 \in S^1$ such that
$$
- x_0 \in {\mathcal G}x_0 \subset {\mathcal G}_2 x_0 = {\mathcal S}^{\textnormal{dom}}(\Lambda_{D})x_0.
$$
This establishes \eqref{e:technical_condition_on_orbits} with $x = x_0$. By following the rest of the proof of Proposition \ref{p:existence_singular_measure}, we obtain the desired measure $\Lambda_D(dx)$.

Note that (given $x$) \eqref{e:-x_in_Gx} holds if and only if
\begin{equation}\label{e:-y_in_Gx}
\textnormal{for all $y \in {\mathcal G}x$,  $-y \in {\mathcal G}x$}.
\end{equation}
(in fact, assuming \eqref{e:-x_in_Gx} holds, $y \in {\mathcal G}x$ $\Leftrightarrow$ $y \in {\mathcal G}(-x)$ $\Leftrightarrow$ for some $O \in {\mathcal G}$, $y = O(-x)$ $\Leftrightarrow $ for some $O \in {\mathcal G}$, $- y = Ox \in {\mathcal G}x$). By Lemma \ref{l:Lambda,Lambda_j_properties}, $(iii)$, the support of the measure $\Lambda_D(dx)$ is the union of all sets $C$, where
\begin{equation}\label{e:C_is_a_connected_component}
\textnormal{$C$ is the connected component of some orbit ${\mathcal G} x_{1}, \hdots, {\mathcal G} x_{J}$}.
\end{equation}
In addition, in view of \eqref{e:-y_in_Gx}, \eqref{e:-xj_notin_Gxj} implies that there is no $y \in {\mathcal G}x_j$ such that $-y \in {\mathcal G}x_j$. Therefore,
$$
\Lambda_D(-C) = 0.
$$
Now define the measure
\begin{equation}\label{e:def_Xi(d_theta)}
\Xi(d\theta) = AA^* \Lambda_D(d \theta) + \overline{AA^*} \Lambda_D(-d \theta),
\end{equation}
which has the form \eqref{e:Xi(dtheta)} with
\begin{equation}\label{e:Lambda(d_theta)=LambdaD(d_theta)}
\Lambda(d \theta) = \Lambda_D(d\theta).
\end{equation}
Then, for $C$ as in \eqref{e:C_is_a_connected_component},
\begin{equation}\label{e:Ksi(BandC)}
\Xi(-(B \cap C)) = \overline{\Xi(B \cap C)}, \quad \Xi(B \cap C) = AA^* \Lambda_D(B \cap C), \quad \quad B \in {\mathcal B}(S^{1}),
\end{equation}
i.e., the measure $\Xi(d\theta)$ is Hermitian and takes values in ${\mathcal S}_{\geq 0}(2,\bbC)$. Expression \eqref{e:supp_Xi(dtheta)} is a consequence of Lemma \ref{l:Lambda,Lambda_j_properties}, ($iii$), and the fact that ${\mathcal G}$ is a compact group. Moreover, since $\Re(AA^*), \Im(AA^*) \neq {\mathbf 0}$,
$$
G \in {\mathcal S}^{\textnormal{dom}}(\Xi) \Leftrightarrow
\Xi(Gdx)  = \Xi(dx)
$$
$$
\Leftrightarrow AA^* \Lambda_D(G d \theta) + \overline{AA^*} \Lambda_D(- G d \theta)
= AA^* \Lambda_D(d \theta) + \overline{AA^*} \Lambda_D(-  d \theta)
$$
$$
\Leftrightarrow \Re AA^* \{\Lambda_D(G d \theta) + \Lambda_D(- G d \theta)\} + i \Im AA^*\{\Lambda_D(G d \theta) - \Lambda_D(- G d \theta)\}
$$
$$
= \Re AA^* \{\Lambda_D(d \theta) + \Lambda_D(- d \theta)\} + i \Im AA^*\{\Lambda_D(d \theta) - \Lambda_D(- d \theta)\}
$$
$$
\Leftrightarrow \Lambda_D(G d \theta) + \Lambda_D(- G d \theta) = \Lambda_D(d \theta) + \Lambda_D(- d \theta), \quad
\Lambda_D(G d \theta) - \Lambda_D(- G d \theta) = \Lambda_D(d \theta) - \Lambda_D(- d \theta).
$$
$$
\Leftrightarrow \Lambda_D(G d\theta)  = \Lambda_D(d\theta) \Leftrightarrow G \in {\mathcal G},
$$
where the last equivalence is a consequence of Proposition \ref{p:existence_singular_measure}.

The case defined by condition $(ii)$ (for which $-I \in {\mathcal G}$, $m \in \bbN$) can be tackled based on the same formalism but without the modification \eqref{e:technical_condition_on_orbits}. In fact, under the aforementioned condition, let $\Lambda_D(dx)$ be the measure given by Proposition \ref{p:existence_singular_measure} (without any modification). Then, $-I \in {\mathcal G}$ implies that $\Lambda_D(-dx) = \Lambda_D(dx)$. Define the measure $\Xi(d \theta)$ by the same expression \eqref{e:def_Xi(d_theta)}, which in this case reduces to
$$
\Xi(d \theta) = 2\Re AA^* \Lambda_D(d \theta).
$$
It is clear that ${\mathcal G} ={\mathcal S}^{\textnormal{dom}}(\Xi)$ holds, and \eqref{e:supp_Xi(dtheta)} is a consequence of Lemma \ref{l:Lambda,Lambda_j_properties}, ($iii$). $\Box$
\end{proof}

\begin{remark}
Without the constraint \eqref{e:condition_on_Im_Delta} in the construction of the measure $\Xi(dx)$, the statement \eqref{e:Sdom(Lambda)=G} is not generally true. Indeed, if $-I \notin {\mathcal G}$ but we set $A_2 = {\mathbf 0}$, then the measure $\Xi(dx)$ as defined by the expression \eqref{e:Xi(dtheta)} becomes $\Xi(dx) = \Re AA^* (\Lambda_D(dx)+\Lambda_D(-dx))$ and hence satisfies $\Xi(-dx) =  \Xi(dx)$, i.e., $- I \in {\mathcal S}^{\textnormal{dom}}(\Xi)$. Consequently, ${\mathcal S}^{\textnormal{dom}}(\Xi) \neq {\mathcal S}^{\textnormal{dom}}(\Lambda_D)$.
\end{remark}

\begin{remark}\label{r:-I_on_mathcalG}
In regard to Lemma \ref{l:Xi(dx)}, it is natural to ask whether, for general $m$, we can drop the assumption that $- I \in {\mathcal G}$. In order to use the proof of the lemma in its current form, we would need to generalize the step \eqref{e:technical_condition_on_orbits}. However, an explicit description of all maximal subgroups of $O(m)$ is not available in dimension $m \geq 3$, and it is currently a conjecture that the claim \eqref{e:technical_condition_on_orbits} holds in general.
\end{remark}

The following lemma is used in the proof of Theorem \ref{t:OFBF_m,2}.
\begin{lemma}\label{l:S(FX(dx))=G}The second equalities in each expression \eqref{e:Gdom(X)=SdomXi} and \eqref{e:Gran(X)=Gran(BH)} hold, i.e.,
\begin{itemize}
\item [$(i)$] ${\mathcal S}^{\textnormal{dom}}(F_X)= {\mathcal S}^{\textnormal{dom}}(\Xi)$;
\item [$(ii)$] ${\mathcal S}^{\textnormal{ran}}(F_X) = G^{\textnormal{ran}}_1(B_{H})$,
\end{itemize}
respectively.
\end{lemma}
\begin{proof}
We first show $(i)$. Recall that, by Lemma \ref{l:Xi(dx)}, $\Xi(d\theta)$ is constructed so that ${\mathcal S}^{\textnormal{dom}}(\Xi) = {\mathcal G}$ for some maximal compact subgroup ${\mathcal G}$ of $O(m)$. Define the class of sets
$$
{\mathcal A} = \{A(s_1,s_2,\Theta):  s_2 \geq s_1 > 0, \hspace{0.5mm}\Theta \in {\mathcal B}(S^{m-1}) \},
$$
where $A(s_1,s_2,\Theta)= \{r \theta: s_1 < r \leq s_2, \hspace{0.5mm} \theta \in \Theta\}$. Then,
\begin{equation}\label{e:sigma(A)=B(Rm)_A1_and_A2_in_A}
\sigma({\mathcal A}) = {\mathcal B}(\bbR^m), \quad A_1, A_2 \in {\mathcal A} \Rightarrow A_1 \cap A_2 \in {\mathcal A}.
\end{equation}
Define the family of measures $\mu_{O}(A) = \int_{A} (F_X)_O(dx)$, $A \in {\mathcal A}$, $O \in O(m)$. First, assume that $O \in {\mathcal G} = {\mathcal S}^{\textnormal{dom}}(\Xi)$, where the latter equality is a consequence of Lemma \ref{l:Xi(dx)}. Then,
$$
\mu_{O}(A) = \int^{s_2}_{s_1}\int_{\Theta} r^{-H}\Xi(O^* d \theta)r^{-H^*}r^{-1}dr = \int^{s_2}_{s_1}\int_{\Theta} r^{-H}\Xi(d \theta)r^{-H^*}r^{-1}dr
= \mu_{I}(A).
$$
Since this holds for any $A \in {\mathcal A}$, then \eqref{e:sigma(A)=B(Rm)_A1_and_A2_in_A} and an entry-wise application of Theorem 1.1.3 in Meerschaert and Scheffler \cite{meerschaert:scheffler:2001} imply that
\begin{equation}\label{e:muO(B)=muI(B)}
\mu_{O}(B) = \mu_{I}(B), \quad B \in {\mathcal B}(\bbR^m).
\end{equation}
Equivalently, $O \in {\mathcal S}^{\textnormal{dom}}(F_X)$. This establishes that ${\mathcal G} = {\mathcal S}^{\textnormal{dom}}(\Xi) \subseteq {\mathcal S}^{\textnormal{dom}}(F_X)$.

Conversely, for some $O \in O(m)$, assume that \eqref{e:muO(B)=muI(B)} holds. In particular, for $B = A \in {\mathcal A}$, Lebesgue's differentiation theorem implies that
$$
s^{-H} \int_{\Theta} \Xi(O d \theta)s^{-H^*}s^{-1} = s^{-H}\int_{\Theta} \Xi(d \theta)s^{-H^*}s^{-1}, \quad s > 0 \hspace{3mm}\textnormal{a.e.},
$$
i.e.,
$$
\int_{\Theta} \Xi(O d \theta) = \int_{\Theta} \Xi(d \theta), \quad \Theta \in {\mathcal B}(S^{m-1}).
$$
In other words, $O^* \in {\mathcal S}^{\textnormal{dom}}(\Xi) = {\mathcal G}$. Hence, statement $(i)$ holds. \\

We now show $(ii)$. Consider an OFBF $X$ with spectral measure \eqref{e:FX(dx)=x^-H_Xi(dx)_x^-H_x^(-1)}. For $C \in GL(n,\bbR)$, $C \in G^{\textnormal{ran}}_1(X)$ if and only if
\begin{equation}\label{e:condition_range_symm_polar}
C r^{-H}\Xi(\theta) r^{-H^*}r^{-1}dr C^* = r^{-H}\Xi(\theta) r^{-H^*}r^{-1}dr.
\end{equation}
By \eqref{e:supp_Xi(dtheta)}, expression \eqref{e:condition_range_symm_polar} is equivalent to
$$
C \hspace{0.5mm}r^{-H}\Xi(B_0 \cap C_j) r^{-H^*}r^{-1}dr \hspace{0.5mm}C^* = r^{-H}\Xi(B_0 \cap C_j) r^{-H^*}r^{-1}dr, \quad B_0 \in {\mathcal B}(S^{m-1}), \quad j = 1,\hdots,k_{\Xi}.
$$
So, fix $j$ and a set $B_0 \in {\mathcal B}(S^{m-1})$ such that $\Xi(B_0 \cap C_j) \neq {\mathbf 0}$. Define the (discrete) spherical measure
$$
\xi_{B_0}(B) = \left\{\begin{array}{cc}
\Xi(B_0 \cap C_j), & B = \{1\};\\
\Xi(-B_0 \cap C_j), & B = \{-1\}.\\
\end{array}\right.
$$
Then,
$$
r^{-H}\xi_{B_0}\{d\vartheta\} r^{-H^*}r^{-1}dr
$$
is, up to a constant, the spectral measure of the OFBM \eqref{e:OFBM_main_theorem}, since by \eqref{e:Ksi(BandC)}
$$
\Re \xi_{B_0}\{1\} = \Re AA^* \hspace{1mm}\Lambda_{D}(B_0 \cap C_j), \quad \Im \xi_{B_0}\{1\} = \Im AA^* \hspace{1mm}\Lambda_{D}(B_0 \cap C_j),
$$
and
$$
\Re \xi_{B_0}\{-1\} = \Re AA^* \hspace{1mm}\Lambda_{D}(B_0 \cap C_j), \quad \Im \xi_{B_0}\{-1\} = -\Im AA^* \hspace{1mm}\Lambda_{D}(B_0 \cap C_j).
$$
Consequently, for $\Theta = B_0 \cap C_j$, and $\Pi_{r,\Theta}$ and $\Pi_{I,\Theta}$ as in \eqref{e:Pi-rTheta_Pi-I,Theta},
$$
\Pi_{r,\Theta} = r^{-W^{-1}_{\Theta}H W_{\Theta}}r^{-W_{\Theta}\hspace{0.5mm}H\hspace{0.5mm}W^{-1}_{\Theta}} = r^{-(\Re AA^*)^{-1/2} \hspace{0.5mm}H\hspace{0.5mm} (\Re AA^*)^{1/2}}r^{-(\Re AA^*)^{1/2} \hspace{0.5mm}H\hspace{0.5mm} ( \Re AA^*)^{-1/2}},
$$
and
$$
\Pi_{I,\Theta} = W^{-1}_{\Theta} \Im \Delta(\Theta) W^{-1}_{\Theta} = \frac{( \Re AA^*)^{-1/2}\Im AA^*( \Re AA^*)^{-1/2}}{\Lambda_{D}(B_0 \cap C_j)} .
$$
Hence,
$$
{\mathcal C}_{O(n)}(\Pi_{r,\Theta}) = {\mathcal C}_{O(n)}(\Pi_{r}),
$$
and
$$
{\mathcal C}_{O(n)}(\Pi_{I,\Theta}) = {\mathcal C}_{O(n)}(( \Re AA^*)^{-1/2}\Im(AA^*)( \Re AA^*)^{-1/2})= {\mathcal C}_{O(n)}(\Pi_{I}),
$$
where $\Pi_r$ and $\Pi_I$ represent the functions \eqref{e:Pi-rTheta_Pi-I,Theta} for the OFBM \eqref{e:OFBM_main_theorem}, and we can write
$$
G_{H,\Theta} = (\Re AA^*)^{1/2}\Big(\bigcap_{r > 0} {\mathcal C}_{O(n)}(\Pi_{r}) \cap {\mathcal C}_{O(n)}(\Pi_{I})\Big) (\Re AA^*)^{-1/2}.
$$
By Proposition \ref{p:Gran_structure}, $(i)$, statement $(ii)$ is established. $\Box$\\
\end{proof}

The following lemma is used in the proof of Lemma \ref{l:Xi(dx)}. Recall that $\subset$ denotes proper set inclusion.
\begin{lemma}\label{l:orbits_maximal_groups}
Consider the maximal subgroups of $O(2)$ in the sense of Section \ref{s:domain_symm_groups_are_maximal}.
\begin{itemize}
\item [$(i)$] Let ${\mathcal G}$ be a maximal subgroup for which $-I \notin {\mathcal G}$. Then, for a given $x \in S^{1}$,
\begin{equation}\label{e:-x_in_Gx}
- x \in {\mathcal G}x
\end{equation}
if and only if for some $F \in O(2)\backslash SO(2)$, $F \in {\mathcal G}$, the vector $x$ lies at $\pi/2$ angular distance from the reflection axis of $F$;
\item [$(ii)$] for two maximal subgroups ${\mathcal G}_1 \subset {\mathcal G}_2$ such that
\begin{equation}\label{e:x_st_-x_notin_G_1x}
-I \notin {\mathcal G}_1,
\end{equation}
\begin{equation}\label{e:x_st_-x_notin_G_2x}
\textnormal{there is $x_0 \in S^1$ such that $-x_0 \notin {\mathcal G}_1 x_0$ and ${\mathcal G}_1 x_0 \subset {\mathcal G}_2 x_0$}.
\end{equation}
\end{itemize}
\end{lemma}
\begin{proof}
Throughout the proof, without loss of generality we suppose the groups' conjugacies is $W = I$.

Statement $(i)$ is a consequence of the complete description of the maximal compact subgroups in dimension $m = 2$ provided in Table \ref{table:OFBF_m=2_n=2_symmgroups}, middle column. In fact, $- I \notin {\mathcal G}$ implies that ${\mathcal G}$ must be one of the subgroups ${\mathcal C}_{\nu}$, ${\mathcal D}_{\nu}$, $2 \nu + 1$, $\nu \in \bbN \cup \{0\}$. For any such cyclic subgroup ${\mathcal C}_\nu$, statement $(i)$ is trivially true, since it does not include reflections. In addition, the reflections in the subgroup ${\mathcal D}_\nu$, which are finite in number, also correspond to a finite number of reflection axes. Therefore, the number of points for which $-x \in {\mathcal G}x$ is also finite (see Example \ref{ex:D_3_-x_notin_Gx} below), whence statement $(i)$ holds.

We now turn to statement $(ii)$.  First, suppose $- I \in {\mathcal G}_{2}$. For any ${\mathcal G}_{1}$ described in Table \ref{table:OFBF_m=2_n=2_symmgroups}, middle column, statement $(i)$ implies that it is possible to choose $x_0 \in S^1$ such that $-x_0 \notin {\mathcal G}_{1}x_0$. Since $-x_0 \in {\mathcal G}_{2}x_0$, \eqref{e:x_st_-x_notin_G_2x} holds. So, from now on we suppose
\begin{equation}\label{e:-I_notin_G2}
- I \notin {\mathcal G}_{2}.
\end{equation}
By Table \ref{table:OFBF_m=2_n=2_symmgroups}, middle column, it suffices to consider the following cases:
\begin{itemize}
\item [$(ii.a)$] ${\mathcal G}_1 = {\mathcal C}_{\nu_1}$, ${\mathcal G}_2 \in \{{\mathcal C}_{\nu_2},{\mathcal D}_{\nu_2}\}$, $\nu_1 \leq \nu_2$;
\item [$(ii.b)$] ${\mathcal G}_1 = {\mathcal D}_{\nu_1}$, ${\mathcal G}_2 = {\mathcal D}_{\nu_2}$, $\nu_1 < \nu_2$,
\end{itemize}
where $O(2)$ is excluded under \eqref{e:x_st_-x_notin_G_1x} and \eqref{e:-I_notin_G2}. For subcase $(ii.a)$, in light of statement $(i)$ and Table \ref{table:OFBF_m=2_n=2_symmgroups}, middle column, by \eqref{e:-I_notin_G2} we can pick $x_0 \in S^1$ such that $-x_0 \notin {\mathcal G}_2 x_0$ (cf.\ Figure \ref{f:asymm_sphere}). Note that for $O_1,O_2 \in SO(2)$, $O_1 x_0 = O_2 x_0 \Rightarrow O_1 = O_2$ (n.b.: this holds for any $x_0 \neq 0$). When $\nu_1 < \nu_2$, this implies ${\mathcal G}_2$ contains more rotations than ${\mathcal G}_1$. Thus, the orbit ${\mathcal G}_2 x_0$ contains more points than the orbit ${\mathcal G}_1 x_0$, whence \eqref{e:x_st_-x_notin_G_2x} holds. Alternatively, when $\nu_1 = \nu_2=: \nu$ (and ${\mathcal G}_2 = {\mathcal D}_{\nu}$), fix the point
\begin{equation}\label{e:x0=e(i2pi/4nu2)}
x_0 \equiv e^{i \frac{2 \pi}{4 \nu} \frac{1}{2}} \in S^1.
\end{equation}
The 4 in the denominator stems from splitting each slice of angular size $\frac{2 \pi}{\nu}$ in half twice, the first time based on the reflection axes, the second time to split the reflection regions (for ease of visualization, in Figure \ref{f:asymm_sphere}, lower left panel, there are $4 \nu = 12$ slices of angular size $\frac{2 \pi}{12}$. The point $x_0$ splits in half the first slice, where we start counting in the counterclockwise sense at $(1,0)^* \equiv 1$; see also Example \ref{ex:D_3_-x_notin_Gx} below). Then,
\begin{equation}\label{e:-x0_notin_(mathcal G)1_x0}
-x_0 \notin {\mathcal G}x_0
\end{equation}
for ${\mathcal G} = {\mathcal G}_1$. In addition, the orbit ${\mathcal D}_{\nu}x_0$ consists of $\nu$ pairs of points around reflection axes, where the counterclockwise angular distance between two successive pairs of points corresponds to a rotation $O_{3(\frac{2 \pi}{4\nu})}$. Since ${\mathcal C}_{\nu}x_0$ consists of the $\nu$ points obtained by successive rotations $O_{\frac{2 \pi}{\nu}}$ starting at $x_0$, ${\mathcal G}_1 x_0 \subset {\mathcal G}_2 x_0$. This shows \eqref{e:x_st_-x_notin_G_2x} for the subcase $(ii.a)$.

Under \eqref{e:x_st_-x_notin_G_1x} and \eqref{e:-I_notin_G2}, for subcase $(ii.b)$ it suffices to consider
$$
{\mathcal D}_{\nu_1} \subset {\mathcal D}_{\nu_2}, \quad \textnormal{$\nu_1$ and $\nu_2$ are odd}.
$$
In particular, ${\mathcal D}_{\nu_2}$ contains all the rotations in ${\mathcal D}_{\nu_1}$. Therefore, $\nu_2$ must be a multiple of $\nu_1$. Since, in addition, $\nu_1$, $\nu_2$ are odd, then
\begin{equation}\label{e:nu2=znu1}
\nu_2 = z \nu_1, \quad z \in \bbN, \quad z \geq 3.
\end{equation}
Choose again the starting point $x_0$ as in \eqref{e:x0=e(i2pi/4nu2)} with $\nu = \nu_2$. Then, again \eqref{e:-x0_notin_(mathcal G)1_x0} holds with ${\mathcal G} = {\mathcal G}_2$ and $\textnormal{card}({\mathcal D}_{\nu_2}x_0) = 2\nu_2$. Moreover, by \eqref{e:nu2=znu1},
$$
\textnormal{card}({\mathcal D}_{\nu_1}x_0) \leq \textnormal{card}({\mathcal D}_{\nu_1}) = 2 \nu_1 < \nu_2 < \textnormal{card}({\mathcal D}_{\nu_2}x_0).
$$
Therefore, \eqref{e:x_st_-x_notin_G_2x} holds, which establishes $(ii)$. $\Box$\\
\end{proof}

\begin{example}\label{ex:D_3_-x_notin_Gx}
Consider the subgroup ${\mathcal G} = {\mathcal D}_3$. Then, there are only six points $x \in S^1$ for which $- x \in {\mathcal G} x$ (see Table \ref{table:antipode} and Figure \ref{f:asymm_sphere}).
\begin{table}[h]
\centering
\begin{tabular}{ccc}\hline
   reflection & reflection axis (angle) & $- x = F_{\bullet}\hspace{0.5mm} x  \Leftrightarrow x = \hdots$\\\hline
   $F_{\frac{2 \pi}{3}}$ & $\frac{\pi}{3}$  & $\{e^{i\frac{5 \pi}{6}}, e^{i \frac{11\pi}{6}}\}$ \\
   $F_{\frac{4 \pi}{3}}$ & $\frac{2\pi}{3}$  &  $\{e^{i\frac{\pi}{6}}, e^{i \frac{7\pi}{6}}\}$  \\
   $F_{2 \pi}$           & $\pi$          & $\{e^{i \frac{\pi}{2}}, e^{i \frac{3\pi}{2}}\}$\\ \hline
\end{tabular}\caption{dihedral group ${\mathcal D}_3$: points in $S^1$ mapped to their antipodes by a reflection.}\label{table:antipode}
\end{table}
\end{example}

\bibliography{ofbf}

\small

\bigskip

\noindent \begin{tabular}{lll}
Gustavo Didier & Mark M.\ Meerschaert & Vladas Pipiras \\
Mathematics Department &  Dept.\ of Statistics and Probability & Dept.\ of Statistics and Operations Research \\
Tulane University & Michigan State University & UNC at Chapel Hill \\
6823 St.\ Charles Avenue  & 619 Red Cedar Road & CB\#3260, Hanes Hall \\
New Orleans, LA 70118, USA & East Lansing, MI 48824, USA & Chapel Hill, NC 27599, USA \\
{\it gdidier@tulane.edu}& {\it mcubed@stt.msu.edu}   & {\it pipiras@email.unc.edu} \\
\end{tabular}\\

\smallskip

\end{document}